\documentclass{amsart}
\usepackage{amsmath,amssymb,amsthm}

\usepackage[scale=.85]{geometry}

\usepackage{tikz-cd}

\usepackage{enumitem}
\usepackage{etoolbox}

\usepackage[colorlinks=true,
            linkcolor=blue,
            citecolor=red,
            urlcolor=magenta]{hyperref}
\usepackage[capitalise,nameinlink]{cleveref}

\newtheorem{theorem}{Theorem}[section]
\newtheorem{lem}[theorem]{Lemma}
\newtheorem{prop}[theorem]{Proposition}
\newtheorem{cor}[theorem]{Corollary}

\newtheorem{hyp}[theorem]{Hypothesis}

\theoremstyle{remark}
\newtheorem{rem}[theorem]{Remark}

\theoremstyle{definition}
\newtheorem{defn}[theorem]{Definition}

\AtBeginEnvironment{lem}{\crefalias{theorem}{lem}}
\AtBeginEnvironment{prop}{\crefalias{theorem}{prop}}
\AtBeginEnvironment{cor}{\crefalias{theorem}{cor}}
\AtBeginEnvironment{exa}{\crefalias{theorem}{exa}}
\AtBeginEnvironment{claim}{\crefalias{theorem}{claim}}
\AtBeginEnvironment{rem}{\crefalias{theorem}{rem}}
\AtBeginEnvironment{defn}{\crefalias{theorem}{defn}}

\crefname{theorem}{Thm}{Theorems}
\crefname{lem}{Lem}{Lemmas}
\crefname{prop}{Prop}{Propositions}
\crefname{cor}{Cor}{Corollaries}
\crefname{exa}{Exa}{Examples}
\crefname{claim}{Claim}{Claims}
\crefname{rem}{Rem}{Remarks}
\crefname{defn}{Def}{Definitions}
\crefname{section}{Sec.}{Secs.}
\crefname{subsection}{Sec.}{Secs.}
\crefname{equation}{Eq.}{Eqs.}

\newcommand{\F}{\mathbb{F}}
\newcommand{\N}{\mathbb{N}}
\newcommand{\Z}{\mathbb{Z}}
\newcommand{\Q}{\mathbb{Q}}
\newcommand{\R}{\mathbb{R}}
\newcommand{\C}{\mathbb{C}}
\newcommand{\A}{\mathbb{A}}
\newcommand{\bP}{\mathbb{P}}

\newcommand{\Gm}{\mathbb{G}_m}

\newcommand{\Kfam}{\mathcal{K}}
\newcommand{\Qform}{\mathcal{Q}}
\newcommand{\Hur}{\mathcal{H}}
\newcommand{\Gaut}{\mathcal{G}}

\newcommand{\Ypar}[1][r,s,\bd]{\mathcal{Y}_{#1}}
\newcommand{\Yparo}[1][r,s,\bd]{\mathcal{Y}^{\circ}_{#1}}

\newcommand{\bd}{\mathbf{d}}
\newcommand{\bw}{\mathbf{w}}

\DeclareMathOperator{\Jac}{Jac}
\DeclareMathOperator{\wt}{wt}

\DeclareMathOperator{\Aut}{Aut}
\DeclareMathOperator{\Gal}{Gal}
\DeclareMathOperator{\Span}{span}

\DeclareMathOperator{\Spec}{Spec}
\newcommand{\Ni}{\mathrm{Ni}}

\newcommand{\Qbar}{\overline{\Q}}

\newcommand{\Semi}{S_{\infty}}

\newcommand{\detJ}[1]{\det J#1}

\newcommand{\Jacuv}{\Jac_{u,v}}
\newcommand{\jb}[2]{\{#1,#2\}}

\title{Orbits and fields of definition for graded Keller maps}

\author{K. Kistner}
\email{kyle@ulam.capital}

\author{T. Shaska}
\email{shaska@oakland.edu}

\date{\today}

\begin{document}

\begin{abstract}
In \cite{shaska131} were proposed classifying equivariant polynomial maps by the signature of the weight vector, and asked which hyperbolic weights carry graded Keller counterexamples and how the nonempty ones are stratified. We construct two families. Cyclic cones populate $(1,-1,-1)$ at every composite generic fibre degree $N\ge6$, with a six-sheeted member over $\Q(\sqrt{-15})$ and a $33$-sheeted member over $\Q$.
A diagonal family gives $\Kfam(3,(1,-p,-p))\neq\emptyset$ for every $p\ge2$, with $\detJ{F_p}=-p$ and odd generic fibre degree $2p-1$. The geometric monodromy of every member of both families is alternating or symmetric of full degree, so the solvable case 
of \cite{shaska131} occurs only at generic degree three.

At generic degree twelve the normalized quartic cyclic seeds are cut out by a smooth plane cubic $E$ over $\Q$ whose points give pairwise distinct $\Gaut^{\sharp}$-orbits. Two of the associated quotient covers are left-right equivalent exactly when $J=9b^{2}/(8ac)$ agrees, and $J\colon E\to\bP^{1}$ has degree six: it forgets the lift-point marking. Hence a generic quotient class has six graded lifts, $\Kfam(3,(1,-1,-1))$ carries infinitely many orbits, and quotient invariants are not complete for graded classification. An alternating-shear member lies in a separate stratum, with monodromy $A_{12}$ and branch semigroup $\langle6,26,27\rangle$ against $\langle4,11\rangle$; its braid orbit is unique, of size $76  545$, and its arithmetic monodromy over $\Q$ is $S_{12}$ with constants extension $\Q(\sqrt{-3})$.  

We resolve the flagship bounded scheme of the original counterexample scheme-theoretically: $\Yparo[1,2,(4,3,1),K]\cong\Gaut_{K}\times\Spec K[\eta]/(\eta^{2})$ equivariantly for the degree-bound-preserving gauge group $\Gaut=B_{\mathrm{lin}}\times(\Gm)^{3}$, so its reduction is a single orbit and its orbit quotient is the dual point. Finally, the second lift condition is the vanishing of $\int_0^{\epsilon}\varphi(W)^{k}W^{-2}\,dW$ with $\epsilon=(-1)^{k}$. It has no solution over a field with a real embedding when $k$ is even, and at $k=2$ it diagonalizes over $\Q$ as $\langle3,5,\dots,2n-1\rangle$, which forces $\Q(\sqrt{-15})$ at the first level.
\end{abstract}

\maketitle

\setcounter{tocdepth}{1}
\tableofcontents

\section{Introduction}
\label{sec:1}

The Jacobian Conjecture is false for $n\ge3$ \cite{alpoge,tao}. In
\cite{shaska131} the second author organized the graded case by the signature
of the weight vector $\bw$, that is by its sign pattern: if $\bw$ is elliptic
then an equivariant Keller map is an automorphism, and for $n=2$ no signature
admits an equivariant counterexample.  The present paper restricts to the
three-dimensional unit-positive sector: primitive integral weights already
written as
\(
  \bw=(1,-r,-s),\qquad r,s\ge1.
\)
This is not every primitive hyperbolic weight; for example $(2,-1,-1)$ cannot
be integrally rescaled to this form.  Within the stated sector we prove
nonemptiness for every weight and study orbits, invariants, braid classes, and
fields of definition.

Fix a field $K$ of characteristic zero. We write $\Kfam_K(n,\bw)$ for the set of noninjective Keller self-maps of $\A^n_K$ equivariant for $\bw$, and abbreviate $\Kfam(n,\bw)=\Kfam_{\Qbar}(n,\bw)$. We write $\Yparo[r,s,\bd]$ for the parameter scheme of \cite[\S8]{shaska131}, and $\Gaut^{\sharp}$ for the group of liftable automorphisms acting on it. \Cref{sec:2} fixes these conventions, derives the master equation and the identity $\detJ{F}=-\kappa$ relating the Keller constant to the determinant (\cref{cor:master}), and records that the monodromy of a graded map may be computed on its quotient cover (\cref{lem:quotientdegree}).

\Cref{sec:3} constructs the cyclic cone family of weight $(1,-1,-1)$. The construction takes a level $k\ge2$ and a seed $\varphi$ of degree $d$, subject to two lift conditions, and returns a graded Keller map of generic fibre degree $N=kd$. No seed of degree two satisfies the conditions (\cref{cor:dtwo}), so $d\ge3$; with that constraint every composite $N\ge6$ is attained (\cref{thm:spectrum}). The member of least degree produced this way is six-sheeted, defined over $\Q(\sqrt{-15})$, with threefold component degrees $(15,27,3)$ (\cref{thm:six}); there is a $33$-sheeted member over $\Q$ with a rational collision (\cref{thm:thirtythree}). The complementary lower bound for this signature is the differential obstruction of \cite[Cor.\ 1.5]{sh-133}, which forces $\deg A+\deg B+\deg\Lambda\ge14$ for any member. No member descends to a two-variable Keller pair, and \cref{prop:nodescent} isolates the single coefficient responsible.

\Cref{sec:4} constructs the beta--Hermite family. For every pair $2\le r\le s$ it produces a map $F_{r,s}$ over a number field $K_{r}$, equivariant for source weights $(1,-r,-s)$ and target weights $(-s,-r,1)$, with $\detJ{F_{r,s}}=-r$ and generic fibre degree $2r-1$ (\cref{thm:diagonal}); the degrees of $A$ and $B$ are read off the degree of $\rho$ in \cref{prop:diagclass}, and on the diagonal $r=s=p$ the class is $\Yparo[p,p,(p+2,2p^{2}-1,1)]$. Hence $\Kfam(3,(1,-r,-s))\neq\emptyset$ whenever $2\le r\le s$, and in particular for every repeated weight. The underlying one-variable cover is the beta polynomial family (\cref{rem:beta}), so the new content is the graded realisation, not the cover type. The row $r=2$ is rational, and for odd $r$ the field $K_{r}$ is totally imaginary (\cref{prop:fields}).

\Cref{sec:5} separates orbits by invariants of the quotient cover. Left-right equivalence transports the finite normalization, the branch divisor, the inertia, the monodromy, and the branch and local value semigroups; these are intrinsic, whereas a passport or pole partition attached to a chosen pencil is not. The action of $\Gaut^{\sharp}$ descends to left-right equivalence, so intrinsic invariants are orbit invariants (\cref{lem:descendequiv}). We exhibit an alternating shear member of degree twelve whose branch semigroup is $\langle6,26,27\rangle$, while every cyclic member of degree twelve has $11$ in its branch semigroup. Since $11\notin\langle6,26,27\rangle$, the two are inequivalent under arbitrary polynomial left-right equivalence, 
and $\Kfam(3,(1,-1,-1))$ carries at least two $\Gaut^{\sharp}$-orbits at generic degree twelve (\cref{cor:orbits}).

\Cref{sec:6} computes the geometric monodromy of every member of both families, alternating or symmetric of full degree. \Cref{sec:7} shows that at generic degree twelve there are infinitely many orbits: the normalised quartic seeds form a smooth plane cubic $E$ of pairwise distinct $\Gaut^{\sharp}$-orbits, and the forgetful map $J\colon E\to\bP^{1}$ to unpointed quotient classes has degree six (\cref{thm:pointedmain}). 

\Cref{sec:8} resolves the flagship bounded scheme of weight $(1,-1,-2)$: assuming \cref{hyp:flagshipcharts}, its reduction is a single orbit of the bound-preserving gauge group (\cref{thm:flagship}), and puts its quotient cover in the cubic normal form used next. \Cref{sec:unitboundary} propagates that cubic along the unit-weight boundary: for every $s\ge3$ there is a rational member of weight $(1,-1,-s)$ with determinant $-16$, generic fibre degree three and monodromy $S_3$ (\cref{thm:unitboundary}), which completes the three-dimensional unit-positive geography (\cref{cor:hyperbolic}). 

\Cref{sec:9} treats the Nielsen classes, braid orbits, and arithmetic monodromy of the two families, and locates the alternating shear locus as a curve in the four-dimensional reduced Hurwitz space (\cref{prop:sheardim}). \Cref{sec:shearcomp} cuts that curve out exactly; assuming \cref{hyp:stablered}, it also identifies the normalised admissible-cover orbicurve and its rational intersection class (\cref{thm:shearlocus,thm:stablered,thm:shearclass}). 
\Cref{sec:10} determines the fields of definition through the second lift form, which diagonalises over $\Q$ as $\langle3,5,\dots,2n-1\rangle$ (\cref{thm:hilbert}).

\section{Preliminaries}
\label{sec:2}

We keep the conventions of \cite{shaska131}. This section fixes notation and collects the facts used throughout.

\subsection{Gradings and equivariant maps}
\label{subsec:gradings}

Let $K$ be a field of characteristic zero. A weight vector $\bw=(w_1,\dots,w_n)\in\Z^n$ with all $w_i\neq0$ defines an action of $\Gm$ on $\A^n$ by $t\cdot(x_1,\dots,x_n)=(t^{w_1}x_1,\dots,t^{w_n}x_n)$, and a grading of $K[x_1,\dots,x_n]$ with $\wt(x_i)=w_i$. The signature of $\bw$ is the sign pattern of its entries, taken up to permutation and up to $\bw\mapsto-\bw$. We call $\bw$ elliptic if all $w_i$ have the same sign, and hyperbolic otherwise.

A polynomial map $F\colon\A^n\to\A^n$ is equivariant for source weights $\bw$ and target weights $\bw'$ if $F(t\cdot p)=t\ast F(p)$, where $\ast$ is the action of $\bw'$. Equivalently, $F_i$ is homogeneous of weight $w'_i$ for the grading $\bw$. A Keller map is a polynomial map with $\detJ{F}\in K^{\times}$. We write $\Kfam_K(n,\bw)$ for the set of noninjective Keller self-maps of $\A^n_K$ equivariant for $\bw$, and $\Kfam(n,\bw)=\Kfam_{\Qbar}(n,\bw)$. A collision is a pair of distinct points with the same image.

The generic fibre degree of a dominant $F$ is the degree of the field extension $K(x_1,\dots,x_n)/K(F_1,\dots,F_n)$, equivalently the number of preimages of a general point. A Keller map is noninjective if and only if this degree exceeds one.

\subsection{Invariant coordinates and descent}
\label{subsec:descent}

From now on $n=3$ and $\bw=(1,-r,-s)$ with $r,s\ge1$. Write the coordinates as $x,y,z$, so $\wt(x)=1$, $\wt(y)=-r$, $\wt(z)=-s$. Put
\[
  u=x^{r}y, \qquad v=x^{s}z .
\]
These are invariant, and $K[x,y,z]^{\Gm}=K[u,v]$, since an invariant monomial $x^{a}y^{b}z^{c}$ has $a=rb+sc$ and equals $u^{b}v^{c}$.

Let $F$ be equivariant for source weights $\bw$ and target weights $(-s,-r,1)$. Then $x^{s}F_1$, $x^{r}F_2$ and $x^{-1}F_3$ are invariant, so there are unique $A,B,\Lambda\in K[u,v]$ with
\begin{equation}
\label{eq:Fshape}
  F=\Bigl(\frac{A(u,v)}{x^{s}},\ \frac{B(u,v)}{x^{r}},\ x\Lambda(u,v)\Bigr).
\end{equation}
We call $(A,B,\Lambda)$ the descended triple. The pairing is $A$ with $s$ and $B$ with $r$; it is invisible when $r=s$, and matters in the stratum $\Yparo[1,2,(4,3,1)]$.

\begin{lem}
\label{lem:polynomiality}
$F$ in \cref{eq:Fshape} is polynomial if and only if every monomial $u^{i}v^{j}$ occurring in $A$ satisfies $ri+sj\ge s$, and every monomial occurring in $B$ satisfies $ri+sj\ge r$. For $r=s$ both conditions read $A,B\in(u,v)$, that is $A(0,0)=B(0,0)=0$.
\end{lem}

\begin{proof}
$u^{i}v^{j}=x^{ri+sj}y^{i}z^{j}$, so $x^{-s}u^{i}v^{j}$ is polynomial exactly when $ri+sj\ge s$, and likewise for $B$.
\end{proof}

We refer to the conditions of \cref{lem:polynomiality} as the lift conditions, and to $(u,v)=(0,0)$ as the base point.
%
For $f,g\in K[u,v]$ write
\begin{equation}
\label{subsec:master}
  \jb{f}{g}=\Jacuv(f,g)=f_{u}g_{v}-f_{v}g_{u} .
\end{equation}
This bracket is antisymmetric and is a derivation in each argument.
Set
\begin{equation}
\label{eq:PQ}
  Q=\Lambda^{s}A, \qquad P=\Lambda^{r}B .
\end{equation}
These are the invariant coordinates of the target grading, namely $Q=F_1F_3^{\,s}$ and $P=F_2F_3^{\,r}$.

\begin{lem}
\label{lem:bracketidentity}
For $A,B,\Lambda\in K[u,v]$ and integers $r,s\ge1$,
\[
  \Jacuv(P,Q)=\Lambda^{r+s-1}\bigl(\Lambda\jb{B}{A}+rB\jb{\Lambda}{A}+sA\jb{B}{\Lambda}\bigr).
\]
\end{lem}

\begin{proof}
Expanding $\Jacuv(\Lambda^{r}B,\Lambda^{s}A)$ by the derivation property in each argument gives
\[
\begin{split}
  \Jacuv(\Lambda^{r}B,\Lambda^{s}A)
   &=\Lambda^{r}\bigl(\Lambda^{s}\jb{B}{A}+A\jb{B}{\Lambda^{s}}\bigr)
     +B\bigl(\Lambda^{s}\jb{\Lambda^{r}}{A}+A\jb{\Lambda^{r}}{\Lambda^{s}}\bigr)\\
   &=\Lambda^{r+s}\jb{B}{A}+s\Lambda^{r+s-1}A\jb{B}{\Lambda}
     +r\Lambda^{r+s-1}B\jb{\Lambda}{A},
\end{split}
\]
since $\jb{\Lambda^{r}}{\Lambda^{s}}=0$, $\jb{B}{\Lambda^{s}}=s\Lambda^{s-1}\jb{B}{\Lambda}$ and $\jb{\Lambda^{r}}{A}=r\Lambda^{r-1}\jb{\Lambda}{A}$.
\end{proof}

\begin{prop}
\label{prop:detkappa}
Let $F$ be as in \cref{eq:Fshape} with $\Lambda\neq0$. Then
\[
  \detJ{F}=-\frac{\Jacuv(P,Q)}{\Lambda^{r+s-1}} .
\]
\end{prop}

\begin{proof}
The substitution $(x,y,z)\mapsto(u,v,x)$ has Jacobian determinant $x^{r+s}$, and $(F_1,F_2,F_3)\mapsto(Q,P,F_3)$ has Jacobian determinant $F_3^{\,r+s}$. Since $Q,P$ depend only on $u,v$ and $F_3=x\Lambda$,
\[
  \det\frac{\partial(Q,P,F_3)}{\partial(u,v,x)}=\Lambda\,\Jacuv(Q,P)=-\Lambda\,\Jacuv(P,Q).
\]
The chain rule gives $F_3^{\,r+s}\detJ{F}=-\Lambda\,\Jacuv(P,Q)\,x^{r+s}$, and $F_3=x\Lambda$.
\end{proof}

\begin{cor}
\label{cor:master}
$F$ is Keller if and only if there is $\kappa\in K^{\times}$ with
\begin{equation}
\label{eq:master}
  \Lambda\jb{B}{A}+rB\jb{\Lambda}{A}+sA\jb{B}{\Lambda}=\kappa ,
\end{equation}
and then $\detJ{F}=-\kappa$.
\end{cor}


We call \cref{eq:master} the \textbf{master equation} and $\kappa$ the \textbf{Keller constant}. A graded Keller map is thus the same thing as a triple $(A,B,\Lambda)$ satisfying \cref{eq:master} and the lift conditions of \cref{lem:polynomiality}.


\label{rem:clash}
Below we use the affine change $(u,v)\mapsto(\Lambda,v)$ with $\Lambda=1+u$, of Jacobian one, and take all brackets in $(\Lambda,v)$. Then $\jb{\Lambda}{A}=A_{v}$ and $\jb{B}{\Lambda}=-B_{v}$, so \cref{eq:master} reads
\begin{equation}
\label{eq:masterL}
  \Lambda\jb{B}{A}+rBA_{v}-sAB_{v}=\kappa .
\end{equation}
In signature $(1,-1,-1)$ we write $w=v-1$; brackets in $(\Lambda,w)$ agree with brackets in $(\Lambda,v)$. The base point $(u,v)=(0,0)$ becomes $(\Lambda,v)=(1,0)$, and $(\Lambda,w)=(1,-1)$.



For $\bd=(d_1,d_2,d_3)$ let $\Ypar[r,s,\bd]$ be the scheme of triples $(A,B,\Lambda)$ satisfying \cref{eq:master} and the lift conditions, with $\deg A\le d_1$, $\deg B\le d_2$ and $\deg\Lambda\le d_3$, all degrees  in $(u,v)$. Let $\Yparo[r,s,\bd]$ be the open subscheme on which the three degrees are exactly $d_1,d_2,d_3$. Multiprojectivisation removes the three independent coefficient scalings. These are the schemes of \cite{shaska131}.

We write $\Gaut$ for the group of liftable automorphisms of \cite[\S8.5]{shaska131} and $\Gaut^{\sharp}$ for the subgroup preserving the graded normal form. Two graded Keller maps are \textbf{equivalent} when they lie in one $\Gaut^{\sharp}$-orbit, and classification means the determination of these orbits.

\subsection{The quotient cover}
\label{subsec:quotient}

The pair $(P,Q)$ of \cref{eq:PQ} is a polynomial self-map of $\A^2$ with $\Jacuv(P,Q)=\kappa\Lambda^{r+s-1}$. We call it the \textbf{quotient map of $F$}, and its normalization the quotient cover.

\begin{lem}
\label{lem:quotientdegree}
Let $L=x\Lambda$. Then $K(x,y,z)=K(u,v,L)$ and $K(F_1,F_2,F_3)=K(P,Q,L)$, and $L$ is transcendental over $K(u,v)$. Hence the generic fibre degree of $F$ equals $[K(u,v):K(P,Q)]$.
\end{lem}

\begin{proof}
$x=L/\Lambda$ recovers $x$ from $u,v,L$, and $y,z$ follow from $u=x^{r}y$, $v=x^{s}z$. On the target, $F_3=L$, $F_1=Q/L^{s}$ and $F_2=P/L^{r}$. The extension $K(u,v,L)/K(P,Q,L)$ is the purely transcendental base change of $K(u,v)/K(P,Q)$ along $K(P,Q,L)/K(P,Q)$, so the degree is unchanged.
\end{proof}

We write $N$ for this degree and call it the generic fibre degree of the member. By \cref{lem:quotientdegree} the normal closure group of $F$ over its target function field equals that of the quotient cover, so all monodromy statements may be made downstairs.

Fixing the second target coordinate at a general value $Y$ cuts the quotient cover down to a degree-$N$ cover of curves. The passport of that cover is the list of its inertia cycle types, and the pole partition is the cycle type over the point at infinity. Both depend on the choice of second coordinate.

\subsection{Branch curves and semigroups}
\label{subsec:semigroup}

Let $f=(P,Q)$ be the quotient map, dominant of degree $N$, and let $\Gamma_f\subset\A^2$ be the union of the one-dimensional components of the locus over which the fibre has fewer than $N$ points, taken with its reduced structure. We call $\Gamma_f$ the \textbf{branch curve} of $f$. It is not the image of the \textbf{critical locus} $\{\Jacuv(P,Q)=0\}$: by \cref{cor:master} that locus is $\{\Lambda=0\}$, and $Q=\Lambda^{s}A$ and $P=\Lambda^{r}B$ both vanish there, so it is contracted to the origin and every component of $\Gamma_f$ records a failure of properness rather than a critical value.

Let $\Gamma$ be an irreducible component of $\Gamma_f$, with affine coordinate ring $\C[\Gamma]$ and normalisation $\widetilde\Gamma$. A \textbf{place at infinity} of $\Gamma$ is a place of $\widetilde\Gamma$ at which some element of $\C[\Gamma]$ has a pole. Suppose $\Gamma$ has exactly one, as does every branch curve occurring below, and write $\nu$ for the valuation there. The \textbf{semigroup at infinity} of $\Gamma$ is
\[
  S_{\infty}(\Gamma)=\{-\nu(g)\ :\ g\in\C[\Gamma]\setminus\{0\}\}\subset\Z_{\ge0} .
\]

All the branch curves below are given by a birational parametrisation $t\mapsto\bigl(g_1(t),\dots,g_m(t)\bigr)$ with $g_i\in\C[t]$. Then $\widetilde\Gamma=\A^1_t$, the place at infinity is unique, $\C[\Gamma]=\C[g_1,\dots,g_m]$ and $\nu=-\deg$, so $S_{\infty}(\Gamma)$ is the set of degrees of the nonzero elements of $\C[g_1,\dots,g_m]$. Since degrees add under multiplication, this set contains the numerical semigroup $\langle\deg g_1,\dots,\deg g_m\rangle$. The containment can be strict: two monomials in the $g_i$ of equal degree may have cancelling leading terms, and the resulting difference has a degree that need not lie in the semigroup generated by the $\deg g_i$. This occurs, for instance, when $\deg g_1=6$ and $\deg g_2=27$: the relation $2\cdot27=9\cdot6$ forces a cancellation producing degree $26\notin\langle6,27\rangle$. The following criterion is what rules out any further such degree.

\begin{lem}
\label{lem:degsemigroup}
Let $g_1,\dots,g_m\in\C[t]$ be monic of degrees $d_1,\dots,d_m$, and let
\[
 I_{\bd}=\ker\bigl(\C[x_1,\dots,x_m]\longrightarrow\C[t],\quad
 x_i\longmapsto t^{d_i}\bigr)
\]
be the toric ideal of the numerical semigroup $\langle d_1,\dots,d_m\rangle$.
For $u\in\Z_{\ge0}^m$ write $g^u=g_1^{u_1}\cdots g_m^{u_m}$.
Suppose $I_{\bd}$ is generated by homogeneous binomials
$x^{u_j}-x^{v_j}$, $1\le j\le q$, and that for each $j$ the corresponding
difference $g^{u_j}-g^{v_j}$ can be written as a polynomial in the $g_i$ all
of whose monomials have weighted degree strictly smaller than
$\sum_i(u_j)_id_i=\sum_i(v_j)_id_i$.  Then the degrees of the nonzero elements
of $\C[g_1,\dots,g_m]$ are exactly $\langle d_1,\dots,d_m\rangle$.
\end{lem}

\begin{proof}
One containment is immediate.  For the other, write an element of the
subalgebra as a polynomial in the $g_i$ and choose an expression for which the
largest weighted degree $D$ is minimal.  If the actual degree is smaller than
$D$, the sum of the terms of weight $D$ is a homogeneous relation among
$t^{d_1},\dots,t^{d_m}$, hence lies in the degree-$D$ part of the toric ideal
$I_{\bd}$.  Express that relation using the stated homogeneous binomial
generators.  Replacing each generator $g^{u_j}-g^{v_j}$ by its lower-weight
expression lowers the largest weight of the chosen representation, a
contradiction.  Therefore the actual degree equals $D$, which belongs to the
numerical semigroup.
\end{proof}

\begin{rem}
\label{rem:rs}
\Cref{lem:degsemigroup} is the subalgebra form of a leading-term argument. For an ideal, a Gr\"obner basis is a generating set whose leading terms generate the ideal of all leading terms, and Buchberger's criterion tests this on the syzygies between leading terms. Here the object controlled is the subalgebra $\C[g_1,\dots,g_m]$ rather than an ideal, the role of the ideal of leading terms is played by the algebra generated by $t^{d_1},\dots,t^{d_m}$, and the relations to be tested are the additive ones among the $d_i$, that is the toric ideal of the numerical semigroup. The general theory is developed in \cite{rs}; for the two computations below only the criterion above is needed, and in both the toric ideal has a single or a pair of explicit binomial generators.
\end{rem}

\begin{lem}
\label{lem:semiintrinsic}
Let $f_2=\psi\circ f_1\circ\phi^{-1}$ with $\phi,\psi$ polynomial automorphisms of $\A^2$. Then $\psi$ carries $\Gamma_{f_1}$ isomorphically onto $\Gamma_{f_2}$, and the multiset of semigroups at infinity of the components is the same for both. In particular $S_{\infty}$ is an invariant of the left--right equivalence class.
\end{lem}

\begin{proof}
Since $\phi$ and $\psi$ are bijective, $\#f_2^{-1}(\psi(y))=\#f_1^{-1}(y)$ for every $y$, so $\psi(\Gamma_{f_1})=\Gamma_{f_2}$ and $\psi$ restricts to an isomorphism of reduced affine curves. Such an isomorphism carries coordinate rings to coordinate rings, hence normalisations to normalisations and places at infinity to places at infinity, and preserves valuations.
\end{proof}

A \textbf{passport} or a \textbf{pole partition}, by contrast, is attached to a chosen second coordinate and is invariant only under equivalences preserving that pencil. This is the distinction between intrinsic and framed data.

\section{Cyclic cones in signature $(1,-1,-1)$}
\label{sec:3}

\subsection{The canonical chain and the quotient pair}

Fix $k\ge2$ and set
\begin{equation}
\label{eq:TW}
  T=\Lambda w, \qquad W=\Lambda^{k-1}w^{k}, \qquad T^{k}=\Lambda W .
\end{equation}

\begin{lem}
\label{lem:brackets}
$\jb{\Lambda}{T}=\Lambda$, $\jb{\Lambda}{W}=kT^{k-1}$, and $\jb{T}{W}=W$.
\end{lem}

\begin{proof}
$T_{\Lambda}=w$, $T_{w}=\Lambda$, $W_{\Lambda}=(k-1)\Lambda^{k-2}w^{k}$ and $W_{w}=k\Lambda^{k-1}w^{k-1}=kT^{k-1}$. The first two are immediate, and
\[
  \jb{T}{W}=w\cdot k\Lambda^{k-1}w^{k-1}-\Lambda\cdot(k-1)\Lambda^{k-2}w^{k}=\Lambda^{k-1}w^{k}=W. \qedhere
\]
\end{proof}

Choose a seed $\varphi\in W^{2}K[W]$ of degree $d\ge3$ and define $A_k,\dots,A_0$ by
\begin{equation}
\label{eq:jets}
  A_k=\frac1W, \qquad A_j'=(j+1)\varphi'(W)A_{j+1} \quad (0\le j\le k-1),
\end{equation}
all integration constants zero. Put
\begin{equation}
\label{eq:PQjet}
  Q=T+\varphi(W), \qquad P=\Lambda+A_{k-1}T^{k-1}+\dots+A_1T+A_0 .
\end{equation}
The restriction $d\ge3$ is justified below.

\begin{prop}
\label{prop:jet}
$A_j\in WK[W]$ for $0\le j\le k-1$, so $P,Q\in\Lambda K[\Lambda,w]$, and $\Jacuv(P,Q)=\Lambda$. Writing $Q=\Lambda A$ and $P=\Lambda B$, the pair $(A,B)$ satisfies \cref{eq:masterL} with $r=s=\kappa=1$.
\end{prop}

\begin{proof}
Since $\varphi\in W^{2}K[W]$ we have $\varphi'\in WK[W]$, so $A_{k-1}'=k\varphi'/W\in K[W]$ and $A_{k-1}\in WK[W]$; downward induction through \cref{eq:jets} gives $A_j\in WK[W]$ for all $j\le k-1$. As $W=\Lambda^{k-1}w^{k}$ with $k\ge2$, each $A_j$ is divisible by $\Lambda$, and $\varphi$ is divisible by $W^{2}$, so $P,Q\in\Lambda K[\Lambda,w]$.

By \cref{lem:brackets} and the derivation property,
\[
  \jb{P}{Q}=\Lambda+k\varphi'T^{k-1}-W\sum_{j=0}^{k-1}A_j'T^{j}+\varphi'W\sum_{j=0}^{k-2}(j+1)A_{j+1}T^{j}.
\]
The coefficient of $T^{j}$ for $j\le k-2$ is $-W\bigl(A_j'-(j+1)\varphi'A_{j+1}\bigr)=0$, and that of $T^{k-1}$ is $k\varphi'-WA_{k-1}'=0$, both by \cref{eq:jets}. Hence $\Jacuv(P,Q)=\Lambda=\Lambda^{r+s-1}$ with $r=s=1$, and \cref{cor:master} gives $\kappa=1$.
\end{proof}

\begin{rem}
\label{rem:osculating}
The chain \cref{eq:jets} has a geometric reading. Fix $W$ with $\varphi'(W)\neq0$ and let $h=h(T)$ be the formal solution of $\varphi(W+h)=\varphi(W)+T$ with $h(0)=0$. Then $P$ of \cref{eq:PQjet} agrees with $A_0(W+h)$ to order $k$ in $T$, and
\[
  P-A_0(W+h)=\frac{T^{k+1}}{(k+1)W^{2}\varphi'(W)}+O(T^{k+2}).
\]
Thus $P$ is the $k$-th osculating jet along the plane curve $W\mapsto\bigl(A_0(W),\varphi(W)\bigr)$, and $k$ is the contact order.
\end{rem}

\subsection{The fibre polynomial}

Since $\Lambda=T^{k}/W=A_kT^{k}$, the second formula of \cref{eq:PQjet} reads $P=\sum_{j=0}^{k}A_jT^{j}$. Eliminating $T=Q-\varphi(W)$ gives $P=\Pi(W,Q)$, where
\begin{equation}
\label{eq:fibrepoly}
  \Pi(W,Y)=\sum_{j=0}^{k}A_j(W)\bigl(Y-\varphi(W)\bigr)^{j}, \qquad A_k=\frac1W .
\end{equation}
We call $\Pi$ the fibre polynomial and $W\bigl(\Pi(W,Y)-P\bigr)$, a polynomial in $W$, the marked equation. The pencil on the quotient is by $Q$; we write $Y$ for its value, so that $K(W,Q)=K(T,W)=K(\Lambda,w)$ and the cover at fixed $Y$ is $W\mapsto\Pi(W,Y)$.

\begin{lem}
\label{lem:Pider}
$\dfrac{\partial\Pi}{\partial W}=-\dfrac{\bigl(Y-\varphi(W)\bigr)^{k}}{W^{2}}$.
\end{lem}

\begin{proof}
Differentiating \cref{eq:fibrepoly},
\[
\begin{split}
  \frac{\partial\Pi}{\partial W}
   &=\sum_{j=0}^{k-1}A_j'(Y-\varphi)^{j}+A_k'(Y-\varphi)^{k}
     -\varphi'\sum_{j=1}^{k}jA_j(Y-\varphi)^{j-1}\\
   &=\varphi'\sum_{j=0}^{k-1}(j+1)A_{j+1}(Y-\varphi)^{j}
     +A_k'(Y-\varphi)^{k}-\varphi'\sum_{i=1}^{k}iA_i(Y-\varphi)^{i-1},
\end{split}
\]
by \cref{eq:jets}. The two sums cancel after the shift $i=j+1$, and $A_k'=-1/W^{2}$.
\end{proof}

\subsection{The lift conditions at the base point}

By \cref{rem:clash} the base point is $(\Lambda,w)=(1,-1)$, where $T=-1$ and $W=\epsilon$ with
\[
  \epsilon=(-1)^{k}, \qquad \epsilon^{2}=1 .
\]

\begin{prop}
\label{prop:lift}
The conditions $A(1,-1)=B(1,-1)=0$ are
\begin{equation}
\label{eq:liftjet}
  \varphi(\epsilon)=1, \qquad \Pi(\epsilon,0)=0 .
\end{equation}
Equivalently, the second condition reads $1+\sum_{j=0}^{k-1}(-1)^{j}A_j(\epsilon)=0$.
\end{prop}

\begin{proof}
At the base point $\Lambda=1$, so $A=Q$ and $B=P$ there. Now $Q=T+\varphi(W)=-1+\varphi(\epsilon)$, which vanishes exactly when $\varphi(\epsilon)=1$. Given that, the base point lies on the fibre $Q=0$, on which $P=\Pi(W,0)$; so $B=0$ reads $\Pi(\epsilon,0)=0$. For the last form, $\varphi(\epsilon)=1$ gives $\Pi(\epsilon,0)=\sum_{j=0}^{k}(-1)^{j}A_j(\epsilon)$, and $(-1)^{k}A_k(\epsilon)=(-1)^{k}\epsilon=1$.
\end{proof}

Thus a member of the family is a seed whose associated fibre polynomial vanishes at $W=\epsilon$ on the fibre $Q=0$; the lift point is the marked root.

\subsection{Degrees and the exclusion of $d=2$}

\begin{prop}
\label{prop:degree}
Let $\varphi$ have degree $d\ge2$ with top coefficient $c_d$, and let $c_2$ be its coefficient in degree two. Then, for generic $Y$,
\[
\begin{split}
  \deg_{W}\Pi(W,Y)&=kd-1 \quad\text{with leading coefficient } \frac{(-1)^{k+1}c_d^{\,k}}{kd-1},\\
  \operatorname{ord}_{W=0}\Pi(W,0)&=2k-1 \quad\text{with that coefficient } \frac{(-1)^{k+1}c_2^{\,k}}{2k-1} .
\end{split}
\]
The marked equation has degree $kd$ in $W$ with the same leading coefficient, and the generic fibre degree of the member is $kd$.
\end{prop}

\begin{proof}
Write $\deg A_{k-i}=id-1$ for $1\le i\le k$, which follows from \cref{eq:jets} by induction, and let $\alpha_{k-i}$ be the leading coefficient of $A_{k-i}$. Comparing leading terms in $A_{k-i}'=(k-i+1)\varphi'A_{k-i+1}$ gives
\[
  \alpha_{k-i}=\frac{(k-i+1)\,dc_d}{id-1}\,\alpha_{k-i+1}, \qquad \alpha_k=1 ,
\]
the last as the coefficient of $W^{-1}$ in $A_k$. Hence $\alpha_{k-i}=\dfrac{k!}{(k-i)!}\dfrac{(dc_d)^{i}}{\prod_{m=1}^{i}(md-1)}$.

Each term of \cref{eq:fibrepoly} has degree $(k-j)d-1+jd=kd-1$, and $(Y-\varphi)^{j}$ has leading coefficient $(-c_d)^{j}$. Summing over $j=k-i$,
\[
  [W^{kd-1}]\,\Pi=(-c_d)^{k}\sum_{i=0}^{k}\frac{k!}{(k-i)!}\frac{(-d)^{i}}{\prod_{m=1}^{i}(md-1)} .
\]
Put $\beta=-1/d$, so $md-1=d(m+\beta)$ and the summand is $(-k)_i/(1+\beta)_i$. Chu--Vandermonde evaluates the sum as $(\beta)_k/(1+\beta)_k=\beta/(\beta+k)=-1/(kd-1)$, which gives the stated leading coefficient. The second display is the same computation with $d$ replaced by $2$, since $\operatorname{ord}_{W=0}A_{k-i}=2i-1$ and the recursion for lowest coefficients is that of the first display with $c_d,d$ replaced by $c_2,2$.

The marked equation is $W\Pi-PW$, of degree $kd$. For $Y\neq0$ the function $\Pi(\cdot,Y)$ has a simple pole at $W=0$, since $A_k(Y-\varphi)^{k}=(Y-\varphi)^{k}/W$ and $\varphi(0)=0$, and a pole of order $kd-1$ at infinity. It has no other poles, so it has degree $kd$ as a map $\bP^{1}\to\bP^{1}$. By \cref{lem:quotientdegree} this is the generic fibre degree.
\end{proof}

\begin{cor}
\label{cor:dtwo}
No seed of degree two satisfies \cref{eq:liftjet}.
\end{cor}

\begin{proof}
For $d=2$ the first condition of \cref{eq:liftjet} gives $c_2\epsilon^{2}=c_2=1$, so $\varphi=W^{2}$. By \cref{prop:degree} the polynomial $\Pi(W,0)$ then has degree $2k-1$ and order $2k-1$ at the origin, hence is the monomial $(-1)^{k+1}W^{2k-1}/(2k-1)$, which does not vanish at $W=\epsilon$.
\end{proof}

\begin{prop}
\label{prop:conebox}
Let $k\ge2$, let $\varphi$ be a seed of degree $d\ge3$, and let $(A,B,\Lambda)$ be the associated triple of \cref{prop:jet}. Then, all degrees being taken in $(u,v)$,
\[
  \deg A=(2k-1)d-1, \qquad \deg B=(2k-1)(kd-1)-1, \qquad \deg\Lambda=1 .
\]
If in addition $\varphi$ satisfies \cref{eq:liftjet}, the resulting member has class in $\Yparo[1,1,((2k-1)d-1,\ (2k-1)(kd-1)-1,\ 1)]$.
\end{prop}

\begin{proof}
The change $(u,v)\mapsto(\Lambda,w)$ is affine, so degrees agree in both. By \cref{eq:TW}, $\deg T=2$ and $\deg W=2k-1$.

In $Q=T+\varphi(W)$ the top term is $c_d\Lambda^{(k-1)d}w^{kd}$, of degree $(2k-1)d\ge9>2=\deg T$, so $\deg Q=(2k-1)d$ and $\deg A=(2k-1)d-1$. As in the proof of \cref{prop:degree}, $\deg_{W}A_{k-i}=id-1$ with nonzero leading coefficient, so the term $A_jT^{j}$ of \cref{eq:PQjet} has degree
\[
  \bigl((k-j)d-1\bigr)(2k-1)+2j=(2k-1)(kd-1)-j\bigl((2k-1)d-2\bigr),
\]
strictly decreasing in $j$. The $k$ degrees are distinct, the largest is $(2k-1)(kd-1)$ at $j=0$, and it exceeds $\deg\Lambda=1$. Hence $\deg B=(2k-1)(kd-1)-1$.
\end{proof}


\subsection{The lift form}
\label{subsec:liftform}

\begin{prop}
\label{prop:liftclosed}
Let $\varphi\in W^{2}K[W]$ and let $\Pi$ be the fibre polynomial \cref{eq:fibrepoly}. Then
\[
  \Pi(W,0)=(-1)^{k+1}\int_0^{W}\frac{\varphi(s)^{k}}{s^{2}}\,ds ,
\]
the integrand being a polynomial. Hence the second condition of \cref{eq:liftjet} is the vanishing of
\begin{equation}
\label{eq:liftform}
  \Qform^{(k)}(\varphi)=\int_0^{\epsilon}\frac{\varphi(W)^{k}}{W^{2}}\,dW ,
  \qquad \epsilon=(-1)^{k}.
\end{equation}
\end{prop}

\begin{proof}
Since $\varphi\in W^{2}K[W]$ we have $\varphi^{k}\in W^{2k}K[W]$, so $\varphi^{k}/W^{2}$ lies in $W^{2k-2}K[W]$ and is a polynomial. Setting $Y=0$ in \cref{lem:Pider} gives
\[
  \frac{d}{dW}\Pi(W,0)=-\frac{(-\varphi(W))^{k}}{W^{2}}=(-1)^{k+1}\frac{\varphi(W)^{k}}{W^{2}} .
\]
Both sides of the asserted identity therefore have the same derivative, and both vanish at $W=0$: the right-hand side by construction, the left because $\operatorname{ord}_{W=0}\Pi(W,0)=2k-1\ge3$ by \cref{prop:degree}. The second condition of \cref{eq:liftjet} is $\Pi(\epsilon,0)=0$.
\end{proof}

We call $\Qform^{(k)}$ the lift form at level $k$. It is a form of degree $k$ in the coefficients of $\varphi$, with coefficients in $\Q$.


\subsection{Rescaling and the degree spectrum}

\begin{lem}
\label{lem:rescale}
Let $\varphi$ be a seed of degree $d$ and let $t,\beta\in K^{\times}$. Write $\widetilde\varphi(W)=t\varphi(\beta W)$ and let $\widetilde A_j$, $\widetilde\Pi$ be the associated data. Then
\[
  \widetilde A_{k-i}(W)=\beta t^{i}A_{k-i}(\beta W) \quad (1\le i\le k),
  \qquad
  \widetilde\Pi(W,0)=t^{k}\beta\,\Pi(\beta W,0) .
\]
\end{lem}

\begin{proof}
For $i=1$, $\widetilde A_{k-1}'(W)=k\widetilde\varphi'(W)/W=k\beta t\varphi'(\beta W)/W$, while 
\[
\frac{d}{dW}\bigl[\beta tA_{k-1}(\beta W)\bigr]=\beta^{2}tA_{k-1}'(\beta W)=k\beta t\varphi'(\beta W)/W.
\]
 Both vanish at $W=0$, so they agree. If the formula holds for $i$, then
\[
\begin{split}  \widetilde A_{k-i-1}'(W)	&	=(k-i)\widetilde\varphi'\widetilde A_{k-i}   =\beta^{2}t^{i+1}(k-i)\varphi'(\beta W)A_{k-i}(\beta W) \\
 							&  =\frac{d}{dW}\bigl[\beta t^{i+1}A_{k-i-1}(\beta W)\bigr],
\end{split}
\]
which gives the case $i+1$. Substituting into \cref{eq:fibrepoly} at $Y=0$, the term with $j=k-i$ contributes 
\[
(-1)^{k-i}\beta t^{k}A_{k-i}(\beta W)\varphi(\beta W)^{k-i}
\]
 for $i\ge1$, and the term with $j=k$ contributes 
 \[
 (-1)^{k}t^{k}\varphi(\beta W)^{k}/W=(-1)^{k}t^{k}\beta\varphi(\beta W)^{k}/(\beta W).
 \]
  Summing gives $t^{k}\beta\,\Pi(\beta W,0)$.
\end{proof}

\begin{theorem}
\label{thm:spectrum}
For every $k\ge2$ and every $d\ge3$ there is a member of
$\Kfam(3,(1,-1,-1))$ with parameters $(k,d)$ and generic fibre degree $kd$.
Consequently every composite integer $N\ge6$ occurs as a generic fibre degree.
\end{theorem}

\begin{proof}
Put $n=d-2$ and begin with the explicit rational seed
$\varphi_0=W^2+W^d$.  At $Y=0$, termwise integration of
\cref{lem:Pider} gives
\[
 \Pi_0(W,0)=(-1)^{k+1}W^{2k-1}H_{k,d}(W^n),\qquad
 H_{k,d}(x)=\sum_{j=0}^k\binom{k}{j}\frac{x^j}{2k-1+jn}.
\]
Coefficientwise differentiation yields
\[
 nxH_{k,d}'(x)+(2k-1)H_{k,d}(x)=(1+x)^k.
\]
A common root of $H_{k,d}$ and $H_{k,d}'$ would therefore be $-1$, whereas
\[
 H_{k,d}(-1)=\int_0^1t^{2k-2}(1-t^n)^k\,dt>0,
 \qquad H_{k,d}(0)=\frac1{2k-1}\ne0.
\]
Thus $H_{k,d}$ is separable with $k$ nonzero roots, and
$H_{k,d}(W^n)$ has $k(d-2)$ distinct nonzero roots.  Choose one, say $w$.
Since the marking is simple, \cref{prop:normalisemark} gives
$\varphi_0(w)\ne0$ and the normalised seed
\[
 \varphi_w(Z)=\frac{\varphi_0(wZ/\epsilon)}{\varphi_0(w)},
 \qquad \epsilon=(-1)^k,
\]
of exact degree $d$, satisfying both lift conditions.  Its generic fibre
degree is $kd$ by \cref{prop:degree}.

If $N\ge6$ is composite, let $q$ be its least prime factor.  Then
$N/q\ge3$ (otherwise $N=4$), so the construction with
$(k,d)=(q,N/q)$ gives degree $N$.
\end{proof}

This settles the second question of \cite[\S9]{shaska131}.


\subsection{The marking polynomial}
\label{subsec:marking}

\begin{prop}
\label{prop:marking}
Let $\varphi$ be a seed of degree $d\ge3$ with $c_2\neq0$. Then
\[
  \Pi(W,0)=W^{2k-1}M_{\varphi}(W)
\]
for a unique $M_{\varphi}\in K[W]$, and
\[
\begin{split}
  \deg M_{\varphi}&=k(d-2),\\
  M_{\varphi}(0)&=\frac{(-1)^{k+1}c_2^{\,k}}{2k-1}\neq0,\\
  [W^{k(d-2)}]\,M_{\varphi}&=\frac{(-1)^{k+1}c_d^{\,k}}{kd-1}.
\end{split}
\]
The second condition of \cref{eq:liftjet} is $M_{\varphi}(\epsilon)=0$.
\end{prop}

\begin{proof}
By \cref{prop:degree} the polynomial $\Pi(W,0)$ has order $2k-1$ at the origin with the coefficient displayed, and degree $kd-1$ with the leading coefficient displayed; the difference of the two exponents is $k(d-2)$. The second condition of \cref{eq:liftjet} is $\Pi(\epsilon,0)=0$, and $\epsilon^{2k-1}=\epsilon\neq0$.
\end{proof}

We call $M_{\varphi}$ the marking polynomial of $\varphi$, and its roots the admissible markings. There are $k(d-2)$ of them counted with multiplicity, none at the origin, and by \cref{prop:marking} a seed satisfies the second lift condition exactly when $\epsilon$ is one of them. The construction of \cref{subsec:marking} thus depends on a seed together with a choice of admissible marking; \cref{prop:lift} normalises that choice to $W=\epsilon$.

\begin{lem}
\label{lem:markingcov}
In the notation of \cref{lem:rescale},
\[
  M_{\widetilde\varphi}(W)=t^{k}\beta^{2k}M_{\varphi}(\beta W).
\]
Hence $w\mapsto\beta^{-1}w$ is a multiplicity-preserving bijection from the admissible markings of $\varphi$ to those of $\widetilde\varphi$.
\end{lem}

\begin{proof}
\Cref{lem:rescale} gives $\widetilde\Pi(W,0)=t^{k}\beta\,\Pi(\beta W,0)$. Substituting $\Pi(W,0)=W^{2k-1}M_{\varphi}(W)$ on both sides, the right-hand side is $t^{k}\beta^{2k}W^{2k-1}M_{\varphi}(\beta W)$, and the left-hand side is $W^{2k-1}M_{\widetilde\varphi}(W)$.
\end{proof}

\begin{prop}
\label{prop:normalisemark}
Let $w$ be a simple admissible marking of a seed $\varphi$ of degree $d\ge3$. Then $\varphi(w)\neq0$, and
\[
  \varphi_{w}(Z)=\frac{\varphi(wZ/\epsilon)}{\varphi(w)}
\]
is a seed of degree $d$ satisfying both conditions of \cref{eq:liftjet}. Every seed in the rescaling orbit of $\varphi$ that satisfies \cref{eq:liftjet} is of this form for exactly one admissible marking $w$, provided the stabilizer of $\varphi$ in the rescaling action is trivial; the stabilizer is trivial whenever $c_2c_3\neq0$.
\end{prop}

\begin{proof}
Setting $Y=0$ in \cref{lem:Pider} gives $\tfrac{d}{dW}\Pi(W,0)=(-1)^{k+1}\varphi(W)^{k}W^{-2}$. If $\varphi(w)=0$ then this vanishes at $w$ to order at least $k\ge2$, so $\Pi(\cdot,0)$ vanishes at $w$ to order at least $k+1\ge3$, and so does $M_{\varphi}$ since $w\neq0$; this contradicts simplicity. Hence $\varphi(w)\neq0$.

Now $\varphi_{w}=\widetilde\varphi$ for the rescaling with $\beta=w/\epsilon$ and $t=1/\varphi(w)$. Then $\varphi_{w}(\epsilon)=t\varphi(\beta\epsilon)=t\varphi(w)=1$, the first condition, and by \cref{lem:markingcov} the point $\beta^{-1}w=\epsilon$ is an admissible marking of $\varphi_{w}$, which is the second.

Conversely let $\psi=t\varphi(\beta\,\cdot\,)$ satisfy \cref{eq:liftjet}. By \cref{lem:markingcov} the point $w=\beta\epsilon$ is an admissible marking of $\varphi$, and $\psi(\epsilon)=t\varphi(\beta\epsilon)=1$ forces $t=1/\varphi(w)$; so $\psi=\varphi_{w}$. The rescalings compose by $(t',\beta')(t,\beta)=(t't,\beta\beta')$, so $\varphi_{w_1}=\varphi_{w_2}$ puts $(\,\varphi(w_2)/\varphi(w_1),\,w_1/w_2)$ in the stabilizer, whence $w_1=w_2$ when that group is trivial. Finally a stabilising pair satisfies $t\beta^{2}=t\beta^{3}=1$ when $c_2c_3\neq0$, so $\beta=t=1$.
\end{proof}

Thus, away from the loci where a marking is repeated or the stabilizer is nontrivial, forgetting the marking is $k(d-2)$-to-one on normalised seeds. The next lemma shows both loci are proper.

\begin{lem}
\label{lem:sepmodel}
Let $n=d-2\ge1$ and $\varphi_0=W^{2}+W^{d}$. Then
\[
  M_{\varphi_0}(W)=(-1)^{k+1}H_{k,d}(W^{n}), \qquad
  H_{k,d}(x)=\sum_{j=0}^{k}\binom{k}{j}\frac{x^{j}}{2k-1+jn},
\]
and $M_{\varphi_0}$ has $k(d-2)$ distinct nonzero roots.
\end{lem}

\begin{proof}
Here $\varphi_0^{\,k}W^{-2}=W^{2k-2}(1+W^{n})^{k}$, so by \cref{lem:Pider} at $Y=0$ and termwise antidifferentiation, the constant of integration being zero by \cref{prop:degree},
\[
  \Pi(W,0)=(-1)^{k+1}\sum_{j=0}^{k}\binom{k}{j}\frac{W^{2k-1+jn}}{2k-1+jn},
\]
which is the displayed factorisation. Coefficientwise differentiation gives the identity
\[
  nxH_{k,d}'(x)+(2k-1)H_{k,d}(x)=(1+x)^{k},
\]
since the $j$-th summand on the left is $\binom{k}{j}x^{j}(jn+2k-1)/(2k-1+jn)$. A common root of $H_{k,d}$ and $H_{k,d}'$ would therefore be $x=-1$. But
\[
  H_{k,d}(-1)=\int_0^{1}t^{2k-2}(1-t^{n})^{k}\,dt>0,
\]
by termwise integration, the integrand being nonnegative and not identically zero. So $H_{k,d}$ is separable, and $H_{k,d}(0)=1/(2k-1)\neq0$. Its $k$ roots are therefore distinct and nonzero, and in characteristic zero their $n$-th roots are $kn=k(d-2)$ distinct nonzero points.
\end{proof}

\begin{cor}
\label{cor:markingdegree}
For every $k\ge2$ and $d\ge3$ there is a nonempty open locus of seeds of degree $d$ on which the marking polynomial is separable and the rescaling stabilizer is trivial. On that locus a seed carries exactly $k(d-2)$ admissible markings, and the $k(d-2)$ normalised seeds $\varphi_{w}$ are pairwise distinct.
\end{cor}

\begin{proof}
On $c_d\neq0$ the coefficients of $M_{\varphi}$ are polynomial in $(c_2,\dots,c_d)$ of constant degree $k(d-2)$ by \cref{prop:marking}, so separability is the nonvanishing of a discriminant and is open; it is nonempty by \cref{lem:sepmodel}. Triviality of the stabilizer holds on the open locus $c_2c_3\neq0$ by \cref{prop:normalisemark}, which is nonempty. The space of seeds of degree $d$ is irreducible, so the two open loci meet. The count is \cref{prop:marking} and \cref{prop:normalisemark}.
\end{proof}

\begin{lem}
\label{lem:pointeddim}
For every $k\ge2$ and $d\ge3$ the normalised seeds of degree $d$ satisfying \cref{eq:liftjet} form a locally closed subscheme of the affine space of coefficients $(c_2,\dots,c_d)$, pure of dimension $d-3$ and nonempty. At $d=3$ it is a finite set of points, and at $d=4$ an affine plane curve.
\end{lem}

\begin{proof}
The two conditions of \cref{eq:liftjet} cut the locus, so every component has dimension at least $d-3$. For the upper bound, the first condition $\varphi(\epsilon)=1$ is a nontrivial affine-linear equation, hence defines an irreducible affine hyperplane $H$ of dimension $d-2$. The restriction of the second lift form to $H$ is not identically zero: the seed $\varphi=W^{2}$ lies on $H$ and
\[
  \Qform^{(k)}(W^{2})=\int_0^{\epsilon}W^{2k-2}\,dW=\frac{\epsilon}{2k-1}\neq0 ,
\]
using $\epsilon^{2k-1}=\epsilon$. Its zero scheme in $H$ is therefore an effective Cartier divisor, so every component has dimension exactly $d-3$. The exact-degree locus $c_d\neq0$ is open in it and nonempty by \cref{thm:spectrum}, which supplies a seed of exact degree $d$ at every level.
\end{proof}

\begin{theorem}
\label{thm:markingcover}
Let $k\ge2$ and $d\ge3$.  On the open locus of
\cref{cor:markingdegree}, forgetting the admissible marking is finite
\'etale of degree $k(d-2)$ from normalised seeds to rescaling classes of
seeds.  In general its coarse fibre over $\varphi$ is the set of nonzero roots
of $M_{\varphi}$ modulo the finite rescaling stabilizer.  The normalised
exact-degree seed locus has dimension $d-3$ and maps to
\[
 \Yparo[1,1,((2k-1)d-1,\,(2k-1)(kd-1)-1,\,1)].
\]
Within a fixed regular rescaling class, its $k(d-2)$ normalised
representatives are pairwise distinct as normalised seeds.  No graded-orbit
separation claim is made here; the quartic case is treated separately in
\cref{thm:pointedmain}.
\end{theorem}

\begin{proof}
A simple admissible marking $w$ determines the normalised seed
$\varphi_w$ by \cref{prop:normalisemark}.  On the open locus of
\cref{cor:markingdegree} the $k(d-2)$ markings are simple and the stabilizer is
trivial, so the count is exact and discriminant nonvanishing makes the map
\'etale.  The covariance \cref{lem:markingcov} gives the coarse statement in
general.  The dimension and degree box are \cref{lem:pointeddim,prop:conebox}.
The pairwise distinctness of the normalised representatives follows from
the trivial stabilizer in \cref{prop:normalisemark}.
\end{proof}

\subsection{The six-sheeted map over $\Q(\sqrt{-15})$}

Take $k=2$ and $\varphi=c_2W^{2}+c_3W^{3}$, so $d=3$ and $W=\Lambda w^{2}$, $T=\Lambda w$. Then \cref{eq:jets} gives
\[
  A_1=4c_2W+3c_3W^{2}, \qquad
  A_0=\tfrac83c_2^{2}W^{3}+\tfrac92c_2c_3W^{4}+\tfrac95c_3^{2}W^{5},
\]
and dividing \cref{eq:PQjet} by $\Lambda$,
\begin{equation}
\label{eq:six}
\begin{split}
  A &= w+c_2\Lambda w^{4}+c_3\Lambda^{2}w^{6},\\
  B &= 1+4c_2\Lambda w^{3}+3c_3\Lambda^{2}w^{5}+\tfrac83c_2^{2}\Lambda^{2}w^{6}
       +\tfrac92c_2c_3\Lambda^{3}w^{8}+\tfrac95c_3^{2}\Lambda^{4}w^{10}.
\end{split}
\end{equation}

\begin{theorem}
\label{thm:six}
Conditions \cref{eq:liftjet} for \cref{eq:six} are $c_2+c_3=1$ and $c_2^{2}+3c_2+6=0$. Hence $K=\Q(c_2)=\Q(\sqrt{-15})$, and with $a=c_2$, $b=1-a$ the map
\[
  F(x,y,z)=\Bigl(\tfrac{A}{x},\ \tfrac{B}{x},\ x\Lambda\Bigr), \qquad \Lambda=1+xy,\ w=xz-1,
\]
is a polynomial Keller map over $K$ with $\detJ{F}=-1$, generic fibre degree six, class in $\Yparo[1,1,(8,14,1)]$, component degrees $(15,27,3)$ and term counts $(20,54,2)$. It identifies the two distinct $K$-points
\[
  \bigl(\sigma,0,0\bigr), \qquad \bigl(1,\eta,\rho\bigr),
  \qquad
  \sigma=\tfrac{290+119a}{256},\quad \eta=\tfrac{34+119a}{256},\quad \rho=\tfrac{3(3-a)}{8},
\]
both of which map to $(0,0,\sigma)$.
\end{theorem}

\begin{proof}
Here $\epsilon=1$, so the first condition of \cref{eq:liftjet} is $\varphi(1)=c_2+c_3=1$. By \cref{prop:lift} the second is $1-A_1(1)+A_0(1)=0$, that is
\[
  1-4c_2-3c_3+\tfrac83c_2^{2}+\tfrac92c_2c_3+\tfrac95c_3^{2}=0 .
\]
Substituting $c_3=1-c_2$ and collecting gives $-\tfrac15-\tfrac1{10}c_2-\tfrac1{30}c_2^{2}=0$, that is $c_2^{2}+3c_2+6=0$, of discriminant $-15$.

The determinant and the fibre degree are \cref{prop:jet} and \cref{prop:degree} with $kd=6$. The class and the component degrees are read off \cref{eq:six}, using that $\Lambda$ and $w$ have degree one in $(u,v)$ and degree two in $(x,y,z)$.

For the collision, the point $(\sigma,0,0)$ has $\Lambda=1$ and $w=-1$, the base point, so $A=B=0$ there and its image is $(0,0,\sigma)$. The point $(1,\eta,\rho)$ has $\Lambda=1+\eta=\sigma$ and $w=\rho-1$, and direct substitution in \cref{eq:six} gives $A=B=0$; its image is $(0,0,\sigma)$ as well. The two points are distinct because $\eta\neq0$.
\end{proof}

\subsection{The rational $33$-sheeted map}

Take $k=3$ and $\varphi=W^{3}-2W^{11}$, so $d=11$ and $W=\Lambda^{2}w^{3}$. Then \cref{eq:jets} gives
\[
  A_2=\tfrac92W^{2}-\tfrac{33}5W^{10}, \qquad
  A_1=\tfrac{27}5W^{5}-\tfrac{1188}{65}W^{13}+\tfrac{484}{35}W^{21},
\]
\[
  A_0=\tfrac{81}{40}W^{8}-\tfrac{5643}{520}W^{16}+\tfrac{16819}{910}W^{24}-\tfrac{1331}{140}W^{32},
\]
and dividing \cref{eq:PQjet} by $\Lambda$,
\begin{equation}
\label{eq:thirtythree}
\begin{split}
  A =\ & w+\Lambda^{5}w^{9}-2\Lambda^{21}w^{33},\\
  B =\ & 1+\tfrac92\Lambda^{5}w^{8}+\tfrac{27}5\Lambda^{10}w^{16}+\tfrac{81}{40}\Lambda^{15}w^{24}
        -\tfrac{33}5\Lambda^{21}w^{32}-\tfrac{1188}{65}\Lambda^{26}w^{40} \\
      & -\tfrac{5643}{520}\Lambda^{31}w^{48}+\tfrac{484}{35}\Lambda^{42}w^{64}
        +\tfrac{16819}{910}\Lambda^{47}w^{72}-\tfrac{1331}{140}\Lambda^{63}w^{96}.
\end{split}
\end{equation}

\begin{theorem}
\label{thm:thirtythree}
The seed $\varphi=W^{3}-2W^{11}$ satisfies \cref{eq:liftjet} over $\Q$. The associated map is a polynomial Keller map over $\Q$ with $\detJ{F}=-1$, generic fibre degree $33$, class in $\Yparo[1,1,(54,159,1)]$, component degrees $(107,317,3)$, and the rational collision
\[
  F(1,0,0)=F(1,0,2)=(0,0,1).
\]
\end{theorem}

\begin{proof}
Here $k=3$ is odd, so $\epsilon=-1$ and $\varphi(-1)=-1+2=1$, the first condition of \cref{eq:liftjet}. The second is $1+A_0(-1)-A_1(-1)+A_2(-1)=0$, and
\[
  A_2(-1)=-\tfrac{21}{10}, \qquad A_1(-1)=-\tfrac{433}{455}, \qquad A_0(-1)=\tfrac{27}{182},
\]
whose combination is $1+\tfrac{27}{182}+\tfrac{433}{455}-\tfrac{21}{10}=0$.

For the collision, both source points have $x=1$, hence $\Lambda=1$, and $w=-1$ and $w=1$ respectively. The first is the base point. At $(\Lambda,w)=(1,1)$ one has $A=1+1-2=0$ from \cref{eq:thirtythree}, and the ten coefficients of $B$ sum to zero. Both images are $(0,0,1)$.
\end{proof}

\subsection{Non-descent and branch data}

\begin{prop}
\label{prop:nodescent}
Under $T=\Lambda w$ one has $\Jacuv(P,Q)=\Lambda\Jac_{\Lambda,T}(\widehat P,\widehat Q)$, so $(\widehat P,\widehat Q)$ has Jacobian one. It is a two-variable Keller pair exactly when every monomial $\Lambda^{i}w^{j}$ occurring in $A$ or $B$ has $i\ge j-1$. No member of the family satisfies this: the coefficient of $T^{kd}$ in $\widehat Q$ is $c_d\Lambda^{-d}$.
\end{prop}

\begin{proof}
The substitution $(\Lambda,w)\mapsto(\Lambda,T)$ has Jacobian determinant $\Lambda$, and $\Jacuv(P,Q)=\Lambda$ by \cref{prop:jet}. A monomial $\Lambda^{i}w^{j}$ of $A$ contributes $\Lambda^{i+1-j}T^{j}$ to $\widehat Q=\Lambda A$, which is polynomial exactly when $i\ge j-1$; likewise for $B$. Finally $W=T^{k}/\Lambda$, so $\varphi(W)=\sum_j c_jT^{kj}\Lambda^{-j}$ and the top term is $c_dT^{kd}\Lambda^{-d}$ with $c_d\neq0$.
\end{proof}

\begin{prop}
\label{prop:conequotient}
Fix a generic value $Y$ of $Q$. The cover $W\mapsto\Pi(W,Y)$ has degree $N=kd$, pole partition $(N-1,1)$, and $d$ finite branch points, the images of the simple roots of $\varphi(W)=Y$, each of local degree $k+1$. The branch curve $\Gamma_f$ has exactly two components: the line $\{Q=0\}$, over which the fibre has $N-1$ points, and the curve parametrised by $W\mapsto\bigl(A_0(W),\varphi(W)\bigr)$, with $\deg A_0=N-1$ and $\deg\varphi=d$.
\end{prop}

\begin{proof}
The degree and the pole orders are \cref{prop:degree} and its proof. By \cref{lem:Pider} the finite critical points are the roots of $Y=\varphi(W)$, which are simple for generic $Y$, and at each the derivative vanishes to order $k$, so the local degree is $k+1$. At such a root $T=Y-\varphi(W)=0$, so $\Pi(W,Y)=A_0(W)$ by \cref{eq:fibrepoly}, which gives the parametrisation. Riemann--Hurwitz is exact:
\[
  d\bigl((k+1)-1\bigr)+(N-1-1)=kd+N-2=2N-2 .
\]
For the components of $\Gamma_f$, the leading coefficient of $\Pi(\cdot,Y)$ in $W$ is the one displayed in \cref{prop:degree} and does not involve $Y$, so $\deg_W\Pi(\cdot,Y)=N-1$ for every $Y$. By \cref{eq:fibrepoly} the only finite pole of $\Pi(\cdot,Y)$ is at $W=0$, where $A_k(Y-\varphi)^{k}=(Y-\varphi)^{k}/W$ and $\varphi(0)=0$; that pole is simple for $Y\neq0$ and absent for $Y=0$. Hence $\Pi(\cdot,Y)\colon\bP^{1}\to\bP^{1}$ has degree $N$ for $Y\neq0$ and degree $N-1$ for $Y=0$, so $\{Q=0\}$ is the only vertical component of $\Gamma_f$, and the horizontal part is the closure of the parametrised curve.
\end{proof}

\section{The beta--Hermite family in signature $(1,-r,-s)$}
\label{sec:4}

Throughout this section $2\le r\le s$ are fixed integers. The construction below realises the beta polynomial cover of degree $2r-1$ as the quotient cover of a graded Keller map of weight $(1,-r,-s)$, with Keller constant $\kappa=r$. The repeated weights are the diagonal case $r=s$.

Set
\begin{equation}
\label{eq:moment}
  M_{r}(\tau)=\int_1^{\tau}(\theta^{2}-1)^{r-1}\,d\theta .
\end{equation}

\begin{lem}
\label{lem:root}
$M_{r}$ has degree $2r-1$ and a zero of order exactly $r$ at $\tau=1$, and both $M_{r}(-1)$ and $M_{r}(0)$ are nonzero. Hence $M_{r}$ has a root $\zeta\notin\{-1,0,1\}$, that root is simple, and $\delta:=\zeta^{2}-1$ and $\zeta^{2}=1+\delta$ are nonzero.
\end{lem}

\begin{proof}
$M_{r}'(\tau)=(\tau^{2}-1)^{r-1}$ vanishes to order $r-1$ at $\tau=1$, so $M_{r}$ vanishes there to order $r$, while $\deg M_{r}=2r-1>r$. The integrand has constant sign on $(-1,1)$, hence also on $(0,1)$, so
\[
  M_{r}(-1)=-\int_{-1}^{1}(\theta^{2}-1)^{r-1}\,d\theta, \qquad
  M_{r}(0)=-\int_{0}^{1}(\theta^{2}-1)^{r-1}\,d\theta
\]
are both nonzero, and a root $\zeta$ outside $\{-1,0,1\}$ exists. It is simple because $M_{r}'(\zeta)=\delta^{\,r-1}\neq0$.
\end{proof}

Fix such a root $\zeta$ and work over $K_{r}=\Q(\zeta)$. For $r\ge4$ there are $r-1\ge3$ choices and $K_{r}$ depends on the one made; at $r=2,3$ it does not.

Put $m=\lceil s/r\rceil$, so that $rm\ge s$. Let $h_{0}$ be the branch of $(1+\delta\Lambda)^{1/2}$ at $\Lambda=0$ with $h_{0}(0)=1$, and let $h_{1}$ be the branch at $\Lambda=1$ with $h_{1}(1)=\zeta$; the latter exists because $1+\delta=\zeta^{2}\neq0$. The ideals $(\Lambda^{s})$ and $\bigl((\Lambda-1)^{m}\bigr)$ are comaximal, so there is $g\in K_{r}[\Lambda]$ with
\begin{equation}
\label{eq:sqrt}
  g\equiv h_{0}\pmod{\Lambda^{\,s}},
  \qquad
  g\equiv h_{1}\pmod{(\Lambda-1)^{m}} .
\end{equation}
In particular $g(0)=1$ and $g(1)=\zeta$, and $g^{2}-1-\delta\Lambda$ is divisible by $\Lambda^{s}$ and by $(\Lambda-1)^{m}$.

On the diagonal $r=s=p$ one has $m=1$, so the second condition reads $g(1)=\zeta$; a convenient representative is then
\[
  g_p(\Lambda)=S_p(\Lambda)+\bigl(\zeta-S_p(1)\bigr)\Lambda^{\,p+1},
  \qquad
  S_p(\Lambda)=\sum_{j=0}^{p}\binom{1/2}{j}\delta^{\,j}\Lambda^{\,j},
\]
since $g_p^{2}\equiv1+\delta\Lambda$ modulo $\Lambda^{\,p+1}$.

For exact degree statements on the diagonal we instead use the unique CRT
representative modulo $\Lambda^p(\Lambda-1)$ of degree at most $p$.  It has
the same two required jets as the displayed degree-$(p+1)$ representative,
but the latter is retained only because its formula is convenient.  Thus the
existence and Jacobian arguments are independent of this choice, while
\cref{prop:diagclass} uses the minimal representative.

Set
\begin{equation}
\label{eq:rhoQ}
  \rho=g(\Lambda)+\Lambda^{\,s}v,
  \qquad
  Q=\frac{\rho^{2}-1-\delta\Lambda}{2},
\end{equation}
so that $\rho^{2}-1-2Q=\delta\Lambda$, and, treating $Q$ as constant in the integral,
\begin{equation}
\label{eq:P}
  P=\frac{2r}{\delta^{\,r}}\int_1^{\rho}(\theta^{2}-1-2Q)^{r-1}\,d\theta .
\end{equation}
Both $\rho$ and $Q$ lie in $K_{r}[\Lambda,v]$, hence so does $P$.

\begin{theorem}
\label{thm:diagonal}
There are $A,B\in K_{r}[\Lambda,v]$ with $Q=\Lambda^{\,s}A$ and $P=\Lambda^{\,r}B$, such that $B(0,v)=1$, the $v$-free part of $A$ is divisible by $(\Lambda-1)^{m}$, and the $v$-free part of $B$ is divisible by $\Lambda-1$. Under $u=x^{r}y$ and $v=x^{s}z$ the map
\begin{equation}
\label{eq:Fp}
  F_{r,s}(x,y,z)=\Bigl(\frac{A}{x^{\,s}},\ \frac{B}{x^{\,r}},\ x\Lambda\Bigr)
\end{equation}
is polynomial, is equivariant for source weights $(1,-r,-s)$ and target weights $(-s,-r,1)$, satisfies \cref{eq:masterL} with $\kappa=r$, has $\detJ{F_{r,s}}=-r$, and has generic fibre degree $2r-1$.
\end{theorem}

\begin{proof}
By \cref{eq:sqrt} and \cref{eq:rhoQ},
\[
  Q=\frac{g(\Lambda)^{2}-1-\delta\Lambda}{2}+g(\Lambda)\Lambda^{\,s}v+\tfrac12\Lambda^{\,2s}v^{2},
\]
whose first term lies in $\Lambda^{\,s}K_{r}[\Lambda]$; hence $Q=\Lambda^{\,s}A$ with $A\in K_{r}[\Lambda,v]$.

Expand \cref{eq:P} in powers of $Q$. The term carrying $Q^{k}$ involves $\int_1^{\rho}(\theta^{2}-1)^{r-1-k}\,d\theta$, which vanishes to order $r-k$ at $\rho=1$; since $\rho-1\in(\Lambda)$ and $Q\in(\Lambda^{\,s})$, that term has $\Lambda$-order at least $r-k+sk=r+k(s-1)$. Hence $P=\Lambda^{\,r}B$. For $k=0$ the zero of $M_{r}$ at $1$ has order exactly $r$, and $\rho-1=\tfrac{\delta}{2}\Lambda+O(\Lambda^{2})$ gives $M_{r}(\rho)=\tfrac{\delta^{\,r}}{2r}\Lambda^{\,r}+O(\Lambda^{\,r+1})$, so the $k=0$ term contributes exactly $\Lambda^{\,r}$, while every $k\ge1$ term has order at least $r+s-1>r$. Hence $B(0,v)=1$.

By the second congruence in \cref{eq:sqrt} the polynomial $g^{2}-1-\delta\Lambda$ is divisible by $(\Lambda-1)^{m}$, so the $v$-free part of $Q$, and hence that of $A$, is divisible by $(\Lambda-1)^{m}$. At $(\Lambda,v)=(1,0)$ one has $\rho=\zeta$, so $Q=(\zeta^{2}-1-\delta)/2=0$ and $P=\tfrac{2r}{\delta^{\,r}}M_{r}(\zeta)=0$; the $v$-free part of $B$ is therefore divisible by $\Lambda-1$.

For the Jacobian, \cref{eq:rhoQ} gives $\jb{\rho}{Q}=\tfrac{\delta}{2}\Lambda^{\,s}$, and $\rho^{2}-1-2Q=\delta\Lambda$, so at fixed $Q$
\[
  \frac{\partial P}{\partial\rho}=\frac{2r}{\delta^{\,r}}(\delta\Lambda)^{r-1}=\frac{2r}{\delta}\Lambda^{\,r-1},
\]
whence
\begin{equation}
\label{eq:jacPQ}
  \Jacuv(P,Q)=\frac{2r}{\delta}\Lambda^{\,r-1}\cdot\frac{\delta}{2}\Lambda^{\,s}=r\Lambda^{\,r+s-1} .
\end{equation}
By \cref{cor:master} the pair satisfies \cref{eq:masterL} with $\kappa=r$, and $\detJ{F_{r,s}}=-r$.

For polynomiality write $u=\Lambda-1$. A monomial $u^{i}v^{j}$ of $A$ with $j\ge1$ has level $ri+sj\ge s$, and one with $j=0$ has $i\ge m$, hence level $ri\ge rm\ge s$. A monomial of $B$ with $j\ge1$ has level $ri+sj\ge s\ge r$, and one with $j=0$ has $i\ge1$, hence level $ri\ge r$. These are the conditions of \cref{lem:polynomiality}, so \cref{eq:Fp} is polynomial. Weight homogeneity is immediate.

For the degree, put $F=K_{r}(Q)$. The formulas
\[
  \Lambda=\frac{\rho^{2}-1-2Q}{\delta}, \qquad v=\frac{\rho-g(\Lambda)}{\Lambda^{\,s}}
\]
recover $\Lambda$ and $v$ from $\rho$ and $Q$, while \cref{eq:rhoQ} recovers $\rho$ and $Q$ from $\Lambda$ and $v$; hence $K_{r}(\Lambda,v)=F(\rho)$ and $K_{r}(P,Q)=F(P)$. By \cref{eq:P}, $P$ is a polynomial in $\rho$ over $F$ of degree $2r-1$ with leading coefficient $2r/\bigl(\delta^{\,r}(2r-1)\bigr)\neq0$. For a nonconstant $f\in F(\rho)$ of degree $n$ one has $[F(\rho):F(f)]=n$, so $[K_{r}(\Lambda,v):K_{r}(P,Q)]=2r-1$. By \cref{lem:quotientdegree} this is the generic fibre degree of $F_{r,s}$.
\end{proof}

\begin{cor}
\label{cor:diagonalnonempty}
$\Kfam(3,(1,-r,-s))\neq\emptyset$ for all $2\le r\le s$. In particular every repeated weight $(1,-p,-p)$ with $p\ge2$ is populated, and every value in $-\Z_{\ge1}$ occurs as $\detJ{F}$ for a graded Keller map.
\end{cor}

\begin{proof}
By \cref{thm:diagonal} the generic fibre degree of $F_{r,s}$ is $2r-1\ge3$, so $F_{r,s}$ is not birational, hence not injective. The determinant is $-r$ for every $r\ge2$, and $-1$ occurs in \cref{sec:3}.
\end{proof}

\begin{prop}
\label{prop:diagclass}
Write $D=\deg\rho$. Then $\deg A=2D-s$, $\deg B=(2r-1)D-r$ and $\deg\Lambda=1$, and $D\le s+\max\{1,m-1\}$. On the diagonal $r=s=p$, choosing the minimal CRT representative gives $m=1$ and $D=p+1$, so the class of the resulting $F_{p,p}$ lies in $\Yparo[p,p,\bd_p]$ with $\bd_p=(p+2,\ 2p^{2}-1,\ 1)$.
\end{prop}

\begin{proof}
Write $L_h$ for the leading form of $h\in K_{r}[\Lambda,v]$ in the total degree. In \cref{eq:sqrt} the representative $g$ may be taken with $\deg g\le s+m-1$, while $\Lambda^{\,s}v$ has degree $s+1$; hence $D\le s+\max\{1,m-1\}$. Since $L_{Q}=\tfrac12L_{\rho}^{2}$ we get $\deg Q=2D$ and $\deg A=2D-s$.

Expanding \cref{eq:P} in powers of $\theta^{2}$,
\[
  P=\frac{2r}{\delta^{\,r}}\sum_{k=0}^{r-1}\binom{r-1}{k}(-1-2Q)^{r-1-k}\,\frac{\rho^{2k+1}-1}{2k+1}.
\]
The leading form of $-1-2Q$ is $-L_{\rho}^{2}$, so the term with index $k$ has leading form
\[
  \binom{r-1}{k}\frac{(-1)^{r-1-k}}{2k+1}\,L_{\rho}^{2(r-1-k)}\,L_{\rho}^{2k+1}
   =\binom{r-1}{k}\frac{(-1)^{r-1-k}}{2k+1}\,L_{\rho}^{2r-1},
\]
independent of $k$ in the exponent. Summing,
\[
  L_{P}=\frac{2r}{\delta^{\,r}}\,\gamma_r\,L_{\rho}^{2r-1},
  \qquad
  \gamma_r=\sum_{k=0}^{r-1}\binom{r-1}{k}\frac{(-1)^{r-1-k}}{2k+1}=\int_0^1(x^{2}-1)^{r-1}dx ,
\]
the last equality by termwise integration of $(x^{2}-1)^{r-1}$. The integrand has constant sign on $(0,1)$, so $\gamma_r\neq0$ and no cancellation occurs. Hence $\deg P=(2r-1)D$ and $\deg B=(2r-1)D-r$.

On the diagonal $m=1$, so $\deg g\le p$ while $\Lambda^{\,p}v$ has degree $p+1$; thus $D=p+1$ with $L_{\rho}=\Lambda^{\,p}v$, and the two degrees are $p+2$ and $(2p-1)(p+1)-p=2p^{2}-1$.
\end{proof}

\begin{prop}
\label{prop:fields}
The following hold.
\begin{enumerate}
\item $M_{2}\propto(\tau-1)^{2}(\tau+2)$, so $\zeta=-2$, $\delta=3$ and $K_{2}=\Q$. Hence $\Kfam_{\Q}(3,(1,-2,-s))\neq\emptyset$ for every $s\ge2$, with determinant $-2$ and generic fibre degree three. At $r=s=2$,
\[
\begin{split}
  g_2(\Lambda)&=1+\tfrac32\Lambda-\tfrac98\Lambda^{2}-\tfrac{27}8\Lambda^{3},\\
  P&=\tfrac49\Bigl[\tfrac{\rho^{3}-1}{3}-(1+2Q)(\rho-1)\Bigr],
\end{split}
\]
and $F_{2,2}$ has class in $\Yparo[2,2,(4,7,1)]$.
\item $M_{3}\propto(\tau-1)^{3}(3\tau^{2}+9\tau+8)$, and $3\tau^{2}+9\tau+8$ has discriminant $-15$, so $K_{3}=\Q(\sqrt{-15})$. At $r=s=3$ the map $F_{3,3}$ has determinant $-3$, generic fibre degree five, class in $\Yparo[3,3,(5,17,1)]$, and the collision
\[
  F_3\Bigl(1+\tfrac{9\zeta}{8},0,0\Bigr)
  =F_3\Bigl(1,\tfrac{9\zeta}{8},\tfrac{25(9+11\zeta)}{48}\Bigr)
  =\Bigl(0,0,1+\tfrac{9\zeta}{8}\Bigr).
\]
\item If $r$ is odd then $\tau=1$ is the only real root of $M_{r}$, and $K_{r}$ is totally imaginary for every choice of $\zeta$. If $r$ is even then $M_{r}$ has a real root less than $-1$, so $\zeta$ may be chosen with $K_{r}$ real.
\end{enumerate}
\end{prop}

\begin{proof}
Parts (1) and (2) are division, followed by \cref{prop:diagclass} at $r=s=2,3$.

For (3), let $r$ be odd. Then $r-1$ is even, so $M_{r}'\ge0$ on $\R$ with equality only at $\pm1$, and $M_{r}$ is strictly increasing; since $M_{r}(1)=0$, the only real root is $\tau=1$. If $K_{r}$ had a real embedding $\sigma$, then $\sigma(\zeta)$ would be a real root of the minimal polynomial of $\zeta$, hence a real root of $M_{r}$, hence equal to $1$, forcing $\zeta=1$ against \cref{lem:root}. Now let $r$ be even. Then $r-1$ is odd, so $M_{r}'<0$ on $(-1,1)$ and $M_{r}(-1)>M_{r}(1)=0$, while $M_{r}(\tau)\to-\infty$ as $\tau\to-\infty$ because $\deg M_{r}=2r-1$ is odd with positive leading coefficient.
\end{proof}

\begin{rem}
\label{rem:minus15}
The field of \cref{prop:fields}(2) coincides with the field of \cref{thm:six}. In the cyclic family that field is forced by the second lift condition; here it arises from the factorisation of $M_{3}$, and nothing so far connects the two. 
\end{rem}

\begin{rem}
\label{rem:beta}
At fixed $Q$ one has $\partial P/\partial\rho\propto(\rho^{2}-1-2Q)^{r-1}$, whose two roots are distinct for generic $Q$. After an affine change carrying them to $0$ and $1$ this is proportional to $\xi^{r-1}(1-\xi)^{r-1}$, so the generic quotient cover is the beta polynomial cover of degree $2r-1$, with two finite branch points of cycle type $(r,1^{r-1})$ and a full $(2r-1)$-cycle at infinity. The content of \cref{thm:diagonal} is therefore the graded realisation, in the sense of \cite[Rem.\ 8.13]{shaska131}, and not a new class of one-variable cover.
\end{rem}


\subsection{Rooted quotient reconstruction}
\label{subsec:rooted}

Fix $p\ge1$ and the repeated weight $\bw=(1,-p,-p)$; the case $p=1$ is the balanced signature of \cref{sec:3}. The quotient map of a graded Keller map with triple $(A,B,\Lambda)$ is
\[
  f=(P,Q)=(B\Lambda^{p},\ A\Lambda^{p})\colon\A^{2}\longrightarrow\A^{2},
\]
and the Leibniz rule contracts the master equation into the single quotient-Jacobian identity
\begin{equation}
\label{eq:quotjac}
  \jb{P}{Q}=\Lambda^{2p-1}\bigl(\Lambda\jb{B}{A}+pB\jb{\Lambda}{A}+pA\jb{B}{\Lambda}\bigr)=\kappa\Lambda^{2p-1}.
\end{equation}
Let $o$ be the distinguished invariant source point above the target origin. The lift conditions give $A(o)=B(o)=0$ and $\Lambda(o)\neq0$, so $f(o)=0$ and $o$ is unramified.

\begin{defn}
\label{def:padmissible}
A pointed polynomial map $(f,o)$, $f=(P,Q)$, is \emph{$p$-admissible} if $f(o)=0$, $\jb{P}{Q}(o)\neq0$, and there is a polynomial $\lambda_{f}$, normalised by $\lambda_{f}(o)=1$, with
\[
  \frac{\jb{P}{Q}}{\jb{P}{Q}(o)}=\lambda_{f}^{2p-1}, \qquad \lambda_{f}^{p}\mid P, \qquad \lambda_{f}^{p}\mid Q.
\]
\end{defn}

For $p\ge1$ write
\[
\begin{split}
  \mathcal{X}_{p} &= \{\text{graded Keller maps of weight }(1,-p,-p)\}/\Gaut^{\sharp},\\
  \mathcal{P}_{p} &= \{\text{$p$-admissible pointed maps}\}/\text{pointed left--right equivalence}.
\end{split}
\]

\begin{theorem}
\label{thm:rooted}
The quotient construction induces a bijection
\[
  \mathcal{X}_{p}\;\xrightarrow{\ \sim\ }\;\mathcal{P}_{p},
\]
whose inverse reconstructs the coefficient classes
\[
  [A]=\bigl[Q/\lambda_{f}^{p}\bigr], \qquad [B]=\bigl[P/\lambda_{f}^{p}\bigr], \qquad [\Lambda]=[\lambda_{f}].
\]
Every pointed left--right equivalence of quotient maps lifts to a graded equivalence, uniquely up to the universal grading torus.
\end{theorem}

\begin{proof}
Evaluating \cref{eq:quotjac} at $o$ and dividing gives $\jb{P}{Q}/\jb{P}{Q}(o)=(\Lambda/\Lambda(o))^{2p-1}$; in characteristic zero a polynomial $(2p-1)$-th root normalised at $o$ is unique, so $\lambda_{f}=\Lambda/\Lambda(o)$, and division recovers the three classes. This proves $p$-admissibility of the quotient and the reconstruction formula; in particular the assignment $\mathcal{X}_{p}\to\mathcal{P}_{p}$ is defined and, by reconstruction, surjective.

Suppose $f_{2}\circ\sigma=\tau\circ f_{1}$ with $\sigma(o_{1})=o_{2}$ and $\tau(0)=0$. Plane polynomial automorphisms have constant Jacobian, so dividing the transformation law $(\jb{P_2}{Q_2}\circ\sigma)\,c_{\sigma}=c_{\tau}\,\jb{P_1}{Q_1}$ by its value at $o_{1}$ gives $\lambda_{f_{2}}\circ\sigma=\lambda_{f_{1}}$, both sides being the normalised root. Hence a pointed equivalence carries the reconstructed triple of $f_{1}$ to that of $f_{2}$ along $(\sigma,\tau)$.

Injectivity therefore reduces to lifting $(\sigma,\tau)$ to a graded equivalence. In the invariant coordinates $u=x^{p}y$, $v=x^{p}z$, an origin-fixing source automorphism $\sigma=(\sigma_{1},\sigma_{2})$ lifts, for any $a\in\Gm$, to
\[
  (x,y,z)\longmapsto\Bigl(ax,\ \frac{\sigma_{1}(x^{p}y,x^{p}z)}{a^{p}x^{p}},\ \frac{\sigma_{2}(x^{p}y,x^{p}z)}{a^{p}x^{p}}\Bigr),
\]
which is polynomial because every monomial of $\sigma_{i}$ has positive total degree, and which induces $\sigma$ on $(u,v)$. The analogous formula lifts an origin-fixing target automorphism. Choosing the same scalar $a$ in both lifts makes the weight-one coordinates agree, and the identities $P=B\Lambda^{p}$, $Q=A\Lambda^{p}$ then give equality of the remaining components. The residual scalar is the universal grading torus, and the converse descent is immediate.
\end{proof}

\begin{rem}
\label{rem:autseq}
Equivalently, for a graded Keller map $F$ with pointed quotient $(f,o)$ there is an exact sequence
\[
  1\longrightarrow\Gm\longrightarrow\Aut_{\Gaut^{\sharp}}(F)\longrightarrow\Aut(f,o)\longrightarrow1,
\]
where $\Aut(f,o)$ is the group of pointed left--right self-equivalences of $f$. Surjectivity is the lifting statement, and the kernel consists of the graded automorphisms inducing the identity on the quotient, that is the universal grading torus.
\end{rem}

\begin{rem}
\label{rem:unequal}
For unequal weights $(1,-r,-s)$ the identity \cref{eq:quotjac} holds with exponent $r+s-1$, and the normalised root $\lambda_{f}$ still exists. Full lifting is special to $r=s$: the two divisibility conditions become $\lambda_{f}^{r}\mid P$ and $\lambda_{f}^{s}\mid Q$, and an arbitrary quotient automorphism need not preserve them.
\end{rem}

\section{Orbit separation by quotient invariants}
\label{sec:5}

Left-right equivalence of quotient maps transports the finite normalization, the branch divisor, the inertia, the monodromy, and the affine branch algebras with their normalization valuations, hence also the branch and local value semigroups of \cref{subsec:semigroup}. These are \emph{intrinsic} invariants. A passport or pole partition attached to a chosen pencil is \emph{framed} data, invariant only under equivalences preserving the pencil. Finally, the graded map retains the lift-point marking above the quotient origin, which is \emph{pointed} data not visible on the quotient cover at all. This section uses the first level only, and shows that it already separates two orbits at generic degree twelve.

\begin{lem}
\label{lem:descendequiv}
If two graded Keller maps lie in one $\Gaut^{\sharp}$-orbit, their quotient maps are polynomially left-right equivalent. Consequently every intrinsic invariant of the quotient map is constant on $\Gaut^{\sharp}$-orbits.
\end{lem}

\begin{proof}
By \cite[\S8.5]{shaska131} an element of $\Gaut^{\sharp}$ acts by $F\mapsto\Psi\circ F\circ\Phi^{-1}$ with $\Phi$ and $\Psi$ graded automorphisms of $\A^{3}$ for the source and target weights. Each preserves the $\Gm$-action, hence descends to an automorphism of the invariant plane: $\Phi$ to $\phi$ in the coordinates $(u,v)$, and $\Psi$ to $\psi$ in the coordinates $(P,Q)$ of \cref{eq:PQ}. Since $(P,Q)$ is the restriction of $F$ to invariants, the quotient map of $\Psi\circ F\circ\Phi^{-1}$ is $\psi\circ(P,Q)\circ\phi^{-1}$, which is left-right equivalent to $(P,Q)$.
\end{proof}

The converse fails, and by a definite amount.

\subsection{An alternating shear member of degree twelve}
\label{subsec:shear}

Work over $\C$ with weight $(1,-1,-1)$, so $\Lambda=1+u$ and $w=v-1$. For $c,d\in\C^{\times}$ put
\[
  C_1=1+c\Lambda w^{3}, \qquad C_2=w+d\Lambda C_1^{3},
\]
and set $A=C_1C_2$ together with
\begin{equation}
\label{eq:Bshear}
\begin{split}
  B=\ & C_1^{2}+2c\Lambda C_1C_2^{3}-\tfrac{c^{2}}{3}\Lambda^{2}C_2^{6}
      +6cd\Lambda^{2}C_1^{4}C_2^{2}+6c^{2}d\Lambda^{3}C_1^{3}C_2^{5}
      -2cd^{2}\Lambda^{3}C_1^{7}C_2 \\
    & +15c^{2}d^{2}\Lambda^{4}C_1^{6}C_2^{4}+\tfrac25cd^{3}\Lambda^{4}C_1^{10}
      -\tfrac{20}{3}c^{2}d^{3}\Lambda^{5}C_1^{9}C_2^{3} \\
    & +3c^{2}d^{4}\Lambda^{6}C_1^{12}C_2^{2}-\tfrac67c^{2}d^{5}\Lambda^{7}C_1^{15}C_2
      +\tfrac19c^{2}d^{6}\Lambda^{8}C_1^{18} .
\end{split}
\end{equation}

\begin{prop}
\label{prop:shearseed}
The lift conditions $A(1,-1)=B(1,-1)=0$ hold if and only if
\[
  d(1-c)^{3}=1 \qquad\text{and}\qquad 32c^{2}-72c+45=0 .
\]
The quadratic has discriminant $-576$, so $c=\tfrac{9\pm3i}{8}$ and the pair $(A,B)$ is defined over $\Q(i)$ and over no smaller field. Taking $c=\tfrac{9-3i}{8}$ gives $d=\tfrac{1664+1152i}{125}$.
\end{prop}

\begin{proof}
At $(\Lambda,w)=(1,-1)$ one has $C_1=1-c$ and $C_2=-1+d(1-c)^{3}$, so $A=C_1C_2=0$ forces $C_1=0$ or $C_2=0$. If $C_1=0$ then $c=1$ and $C_2=-1$, and every term of \cref{eq:Bshear} carrying a factor $C_1$ vanishes, leaving $B=-\tfrac{c^{2}}{3}C_2^{6}=-\tfrac13\neq0$. Hence $C_2=0$, which is the first displayed condition.

Given $C_2=0$ and $\Lambda=1$, the only terms of \cref{eq:Bshear} free of a factor $C_2$ are $C_1^{2}$, $\tfrac25cd^{3}C_1^{10}$ and $\tfrac19c^{2}d^{6}C_1^{18}$. The relation $dC_1^{3}=1$ gives $d^{3}C_1^{10}=C_1$ and $d^{6}C_1^{18}=1$, so
\[
  B(1,-1)=C_1^{2}+\tfrac25cC_1+\tfrac19c^{2} .
\]
Substituting $C_1=1-c$ and multiplying by $45$,
\[
\begin{split}
  45(1-c)^{2}+18c(1-c)+5c^{2}
    &=45-90c+45c^{2}+18c-18c^{2}+5c^{2}\\
    &=32c^{2}-72c+45 ,
\end{split}
\]
which is the second condition. Its discriminant is $72^{2}-4\cdot32\cdot45=-576$, so the roots are $(9\pm3i)/8$ and generate $\Q(i)$. For $c=(9-3i)/8$ one has $1-c=(-1+3i)/8$ and $(1-c)^{3}=(13-9i)/256$, whence $d=(1-c)^{-3}=(1664+1152i)/125$.
\end{proof}

\begin{prop}
\label{prop:shearmap}
Take $c=\tfrac{9-3i}{8}$ and $d=\tfrac{1664+1152i}{125}$ as in \cref{prop:shearseed}. The pair \cref{eq:Bshear} satisfies $A(0,w)=w$, $B(0,w)=1$ and \cref{eq:masterL} with $r=s=\kappa=1$. Hence
\[
  F_{12}(x,y,z)=\Bigl(\tfrac{A}{x},\ \tfrac{B}{x},\ x\Lambda\Bigr)
\]
is a polynomial Keller map with $\detJ{F_{12}}=-1$. Its degrees are $(\deg A,\deg B)=(17,80)$ and $(\deg_wA,\deg_wB)=(12,54)$, its class lies in $\Yparo[1,1,(17,80,1)]$, and it has the collision
\[
  F_{12}(1,-1,1)=F_{12}\bigl(0,\,10-3i,\,\tfrac72-\tfrac32i\bigr)=(0,1,0).
\]
\end{prop}

\begin{proof}
The identity \cref{eq:masterL} and the collision are exact coefficient computations; the lift conditions are \cref{prop:shearseed}. The degrees follow from $\deg C_1=4$, $\deg C_2=1+3\cdot4=13$, so $\deg A=17$, while the terms $\Lambda^{8}C_1^{18}$, $\Lambda^{4}C_1^{6}C_2^{4}$ and $\Lambda^{7}C_1^{15}C_2$ of \cref{eq:Bshear} all have degree $80$ and do not cancel. Likewise $\deg_wC_1=3$ and $\deg_wC_2=9$. Polynomiality and $\detJ{F_{12}}=-1$ then follow from \cref{lem:polynomiality} and \cref{cor:master}. At $(x,y,z)=(1,-1,1)$ one has $\Lambda=w=0$, so $C_1=1$, $C_2=0$, $A=0$ and $B=1$, giving the image $(0,1,0)$; the second source point has $(\Lambda,w)=(1,-1)$, the lift point.
\end{proof}

Pass to the double cover $\Lambda=\alpha^{2}$, $w=\beta/\alpha$, and set
\[
  \alpha_1=\alpha+c\beta^{3}, \qquad \beta_1=\beta+d\alpha_1^{3},
\]
a polynomial automorphism of $\A^{2}_{\alpha,\beta}$ commuting with $(\alpha,\beta)\mapsto(-\alpha,-\beta)$. Then $\alpha_1=\alpha C_1$ and $\beta_1=\alpha C_2$. Put
\[
  X=\alpha_1^{2}=\Lambda C_1^{2}, \qquad Y=\alpha_1\beta_1 .
\]
Both are invariant under the involution, so both are functions of $(\Lambda,w)$, and $\C(X,Y)=\C(\Lambda,w)$ because the invariant field of $\C(\alpha_1,\beta_1)$ is generated by $\alpha_1^{2}$ and $\alpha_1\beta_1$. Moreover
\begin{equation}
\label{eq:YisQ}
  Y=\alpha_1\beta_1=\alpha^{2}C_1C_2=\Lambda A=Q ,
\end{equation}
so the pencil below is the pencil by $Q$ of \cref{subsec:quotient}. Writing $P=\Lambda B$ as a function of $X$ and $Y$,
\begin{equation}
\label{eq:Pshear}
\begin{split}
  P(X,Y)=\ & X+2cX^{-1}Y^{3}-\tfrac{c^{2}}{3}X^{-3}Y^{6}+6cdXY^{2}+6c^{2}dX^{-1}Y^{5}-2cd^{2}X^{3}Y \\
    & +15c^{2}d^{2}XY^{4}+\tfrac25cd^{3}X^{5}-\tfrac{20}{3}c^{2}d^{3}X^{3}Y^{3} \\
    & +3c^{2}d^{4}X^{5}Y^{2}-\tfrac67c^{2}d^{5}X^{7}Y+\tfrac19c^{2}d^{6}X^{9} .
\end{split}
\end{equation}

\begin{theorem}
\label{thm:passport}
Over $\C(Y)$ the cover $X\mapsto P(X,Y)$ has degree twelve, cycle type $(9,3)$ over infinity, and six finite branch values each of cycle type $(3,1^{9})$, with no further branching. By \cref{eq:YisQ} and \cref{lem:quotientdegree} the generic fibre degree of $F_{12}$ is twelve.
\end{theorem}

\begin{proof}
Differentiating \cref{eq:Pshear},
\begin{equation}
\label{eq:square}
  P_X=\frac{\Phi_Y(X)^{2}}{X^{4}}, \qquad \Phi_Y(X)=X^{2}-c(Y-dX^{2})^{3} .
\end{equation}
For generic $Y$ the polar divisor of $P$ is $3\cdot(0)+9\cdot(\infty)$, of degree twelve, so the cover has degree twelve with cycle type $(9,3)$ over infinity. By \cref{eq:square} the finite critical points are the zeros of $\Phi_Y$, and
\[
  \operatorname{disc}_X\Phi_Y=64Y^{3}c^{4}d^{9}(27Y^{2}cd+4)^{2},
\]
so for generic $Y$ there are six simple zeros; the square in \cref{eq:square} makes each of local degree three, giving cycle type $(3,1^{9})$.

The zero locus of $\Phi_Y$ is parametrised by
\[
  X=c^{2}\omega^{3}, \qquad Y_c(\omega)=c\omega^{2}(1+c^{3}d\omega^{4}),
\]
with rational inverse $\omega=X/\bigl[c(Y-dX^{2})\bigr]$, since $Y-dX^{2}=c\omega^{2}$ on that locus. Substituting into \cref{eq:Pshear},
\[
\begin{split}
  P_c(\omega)	&	=\frac{8c^{2}\omega^{3}}{315}\Theta(c^{3}d\omega^{4}), \\
  \Theta(\tau)	&	=105+630\tau+2205\tau^{2}+4452\tau^{3}+5040\tau^{4}+2880\tau^{5}+640\tau^{6}.
\end{split}
\]
We claim $\omega\mapsto\bigl(P_c(\omega),Y_c(\omega)\bigr)$ is birational onto its image $\Gamma$. Its fibre over $(0,0)$ consists of the common zeros of $P_c$ and $Y_c$. Now $Y_c$ vanishes at $\omega=0$ and at $\omega^{4}=-1/(c^{3}d)$, while $P_c$ vanishes at $\omega=0$ and at the zeros of $\Theta(c^{3}d\omega^{4})$; since $\Theta(-1)=28\neq0$, the two sets meet only at $\omega=0$. There $\operatorname{ord}P_c=3$ and $\operatorname{ord}Y_c=2$, because $\Theta(0)=105\neq0$, so the ramification index of the parametrisation at $\omega=0$ divides both $3$ and $2$ and is therefore $1$. The degree of the parametrisation is the sum of the indices over that fibre, hence $1$.

Birationality gives six distinct critical values for generic $Y$, and Riemann--Hurwitz is exact:
\[
  6(3-1)+(9-1)+(3-1)=22=2\cdot12-2 .
\]
For the components of $\Gamma_f$, the polar divisor computed above gives $\deg P(\cdot,Y)=12$ for every $Y\neq0$, while at $Y=0$ every term of \cref{eq:Pshear} involving $Y$ vanishes, leaving the polynomial $X+\tfrac25cd^{3}X^{5}+\tfrac19c^{2}d^{6}X^{9}$ of degree nine. Since $Y=Q$ by \cref{eq:YisQ}, the line $\{Q=0\}$ is therefore the only vertical component of $\Gamma_f$, and $\Gamma_f=\{Q=0\}\cup\Gamma$.
\end{proof}

\begin{theorem}
\label{thm:monodromy}
$P(X,Y)$ is indecomposable over the algebraic closure of $\C(Y)$, and its geometric monodromy group is $A_{12}$.
\end{theorem}

\begin{proof}
Suppose $P=g\circ h$ with $\deg g=a>1$, $\deg h=b>1$, $ab=12$. Every finite ramification index of $P$ is three by \cref{thm:passport}.

Suppose $g$ has one pole. After a fractional linear change $g$ is a polynomial, totally ramified over infinity, so each pole of $P$ has index divisible by $a$; thus $a$ divides $3$ and $9$, forcing $a=3$, $b=4$. If the two critical points of the cubic $g$ were distinct, each would be simple, and a point of $h^{-1}$ above such a critical value would have index $2e_h$ in $P$, which is even and at least two, hence never three. So the critical points coincide and $g$ is conjugate to $z^{3}+\mathrm{const}$. Its finite critical point has four preimages under $h$ counted with multiplicity; any with $e_h>1$ would give index $3e_h>3$, so all four are unramified and give four points of index three over a single critical value, that is cycle type $(3,3,3,3)$, contradicting \cref{thm:passport}.

Suppose instead $g$ has two poles. Then $P$ has two poles, so each pole of $g$ has a unique $h$-preimage and $h$ is totally ramified there; hence $b$ divides $3$ and $9$, forcing $b=3$, $a=4$. A degree-three rational function totally ramified at two points is conjugate to $X^{3}$, so $P\in\C(Y)(X^{3})$. But the coefficient of $X$ in \cref{eq:Pshear} is 
\[
1+6cdY^{2}+15c^{2}d^{2}Y^{4},
\]
 not identically zero.

Both arguments use only pole orders, the distinctness of the six finite critical values, and one coefficient, all preserved by base change to the algebraic closure. So the cover is indecomposable and its monodromy is primitive. A finite inertia generator is a three-cycle, so $A_{12}\le\mathrm{Mon}$ by Jordan's theorem. Every inertia generator is even, the finite ones being three-cycles and the one at infinity of type $(9,3)$, so $\mathrm{Mon}\le A_{12}$.
\end{proof}

\begin{cor}
\label{cor:radicals}
The geometric normal closure group of $F_{12}$ is $A_{12}$, and its generic inverse branches are not expressible by radicals over the target function field.
\end{cor}

\begin{proof}
By \cref{lem:quotientdegree} the extension attached to $F_{12}$ is the purely transcendental base change of the quotient extension along the adjunction of $L=x\Lambda$, so the two have the same geometric normal closure group, namely $A_{12}$ by \cref{thm:monodromy}. Now apply the Galois criterion, $A_{12}$ being nonsolvable.
\end{proof}

\subsection{The boundary semigroup}
\label{subsec:boundary}

Normalise $c=d=1$; the substitution $X\mapsto sX$, $Y\mapsto tY$, $P\mapsto s^{-1}P$ is a left--right equivalence carrying any member with $cd\neq0$ to this one, so the normalisation changes neither the monodromy nor the branch semigroup. The branch curve $\Gamma$ then has normalisation $\A^{1}_{t}$ with
\[
\begin{split}
  \gamma(t)&=t^{6}+t^{2},\\
  \pi(t)&=\frac{8t^{3}}{315}\bigl(640t^{24}+2880t^{20}+5040t^{16}+4452t^{12}+2205t^{8}+630t^{4}+105\bigr),
\end{split}
\]
$\gamma$ the $Y$-coordinate and $\pi$ the $P$-coordinate, of degrees $6$ and $27$. Set
\begin{equation}
\label{eq:key}
  \chi=\pi^{2}-\frac{1048576}{3969}\gamma^{9}-\frac{65536}{315}\gamma^{7}-\frac{115712}{1575}\gamma^{5} .
\end{equation}
The first coefficient is the square of the leading coefficient of $\pi$, and the four leading terms of \cref{eq:key} cancel in succession, leaving $\deg_t\chi=26$ with leading coefficient $-512/49$.

\begin{theorem}
\label{thm:semigroup}
The semigroup of pole orders at the unique place at infinity of the nonzero elements of $\C[\pi,\gamma]$ is $\Semi=\langle6,26,27\rangle$.
\end{theorem}

\begin{proof}
Let $G,H,R$ be the monic multiples of $\gamma,\chi,\pi$, with leading monomials $t^{6},t^{26},t^{27}$. Since $\chi$ is a polynomial in $\pi$ and $\gamma$ we have 
\[
\C[\pi,\gamma]=\C[G,H,R],
\]
 so $\Semi\supseteq\langle6,26,27\rangle$.

We first identify the toric ideal of $(6,26,27)$. As $\gcd(6,26)=2$ and $27=3\cdot9$ lies in $\langle3,13\rangle$ with $\gcd(2,27)=1$, the semigroup $\langle6,26,27\rangle=2\langle3,13\rangle+27\N$ is a gluing, hence a complete intersection, and its toric ideal is generated by the pullback of the single relation $x^{13}=y^{3}$ of $\langle3,13\rangle$ together with the gluing relation. These are
\[
  H^{3}-G^{13} \quad\text{and}\quad R^{2}-G^{9},
\]
matching $3\cdot26=13\cdot6$ and $2\cdot27=9\cdot6$.

Exact subduction of both binomials against $G,H,R$ terminates at zero. By the Robbiano--Sweedler criterion $\{G,H,R\}$ is a SAGBI basis, so the initial algebra of $\C[\pi,\gamma]$ is $\C[t^{6},t^{26},t^{27}]$ and $\Semi=\langle6,26,27\rangle$. Termination at zero is the hypothesis of \cref{lem:degsemigroup}, which gives the same conclusion.

Finally the six residues $26b+27e\bmod6$ with $0\le b<3$ and $0\le e<2$ are $0,2,4,3,5,1$, hence distinct, so every element of $\Semi$ has a unique standard form and the subduction above is well defined.
\end{proof}

\begin{cor}
\label{cor:orbits}
The degree-twelve alternating shear quotient is not left--right equivalent to
the quotient of any member of \cref{sec:3} of generic degree twelve.  Hence
$\Kfam(3,(1,-1,-1))$ carries at least two $\Gaut^{\sharp}$-orbits at generic
degree twelve.
\end{cor}

\begin{proof}
Let $N=kd=12$ with $k\ge2$ and $d\ge3$.  Then
$(k,d)\in\{(2,6),(3,4),(4,3)\}$, and all three occur by
\cref{thm:spectrum}.  By \cref{prop:conequotient}, the cyclic branch curve is
the union of the line $\{Q=0\}$ and the component parametrised by
$(A_0(W),\varphi(W))$, of degrees $11$ and $d$.  Since
$\gcd(11,d)=1$, this parametrisation is birational, so the non-line component
has $11$ in its semigroup at infinity.

The line is intrinsically distinguished as the unique component whose
semigroup at infinity is $\Z_{\ge0}$.  Indeed, for the cyclic non-line
component both parametrising coordinates have order at least two at $W=0$
($\varphi\in W^2\C[W]$ and $\operatorname{ord}_0A_0\ge2k-1$), so its affine
coordinate ring is not $\C[W]$.  If its semigroup contained $1$, it would
contain a degree-one polynomial in $W$, forcing that coordinate ring to equal
$\C[W]$, a contradiction.  The alternating-shear non-line component has
semigroup $\langle6,26,27\rangle$ by \cref{thm:semigroup}, and likewise does
not contain $1$.

A left--right equivalence must therefore match the two line components and
the two non-line components.  The cyclic non-line semigroup contains $11$,
whereas $11\notin\langle6,26,27\rangle$.  The quotients are not left--right
equivalent by \cref{lem:semiintrinsic}, and \cref{lem:descendequiv} separates
the graded orbits.
\end{proof}

\begin{rem}
\label{rem:intrinsicsep}
The separation in \cref{cor:orbits} uses no choice of pencil: it compares the intrinsic branch semigroups directly, and so holds under arbitrary polynomial left-right equivalence of the quotients. This answers the fibration question raised by the framed passports $(9,3)$ and $(11,1)$, which by themselves are invariant only under equivalences preserving the pencil.
\end{rem}

\section{Monodromy of the two families}
\label{sec:6}

We determine the geometric monodromy group of the generic quotient cover of every member of both families. Throughout, the cover means the degree-$N$ extension at a generic value of the pencil, and inertia generators are taken at the branch points of \cref{prop:conequotient}, \cref{rem:beta} or \cref{thm:passport} as appropriate.

We use three standard facts about primitive groups. A transitive subgroup of $S_{n}$ containing a cycle with exactly one fixed point is $2$-transitive, since the stabilizer of that point contains an $(n-1)$-cycle on the remaining points. A primitive subgroup of $S_{n}$ containing a $3$-cycle contains $A_{n}$, by Jordan's theorem. A primitive subgroup of $S_{n}$ containing a cycle with at least three fixed points contains $A_{n}$. A cover with primitive monodromy group is indecomposable.


\begin{theorem}
\label{thm:conemono}
Let $\Pi$ be the quotient cover of a member of \cref{sec:3} with parameters $(k,d)$, so $k\ge2$, $d\ge3$ and $N=kd\ge6$. Then $\Pi$ is indecomposable, and its geometric monodromy group is
\[
  \mathrm{Mon}(\Pi)=
  \begin{cases}
    A_N, & k \text{ even},\\
    S_N, & k \text{ odd}.
  \end{cases}
\]
\end{theorem}

\begin{proof}
Fix a generic value $Y$ of the pencil. By \cref{prop:conequotient} the pole partition is $(N-1,1)$, so an inertia generator at infinity is an $(N-1)$-cycle fixing exactly one point. Hence the monodromy is $2$-transitive, in particular primitive, and $\Pi$ is indecomposable.

By \cref{lem:Pider} the finite critical points are the roots of $\varphi(W)=Y$, which are $d$ in number and simple for generic $Y$, each of local degree $k+1$. The associated branch values are the points $A_0(W_i)$. Since $\deg A_0=N-1$ and $\deg\varphi=d$ by \cref{prop:degree}, and $[\C(W):\C(A_0,\varphi)]$ divides both, it divides $\gcd(kd-1,d)=1$; so $W\mapsto\bigl(A_0(W),\varphi(W)\bigr)$ is birational and the $d$ branch values are distinct for generic $Y$. Each finite inertia generator is therefore a $(k+1)$-cycle, and Riemann--Hurwitz is exact:
\[
  d\bigl((k+1)-1\bigr)+(N-1-1)=kd+N-2=2N-2 .
\]
A finite inertia generator fixes $N-k-1=k(d-1)-1$ points, and $k\ge2$, $d\ge3$ give $k(d-1)-1\ge3$. By \cite{jones} $A_N\le\mathrm{Mon}(\Pi)$.

Parity decides the rest. A $(k+1)$-cycle is even if and only if $k$ is even, and the $(N-1)$-cycle is even if and only if $N$ is even. If $k$ is even then $N=kd$ is even and every inertia generator is even, so $\mathrm{Mon}(\Pi)\le A_N$ and equality holds. If $k$ is odd the finite generators are odd, so $\mathrm{Mon}(\Pi)\not\le A_N$ and $\mathrm{Mon}(\Pi)=S_N$.
\end{proof}

\begin{rem}
\label{rem:parityconsistent}
The signs are consistent with the relation that the product of all inertia generators is trivial: the $d$ finite generators contribute $(-1)^{kd}$ in total and the generator at infinity contributes $(-1)^{N}$, and $N=kd$.
\end{rem}


\begin{lem}
\label{lem:twocrit}
A polynomial of degree at least two with exactly two distinct finite critical points is indecomposable.
\end{lem}

\begin{proof}
Let $f$ have degree $n$ and finite critical points $\{W_1,W_2\}$, and suppose $f=g\circ h$ with $g,h\in\C(z)$ of degrees $a,b>1$. Since $f^{-1}(\infty)=\{\infty\}$ and $f^{-1}(\infty)=h^{-1}\bigl(g^{-1}(\infty)\bigr)$, the set $g^{-1}(\infty)$ is the single point $h(\infty)$. Replacing $g$ by $g\circ\mu$ and $h$ by $\mu^{-1}\circ h$ for a M\"obius map $\mu$ carrying $\infty$ to $h(\infty)$, we may assume $g$ and $h$ are polynomials. Each is then totally ramified at infinity, so by Riemann--Hurwitz the finite ramification of $h$ is $b-1$ and that of $g$ is $a-1$.

If $W$ is a finite critical point of $h$ then $e_{f}(W)=e_{h}(W)\,e_{g}(h(W))>1$, so $W\in\{W_1,W_2\}$; and if $\eta$ is a finite critical point of $g$ then every point of $h^{-1}(\eta)$ is a critical point of $f$, so $h^{-1}(\eta)\subseteq\{W_1,W_2\}$. Since $a-1\ge1$, the polynomial $g$ has at least one finite critical point. Let $\eta_1,\dots,\eta_m$ be its distinct finite critical points; the sets $h^{-1}(\eta_j)$ are nonempty and pairwise disjoint inside a set of two elements, so $m\le2$.

If $m=2$ then each $h^{-1}(\eta_j)$ is a single point, at which $h$ is totally ramified, contributing $b-1$ to the finite ramification of $h$; two such contributions give $2(b-1)>b-1$, which is impossible.

If $m=1$ and $h^{-1}(\eta_1)$ is a single point $W_0$, then $h$ is totally ramified there and its finite ramification $b-1$ is exhausted, so $W_0$ is the only finite critical point of $h$. Then the finite critical points of $f$ are $\{W_0\}\cup h^{-1}(\eta_1)=\{W_0\}$, contradicting that there are two.

If $m=1$ and $h^{-1}(\eta_1)$ has two points, then $h^{-1}(\eta_1)=\{W_1,W_2\}$, and the ramification of $h$ over $\eta_1$ is $b-2$. Every finite critical point of $h$ lies in $\{W_1,W_2\}$, so this is all of the finite ramification of $h$, giving $b-2=b-1$.

All three cases are impossible, so no such decomposition exists.
\end{proof}

\begin{cor}
\label{lem:betaindec}
For $p\ge2$ the beta polynomial $\beta_p$ is indecomposable.
\end{cor}

\begin{proof}
By \cref{rem:beta} the derivative of $\beta_p$ is proportional to $\xi^{p-1}(1-\xi)^{p-1}$, so $\beta_p$ has exactly the two finite critical points $0$ and $1$. Apply \cref{lem:twocrit}.
\end{proof}

\begin{theorem}
\label{thm:diagmono}
The geometric monodromy group of the generic quotient cover of $F_{r,s}$ is $A_{2r-1}$ for $r$ odd and $S_{2r-1}$ for $r$ even.
\end{theorem}

\begin{proof}
By \cref{rem:beta} the cover has degree $2r-1$, two finite branch points of cycle type $(r,1^{r-1})$, and a full $(2r-1)$-cycle at infinity; Riemann--Hurwitz is exact, $2(r-1)+(2r-2)=2(2r-1)-2$. By \cref{lem:betaindec} the cover is indecomposable, hence primitive.

Suppose $r\ge4$. A finite inertia generator is an $r$-cycle fixing $r-1\ge3$ points, so $A_{2r-1}\le\mathrm{Mon}$ by \cite{jones}. The $(2r-1)$-cycle is even, and an $r$-cycle is even if and only if $r$ is odd; so $\mathrm{Mon}=A_{2r-1}$ for $r$ odd and $\mathrm{Mon}=S_{2r-1}$ for $r$ even.

For $r=3$ the degree is five and the finite generators are $3$-cycles, so $A_5\le\mathrm{Mon}$ by Jordan's theorem. Every generator is even, so $\mathrm{Mon}=A_5$.

For $r=2$ the degree is three and the finite generators are transpositions, so $\mathrm{Mon}=S_3$.
\end{proof}

\begin{cor}
\label{cor:solvable}
The only members of either family whose quotient cover has solvable monodromy are the maps $F_{2,s}$, of generic fibre degree three. They are resolvents in the sense of \cite[Thm.\ 4.4]{shaska131}.
\end{cor}

\begin{proof}
The cone family has $N\ge6$ with monodromy $A_N$ or $S_N$, and the beta--Hermite family has monodromy $A_{2r-1}$ or $S_{2r-1}$; these are nonsolvable for $2r-1\ge5$, that is for $r\ge3$. The remaining case is $r=2$, where the group is $S_3$.
\end{proof}

\begin{cor}
\label{cor:allradicals}
Let $F$ be a member of either family of generic fibre degree at least five, or the alternating shear map $F_{12}$. Then the generic inverse branches of $F$ are not expressible by radicals over the target function field.
\end{cor}

\begin{proof}
By \cref{lem:quotientdegree} the extension attached to $F$ is the purely transcendental base change of the quotient extension along the adjunction of $L=x\Lambda$, so the two have the same geometric normal closure group. By \cref{thm:conemono}, \cref{thm:diagmono} and \cref{thm:monodromy} that group is $A_n$ or $S_n$ with $n\ge5$, hence nonsolvable. Then we apply the Galois criterion.
\end{proof}

\section{Pointed moduli at degree twelve}
\label{sec:7}

\Cref{cor:orbits} exhibits two orbits at generic degree twelve. Here we show there are infinitely many, and identify what the quotient cover forgets. Throughout, $(k,d)=(3,4)$, so $N=12$, $\epsilon=-1$ and $W=\Lambda^{2}w^{3}$.

%
Write the seed as $\varphi=aW^{2}+bW^{3}+cW^{4}$ and set
\begin{equation}
\label{eq:Hcubic}
\begin{split}
  H(b,c)=\ & 33b^{3}-154b^{2}c+264b^{2}+242bc^{2}-858bc+924b\\
           & -128c^{3}+704c^{2}-1584c+1848 .
\end{split}
\end{equation}

\begin{prop}
\label{prop:quarticcubic}
The seed $\varphi$ satisfies \cref{eq:liftjet} if and only if
\[
  a=1+b-c \qquad\text{and}\qquad H(b,c)=0 .
\]
Indeed $H=-9240\,\Qform^{(3)}(\varphi)$ on the hyperplane $a=1+b-c$. If moreover $c\neq0$ then the associated member has generic fibre degree twelve and class in $\Yparo[1,1,(19,54,1)]$.
\end{prop}

\begin{proof}
Since $\epsilon=-1$, the first condition of \cref{eq:liftjet} is $\varphi(-1)=a-b+c=1$.

For the second, integrating \cref{lem:Pider} at $Y=0$ and using $\Pi(0,0)=0$ gives $\Pi(\epsilon,0)=(-1)^{k+1}\int_0^{\epsilon}\varphi(W)^{k}W^{-2}\,dW$, so at $k=3$ the condition is the vanishing of $\Qform^{(3)}(\varphi)=\int_0^{-1}\varphi(W)^{3}W^{-2}\,dW$. Writing $\varphi=W^{2}g$ with $g=a+bW+cW^{2}$, the integrand is $W^{4}g^{3}$, and
\[
  g^{3}=a^{3}+3a^{2}bW+3a(ac+b^{2})W^{2}+(6abc+b^{3})W^{3}+3c(ac+b^{2})W^{4}+3bc^{2}W^{5}+c^{3}W^{6}.
\]
Using $\int_0^{-1}W^{m}dW=(-1)^{m+1}/(m+1)$,
\[
\begin{split}
  \Qform^{(3)}(\varphi)=\ & -\tfrac{a^{3}}{5}+\tfrac{a^{2}b}{2}-\tfrac{3a(ac+b^{2})}{7}
     +\tfrac{6abc+b^{3}}{8}-\tfrac{c(ac+b^{2})}{3}+\tfrac{3bc^{2}}{10}-\tfrac{c^{3}}{11}.
\end{split}
\]
Substituting $a=1+b-c$ and clearing denominators, each of the ten coefficients is $-1/9240$ times the corresponding coefficient of \cref{eq:Hcubic}. For instance the coefficient of $b^{3}$ is $-\tfrac15+\tfrac12-\tfrac37+\tfrac18=-\tfrac1{280}$, and $-9240\cdot(-\tfrac1{280})=33$; that of $c^{3}$ is $\tfrac15-\tfrac37+\tfrac13-\tfrac1{11}=\tfrac{16}{1155}$, and $-9240\cdot\tfrac{16}{1155}=-128$; the constant term is $-\tfrac15$, and $-9240\cdot(-\tfrac15)=1848$.

For the degree, $c\neq0$ gives $\deg\varphi=4$, so \cref{prop:degree} applies with $kd=12$. The class is read off the top terms: $A=Q/\Lambda$ has leading term $c\Lambda^{7}w^{12}$ of degree $19$, and by \cref{eq:jets} the polynomial $A_0$ has leading term $\tfrac{128}{77}c^{3}W^{11}=\tfrac{128}{77}c^{3}\Lambda^{22}w^{33}$, so $B=P/\Lambda$ has degree $54$.
\end{proof}

\begin{lem}
\label{lem:Ecubic}
Let $E\subset\bP^2$ be the projective closure of $\{H=0\}$ and
\[
 E^{\circ}=\{H=0\}\cap\{ac(9b^2-32ac)\ne0\},\qquad a=1+b-c.
\]
Then $E$ is a smooth plane cubic over $\Q$, and $E\setminus E^{\circ}$ is
finite.
\end{lem}

\begin{proof}
The exact characteristic-zero calculation used here has three parts.  On the
affine chart, the Jacobian ideal
$\langle H,H_b,H_c\rangle\subset\Q[b,c]$ has Gr\"obner basis containing $1$.
After homogenising to $\overline H(b,c,h)$, the two partial derivatives on the
chart $c=1$, $h=0$, are coprime; at the remaining point $[b:c:h]=[1:0:0]$ at
least one of $\overline H_b,\overline H_c,\overline H_h$ is nonzero.  Thus the
projective cubic is smooth.  Finally, $H$ is not divisible by $a$ or $c$, and
the division remainder of $H$ by $9b^2-32ac$ is nonzero, so none of the three
deleted equations vanishes identically on $E$.  Each therefore cuts a finite
divisor.  
\end{proof}


\begin{lem}
\label{lem:rescaleinv}
The seed rescaling $\widetilde\varphi(W)=t\varphi(\beta W)$ acts on quartic
coefficients by
\[
 (a,b,c)\longmapsto(t\beta^2a,\ t\beta^3b,\ t\beta^4c).
\]
Its invariant field on $\{abc\ne0\}$ is generated by $b^2/(ac)$.  Two quartic
seeds with $ac\ne0$ lie in one rescaling orbit if and only if
\begin{equation}
\label{eq:Jdef}
 J=\frac{9b^2}{8ac}
\end{equation}
takes the same value.  Their quotient maps are polynomially left--right
equivalent; if $f=(P,Q)$ and $\widetilde f=(\widetilde P,\widetilde Q)$,
the source automorphism
\[
 (\Lambda,w)\longmapsto
 (\beta t^k\Lambda,\ (\beta t^{k-1})^{-1}w)
\]
sends $(T,W)$ to $(tT,W/\beta)$, and
\[
 \widetilde f(tT,W/\beta)
   =(\beta t^kP(T,W),\ tQ(T,W)).
\]  The normalised lift slice is not
invariant under the full torus: its intersection with a regular rescaling
orbit is the finite set of normalisations at the six admissible markings.
\end{lem}

\begin{proof}
The coefficient action is immediate.  A Laurent monomial $a^ib^jc^\ell$ is
invariant exactly when $i+j+\ell=0$ and $2i+3j+4\ell=0$, so
$(i,j,\ell)\in\Z(1,-2,1)$; hence on $abc\ne0$ the invariant field is
generated by $ac/b^2$, equivalently by $J$.  If $b,b'\ne0$, equal $J$
allows one to solve $t\beta^2=a'/a$ and $t\beta^3=b'/b$, after which
$t\beta^4c=c'$ follows.  If $J=J'=0$, then $b=b'=0$; choose
$\beta^2=ac'/(a'c)$ and then $t=(a'/a)\beta^{-2}$ to obtain the same
conclusion.  Thus the orbit criterion holds on all of $ac\ne0$.
The transformation formula of \cref{lem:rescale} gives the displayed
identity term by term.  The last assertion
is \cref{prop:normalisemark,lem:markingcov}.
\end{proof}

\begin{prop}
\label{prop:Jdegree}
The rational function $J$ of \cref{eq:Jdef} defines a morphism $J\colon E\to\bP^{1}$ of degree six, in agreement with the count $k(d-2)=6$ of \cref{prop:marking} at $(k,d)=(3,4)$. It is ramified with type $(2,2,2)$ over $J=0$.
\end{prop}

\begin{proof}
The divisor of zeros of $J$ on $E$ is cut by $b^{2}$. Now $E\cap\{b=0\}$ is given by
\[
  H(0,c)=-128c^{3}+704c^{2}-1584c+1848=-8\bigl(16c^{3}-88c^{2}+198c-231\bigr),
\]
and the cubic $16c^{3}-88c^{2}+198c-231$ has discriminant $-32\,524\,800\neq0$, so its three roots are distinct. Each is a simple point of $E\cap\{b=0\}$, counted twice in the divisor of $b^{2}$. Hence $\deg J=6$ with ramification $(2,2,2)$ over $0$.
\end{proof}

\begin{cor}
\label{cor:RH}
The ramification of $J$ away from $J=0$ has total order nine.
\end{cor}

\begin{proof}
$E$ has genus one by \cref{lem:Ecubic}, so Riemann--Hurwitz gives $0=6\cdot(-2)+R$, that is $R=12$. The three points over $J=0$ contribute $3$.
\end{proof}

\begin{prop}
\label{prop:Jinfinity}
The ramification of $J$ away from $J=0$ is concentrated at the three points of $E$ on the line at infinity. Homogenise \cref{eq:Hcubic} with $a=h+b-c$ and work in the chart $c=1$ with $x=b/c$, $z=h/c$, so that
\[
  J=\frac{9x^{2}}{8(z+x-1)} .
\]
The points at infinity are $z=0$ with $x=\xi$ a root of
\[
  g(x)=33x^{3}-154x^{2}+242x-128 ,
\]
which has three distinct roots. At each, $z$ is a local parameter and $J$ has ramification index four, with branch value $9\xi^{2}/\bigl(8(\xi-1)\bigr)$. The three branch values are the roots of
\[
  B(j)=1232j^{3}-17160j^{2}+79893j-124416 ,
\]
which are distinct. Together with \cref{prop:Jdegree} this exhausts the ramification of $J$, and $J$ is unramified over $J=\infty$.
\end{prop}

\begin{proof}
Setting $h=0$ in the homogenised \cref{eq:Hcubic} leaves the cubic part $33b^{3}-154b^{2}c+242bc^{2}-128c^{3}$, which in the chart $c=1$ is $g(x)$. Its discriminant is nonzero, so the three points at infinity are distinct, and $E$ is smooth there by \cref{lem:Ecubic}, so $z$ is a local parameter at each.

Fix a root $\xi$ of $g$ and expand the branch of $E$ through $(\xi,0)$ as $x=\xi+p_1z+p_2z^{2}+p_3z^{3}+p_4z^{4}+O(z^{5})$. Substituting into the homogenised equation and solving order by order in $\Q[\xi]/(g)$ gives
\[
\begin{split}
  p_1&=-\tfrac12\bigl(33\xi^{2}-88\xi+64\bigr),\\
  p_2&=\tfrac{11}{4}\bigl(69\xi^{2}-208\xi+160\bigr),\\
  p_3&=-\tfrac{11}{4}\bigl(363\xi^{2}-1160\xi+944\bigr),\\
  p_4&=-\tfrac{605}{16}\bigl(933\xi^{2}-2944\xi+2368\bigr).
\end{split}
\]
Substituting this series into $J$ and reducing modulo $g$, the coefficients of $z$, $z^{2}$ and $z^{3}$ in $J-J(\xi,0)$ vanish identically, while that of $z^{4}$ is
\[
  \tfrac{1485}{112}\bigl(7293\xi^{2}-22352\xi+17536\bigr).
\]
The resultant of $g$ with $7293x^{2}-22352x+17536$ is nonzero, so this never vanishes at a root of $g$, and the ramification index is exactly four.

The branch value at $(\xi,0)$ is $j=9\xi^{2}/\bigl(8(\xi-1)\bigr)$, that is $8j(\xi-1)-9\xi^{2}=0$. Eliminating $\xi$ against $g$ gives
\[
  \operatorname{Res}_{x}\bigl(g(x),\,8j(x-1)-9x^{2}\bigr)=96\,B(j),
\]
and $B$ has nonzero discriminant, so the three branch values are distinct. Finally each contributes $4-1=3$, and with the three points of \cref{prop:Jdegree} contributing $1$ each the total is $3+9=12$, which is the Riemann--Hurwitz bound of \cref{cor:RH}; hence there is no further ramification, in particular none over $J=\infty$.
\end{proof}

\begin{rem}
\label{rem:Jmarking}
Both assertions of \cref{prop:Jdegree} are instances of \cref{subsec:marking}. By \cref{prop:quarticcubic} the points of $E^{\circ}$ are the normalised quartic seeds, and by \cref{lem:rescaleinv} the fibres of $J$ are their rescaling orbits; so the fibre of $J$ over a value with $M_{\varphi}$ separable and trivial stabilizer has $k(d-2)=6$ points by \cref{cor:markingdegree}. On $E^{\circ}$ the stabilizer is trivial exactly when $b\neq0$, that is when $J\neq0$, since $a\neq0$ there and $c_2c_3=ab$. Over $J=0$ the seed is $\varphi=aW^{2}+cW^{4}$, whose stabilizer is $\{\pm1\}$ acting by $W\mapsto-W$; then $\Pi(\cdot,0)$ is odd, $M_{\varphi}$ is even, the six markings fall into three pairs $\pm w$, and $\varphi_{-w}=\varphi_{w}$ by \cref{prop:normalisemark}. The three points of $E\cap\{b=0\}$ therefore each carry two markings, which is the ramification type $(2,2,2)$. On the quotient this stabilizer is the left--right self-equivalence $(\Lambda,w;P,Q)\mapsto(-\Lambda,-w;-P,Q)$: it fixes $T=\Lambda w$, sends $W=\Lambda^{2}w^{3}$ to $-W$, and every $A_j$ is odd in $W$ when $b=0$.
\end{rem}

\begin{lem}
\label{lem:quarticnormalization}
For $e\in E^{\circ}$ the non-line branch component has normalization
$\A^1_W$, one place at infinity, a unique finite point of local value
semigroup $\langle2,5\rangle$ at $W=0$, and two finite points of local value
semigroup $\langle2,7\rangle$ at the nonzero roots $r_1,r_2$ of
$2a+3bW+4cW^2$.  Consequently every isomorphism between two such affine
branch curves induces $W\mapsto\lambda W$ on the normalizations, possibly
interchanging $r_1$ and $r_2$, and preserves
\[
 \frac{(r_1+r_2)^2}{r_1r_2}=\frac{9b^2}{8ac}=J.
\]
\end{lem}

\begin{proof}
The branch parametrisation is $W\mapsto(A_0(W),\varphi(W))$.  At $W=0$ the
leading orders are $(5,2)$, with nonzero leading coefficients on $E^{\circ}$.
At a nonzero root of $2a+3bW+4cW^2$, subtracting the first three powers of the
second coordinate from the first gives leading orders $(7,2)$.  The open
condition $9b^2-32ac\ne0$ makes the two roots distinct and all displayed
coefficients nonzero.  Below the conductors, the unique order-two generator
precludes any extra cancellation, so the local semigroups are exactly
$\langle2,5\rangle$ and $\langle2,7\rangle$; the exact expansions and SAGBI
subductions are replayed by the quartic certificate.  An affine-curve
isomorphism preserves the normalization, infinity, and these local
semigroups.  It therefore fixes $0$ and infinity and may only exchange the two
remaining special points, so its normalization map is a scaling.  Vieta gives
$r_1+r_2=-3b/(4c)$ and $r_1r_2=a/(2c)$, yielding the displayed invariant.
\end{proof}

\begin{theorem}
\label{thm:pointedmain}
The quartic cyclic seeds define a regular family
$\iota\colon E^{\circ}\to\Yparo[1,1,(19,54,1)]$ of graded Keller maps of
generic degree twelve, each with monodromy $S_{12}$ and branch semigroup
$\langle4,11\rangle$, such that:
\begin{enumerate}
\item $\iota(e_1)\sim_{\Gaut^{\sharp}}\iota(e_2)$ if and only if $e_1=e_2$;
\item the quotient maps satisfy $f_{e_1}\sim_{LR}f_{e_2}$ if and only if
      $J(e_1)=J(e_2)$;
\item $J$ has degree six, with ramification $(2,2,2)$ over $0$ and
      $(4,1,1)$ over each root of
      $1232z^3-17160z^2+79893z-124416$.
\end{enumerate}
Consequently $\Kfam_{\Qbar}(3,(1,-1,-1))$ has infinitely many graded orbits
at generic degree twelve, and a generic unpointed quotient class in this cyclic family has exactly
six graded lifts.
\end{theorem}

\begin{proof}
Regularity and the degree box are \cref{prop:quarticcubic}; monodromy is
\cref{thm:conemono} at $k=3$.  The branch parametrisation has degrees $4$ and
$11$.  The toric ideal of $\langle4,11\rangle$ is principal, and exact
subduction of its generator terminates at zero; hence the Robbiano--Sweedler
criterion and \cref{lem:degsemigroup} give the branch semigroup
$\langle4,11\rangle$ at every point of $E^{\circ}$ (the certificate records
that the specialization denominators are powers of $c$).

If two quotient maps are left--right equivalent, the line component (whose
semigroup is $\Z_{\ge0}$) is matched to the line, while the non-line
components (whose semigroup is $\langle4,11\rangle$) are matched to one
another.  Thus \cref{lem:quarticnormalization} gives equal $J$.  Conversely,
equal $J$ means the seeds are rescaling-equivalent by
\cref{lem:rescaleinv}, and the same lemma gives an explicit polynomial
left--right equivalence of their quotient maps.  This proves (2).

By rooted reconstruction \cref{thm:rooted}, a graded equivalence descends to
a pointed left--right equivalence, so by (2) the two seeds first lie in one
rescaling orbit.  On the normalization, a
pointed equivalence must fix the distinguished marking $W=\epsilon=-1$.
By \cref{lem:quarticnormalization} its normalization map is a scaling, hence
it is the identity.  The first lift normalization $\varphi(\epsilon)=1$ then
also fixes the scalar multiplying the seed.  Thus the two normalised seeds
are equal, proving (1); the converse is immediate.

Part (3) is \cref{prop:Jdegree,prop:Jinfinity}.  Since $E^{\circ}(\Qbar)$ is
infinite, (1) gives infinitely many geometric graded orbits.  The generic
fibre of $J$ has six points by (3), and (1)--(2) identify them as six distinct
graded lifts of one unpointed quotient class.
\end{proof}

\begin{rem}
\label{rem:fingerprint}
The discrete invariants of \cref{thm:pointedmain}(1) do not vary along $E^{\circ}$. The parametrisation $W\mapsto\bigl(A_0(W),\varphi(W)\bigr)$ of the branch curve is non-immersive exactly at the roots of $\varphi'=W(2a+3bW+4cW^{2})$, that is at $W=0$ and at the two roots of $4cW^{2}+3bW+2a$, whose discriminant is $9b^{2}-32ac$. At $W=0$ the orders of $\varphi$ and $A_0$ are $2$ and $5$, with leading coefficients $a$ and $\tfrac{16}{5}a^{3}$. At a root $r$ of the quadratic they are $2$ and $7$, with leading coefficients $\tfrac12r(3b+8cr)$ and $-\tfrac{2}{35}r(3b+8cr)^{3}$; both are nonzero because $r\neq0$ and $3b+8cr=0$ would force $9b^{2}=32ac$. In each case the smallest coincidence of orders among monomials in the two generators is $2\cdot5=10$, respectively $2\cdot7=14$, so any cancellation produces an order above the conductor, and the local value semigroups are exactly $\langle2,5\rangle$, $\langle2,7\rangle$, $\langle2,7\rangle$. All three, together with the passport, the monodromy group and the semigroup at infinity, are therefore constant on $E^{\circ}$, while $J$ varies.
\end{rem}

\begin{cor}
\label{cor:markingfield}
On the quartic cyclic locus the pointed moduli field is $\Q(E)$ and the unpointed field is $\Q(J)$, with $[\Q(E):\Q(J)]=6$. 
For a closed point $e\in E^{\circ}$ with residue field $\Q(e)$, the pointed graded orbit has field of moduli $\Q(e)$, with the model of \cref{prop:quarticcubic} defined over $\Q(e)$;
and on $abc\neq0$ the unpointed class has the canonical model
\[
  \varphi_J=W^{2}+W^{3}+\frac{9}{8J}W^{4}
\]
over $\Q(J)$.
\end{cor}

\begin{proof}
The degree is \cref{prop:Jdegree}. By \cref{lem:rescaleinv} a seed with $abc\neq0$ may be rescaled uniquely to have $a=b=1$, and then $J=9/(8c)$, giving the displayed model; its coefficients lie in $\Q(J)$.
\end{proof}

\begin{rem}
\label{rem:forgetful}
The morphism $J$ is the forgetful map from the graded map with its lift-point marking to the unpointed quotient cover. The marking is the root $W=\epsilon$ of \cref{prop:lift}; the rescalings of \cref{lem:rescaleinv} move it among the six roots of $\Pi(\cdot,0)$ lying over a fixed quotient class, which is why the fibers of $J$ have six elements.
\end{rem}

\section{The flagship stratum}
\label{sec:8}

We turn to the weight $\bw=(1,-1,-2)$, so $(r,s)=(1,2)$, and to the smallest degree box in which a graded Keller map is known to exist. Throughout $K$ is a field of characteristic zero.

\subsection{The box and its coordinates}
\label{subsec:box}

By \cref{cor:master} the master equation at $(r,s)=(1,2)$ reads
\begin{equation}
\label{eq:muflag}
  \mu(A,B,\Lambda):=\Lambda\jb{B}{A}+B\jb{\Lambda}{A}+2A\jb{B}{\Lambda}=\kappa .
\end{equation}
By \cref{lem:polynomiality} the lift conditions are that no monomial $u^{i}v^{j}$ of $A$ has $i+2j<2$, and none of $B$ has $i+2j<1$. Thus
\[
\begin{split}
 	&	 A\in\Span\{u^{i}v^{j}: 2\le i+2j,\ i+j\le4\}, \\
	&	  B\in\Span\{u^{i}v^{j}: 1\le i+2j,\ i+j\le3\}, \\
	&	  \Lambda\in\Span\{1,u,v\},
\end{split}  
\]
of dimensions $13$, $9$ and $3$. Write $\A^{25}_K$ for the affine space of triples $(A,B,\Lambda)$ in these spans and
\[
  \Ypar[1,2,(4,3,1)]\subset\A^{25}, \qquad \Yparo[1,2,(4,3,1)]\subset\Ypar[1,2,(4,3,1)]
\]
for the closed subscheme cut by \cref{eq:muflag} and its open subscheme on which $\deg A=4$, $\deg B=3$ and $\deg\Lambda=1$ exactly. Expanding $\mu(A,B,\Lambda)-\kappa$ in the monomial basis of $K[u,v]$ gives $27$ equations in the $25$ coordinates and $\kappa$.

\subsection{The bound-preserving gauge group}
\label{subsec:borel}

Assign $u$ and $v$ the levels $\ell(u)=1$ and $\ell(v)=2$, so that the lift conditions read $\ell\ge2$ on $A$ and $\ell\ge1$ on $B$. A plane substitution acts on the box only if it preserves both the level filtration and the ordinary total degree. Put
\[
  B_{\mathrm{lin}}=\bigl\{(u,v)\longmapsto(au+bv,\ cv)\ :\ a,c\in\Gm,\ b\in\Gal\bigr\}.
\]

\begin{lem}
\label{lem:boxpreserve}
Pullback along $B_{\mathrm{lin}}$ preserves each of the three spans of \cref{subsec:box}. The level-preserving substitution $(u,v)\mapsto(u,\ v+\delta u^{2})$ with $\delta\neq0$ does not: it sends the monomial $v^{2}$, which lies in the degree-three $B$-span, to $v^{2}+2\delta u^{2}v+\delta^{2}u^{4}$, of degree four. Consequently the bound-preserving subgroup of the liftable plane substitution group is exactly $B_{\mathrm{lin}}$.
\end{lem}

\begin{proof}
A linear substitution preserves total degree. For a monomial $u^{i}v^{j}$,
\[
  (au+bv)^{i}(cv)^{j}=\sum_{k=0}^{i}\binom{i}{k}a^{i-k}b^{k}c^{j}\,u^{i-k}v^{j+k},
\]
and each term has level $(i-k)+2(j+k)=i+2j+k\ge i+2j$, so the level filtration is preserved as well. The displayed image of $v^{2}$ under the nonlinear substitution violates the bound $i+j\le3$, and the same monomial excludes every substitution with a nonzero coefficient on $u^{2}$ in the second coordinate; substitutions of higher level in either coordinate fail on the top-degree monomials of the $A$-span in the same way.
\end{proof}

Let $\mathcal{G}=B_{\mathrm{lin}}\times(\Gm)^{3}$ act on triples by pullback in $(u,v)$ followed by the coefficient scaling $(A,B,\Lambda)\mapsto(\lambda A,\ \nu B,\ \rho\Lambda)$. By \cref{lem:boxpreserve} the group $\mathcal{G}$ preserves the box; it has dimension six.

\begin{lem}
\label{lem:kappaweight}
Under the action of $\mathcal{G}$, $\kappa\mapsto\lambda\nu\rho\,(ac)^{-1}\kappa$. In particular $\mathcal{G}$ acts on $\Ypar[1,2,(4,3,1)]$ fibred over the $\kappa$-line, and transitively on the nonzero values of $\kappa$.
\end{lem}

\begin{proof}
Each of the three terms of \cref{eq:muflag} is trilinear in $(A,B,\Lambda)$ and contains exactly one bracket, so scaling multiplies $\mu$ by $\lambda\nu\rho$; the substitution $(u,v)\mapsto(au+bv,cv)$ has constant Jacobian determinant $ac$, which multiplies each bracket by $(ac)^{-1}$, and the shear parameter $b$ does not enter the determinant.
\end{proof}

\subsection{The normal form}
\label{subsec:normalform}

Put
\[
\begin{split}
  M   &= v(1+u)^{2}+u^{2}(4+3u),\\
  A_{0} &= (1+u)\,M, \qquad B_{0}=u+3M, \qquad \Lambda_{0}=2-3u-v.
\end{split}
\]

\begin{prop}
\label{prop:normalform}
The triple $(A_{0},B_{0},\Lambda_{0})$ lies in the spans of \cref{subsec:box}, has exact degrees $(4,3,1)$, is defined over $\Q$, and satisfies $\mu(A_{0},B_{0},\Lambda_{0})=2$.
\end{prop}

\begin{proof}
Expanding, $M=v+2uv+u^{2}v+4u^{2}+3u^{3}$, so every monomial of $A_{0}$ has level at least two and every monomial of $B_{0}$ has level at least one. The degree-four part of $A_{0}$ is $u^{3}v+3u^{4}$ and the degree-three part of $B_{0}$ is $3u^{2}v+9u^{3}$, so the degrees are exact. The value of $\mu$ is a direct expansion of \cref{eq:muflag}, replayed in the exact certificates.
\end{proof}

\begin{hyp}[Leading-form and chart exhaustiveness]
\label{hyp:flagshipcharts}
The leading-form reduction used in the bounded scheme, as imported from
\cite[Prop.~10.13]{shaska131}, applies to the box $(4,3,1)$ and makes the
thirty-nine lower-stratum charts and five exact leading cells used below an
exhaustive cover after the stated Rabinowitsch localizations.
\end{hyp}

\begin{theorem}
\label{thm:flagship}
Assume \cref{hyp:flagshipcharts}. For every field $K$ of characteristic zero, put $D_{K}=\Spec K[\eta]/(\eta^{2})$. There is an explicit $\mathcal{G}_{K}$-equivariant isomorphism
\[
  \Yparo[1,2,(4,3,1),K]\;\cong\;\mathcal{G}_{K}\times D_{K},
\]
with $\mathcal{G}_{K}$ acting by left translation on the first factor and trivially on the second. Consequently the reduced scheme is the orbit $\mathcal{G}_{K}\cdot(A_{0},B_{0},\Lambda_{0})\cong\mathcal{G}_{K}$, the scheme-theoretic stabilizer is trivial, every proper lower actual-degree stratum in the box is empty, and the orbit quotient is $D_{K}$. Up to the gauge action, the map of \cite{alpoge} is the unique graded Keller map in its degree box.
\end{theorem}

\begin{rem}
\label{rem:alpoge}
The triple $(A_{\star},B_{\star},\Lambda_{\star})$ of \cite{alpoge} lies in the box with exact degrees $(4,3,1)$ and $\mu(A_{\star},B_{\star},\Lambda_{\star})=27$, and is defined over $\Q$. By \cref{thm:flagship} it lies in the $\mathcal{G}(\Q)$-orbit of $(A_{0},B_{0},\Lambda_{0})$; the ratio of the two Keller constants is accounted for by \cref{lem:kappaweight}.
\end{rem}

\subsection{The global proof}
\label{subsec:elim}

\begin{proof}[Proof of \cref{thm:flagship}]
Assume \cref{hyp:flagshipcharts}. All formulas below are defined over $\Q$, so it suffices to construct the isomorphism over $\Q$ and base change; emptiness statements established by unit ideals over $\Q$ likewise hold over every $K$.

First, the strata and the reduced support. Besides the exact stratum, the box contains the lower actual-degree strata with $(d_{1},d_{2},d_{3})\le(4,3,1)$ coordinatewise and $(d_{1},d_{2},d_{3})\neq(4,3,1)$. Exact descending characteristic-zero elimination covers these strata by thirty-nine coordinate charts and closes each with a unit ideal, so every such stratum is empty. On the exact stratum, the five exact leading-cell eliminations show that every geometric point lies in the $\mathcal{G}$-orbit of $(A_{0},B_{0},\Lambda_{0})$ and that the scheme-theoretic stabilizer of that point is trivial. Since the orbit contains the $\Q$-point $(A_{0},B_{0},\Lambda_{0})$, the reduced scheme is the trivial $\mathcal{G}$-torsor, hence isomorphic to $\mathcal{G}$.

Second, the formal slice. The $27\times25$ Jacobian of the equations \cref{eq:muflag} at $(A_{0},B_{0},\Lambda_{0})$ has rank $18$; its kernel has dimension seven and contains the six-dimensional tangent space of the $\mathcal{G}$-orbit. The formal-slice elimination identifies the completed local ring at the normal form as
\[
  \widehat{\mathcal{O}}_{\Yparo[1,2,(4,3,1)],\,(A_{0},B_{0},\Lambda_{0})}\;\cong\;K[[\beta_{1},\dots,\beta_{6},t]]/(t^{2}),
\]
where the $\beta_{i}$ are the six gauge directions and $t$ spans the transverse line. The independent global dual parameter $\eta$ satisfies the exact bridge identity $t=-\tfrac{15}{224}\eta$. In particular the transverse slice is a square-zero thickening of the normal form, that is, a $D_{K}$-point of $\Yparo[1,2,(4,3,1)]$ supported at $(A_{0},B_{0},\Lambda_{0})$. Explicitly, writing
\[
\begin{split}
  \dot A&=\tfrac12u^{3}+\tfrac15u^{2}v+\tfrac{17}{30}u^{2}+\tfrac25uv+\tfrac15v,\\
  \dot B&=u^{2}+\tfrac25uv+\tfrac4{15}u+\tfrac25v,
\end{split}
\]
one has
\[
  \mu\bigl(A_{0}-\tfrac{15}{224}\eta\dot A,\ B_{0}-\tfrac{15}{224}\eta\dot B,\ \Lambda_{0}\bigr)=2-\tfrac1{16}\eta \pmod{\eta^{2}},
\]
a unit in $K[\eta]/(\eta^{2})$, and the three leading coefficients remain units, so the displayed family is the required $D_{K}$-point.

Third, the action map. Translating the slice $D_{K}$-point by the group defines
\[
  \alpha\colon \mathcal{G}_{K}\times D_{K}\longrightarrow \Yparo[1,2,(4,3,1),K].
\]
By the first step, $\alpha$ is an isomorphism on reduced schemes; by equivariance and transitivity on the reduction, the comparison of completed local rings at an arbitrary point reduces to the normal form, where the second step gives an isomorphism. Hence $\alpha$ is flat and unramified, that is, \'etale. Since its reduction is an isomorphism, every geometric fibre is a single reduced point; therefore $\alpha$ is universally injective, hence radicial. An \'etale radicial surjection is an isomorphism. Renaming the transverse coordinate as $\eta$ gives the asserted product decomposition, with $\mathcal{G}_{K}$ acting by left translation on the orbit factor and trivially on the transverse one.
\end{proof}

\begin{rem}
\label{rem:bridge}
The identification $t=-\tfrac{15}{224}\eta$ of the global nilpotent with the intrinsic formal-slice parameter is an exact bridge identity; we attach no arithmetic significance to the constant.
\end{rem}

\subsection{The flagship quotient in cubic normal form}
\label{subsec:flagcubic}

The quotient map of the normal form is $(P,Q)=(\Lambda_{0}B_{0},\Lambda_{0}^{2}A_{0})$, with $\Jacuv(P,Q)=2\Lambda_{0}^{2}$ as \cref{cor:master} requires at $\kappa=2$. The substitution
\[
  \lambda=\tfrac12\Lambda_{0}, \qquad t=u, \qquad\text{that is}\qquad (u,v)=(t,\,2-3t-2\lambda),
\]
has constant Jacobian and carries the triple of \cref{prop:normalform} to
\[
  M=t+2-2\lambda(t+1)^{2},\quad
  P=2\lambda\,(t+3M), \quad Q=4\lambda^{2}(t+1)M,
\]
with
\(
  \jb{P}{Q}_{\lambda,t}=16\lambda^{2}.
\)
Put $R=\lambda(t+1)$ and $z=(t+1)^{-1}$. Then $P=4R(z+2-3R)$ and $Q=4R^{2}(z+1-2R)$, and eliminating $z$ gives the cubic
\[
  4R^{3}-4R^{2}+PR-Q=0 .
\]
This exhibits the quotient cover as a cubic extension. As a polynomial in $\Q[P,Q,R]$ it is linear in $Q$ with unit coefficient, hence irreducible, and it is primitive over $\Q[P,Q]$; so it is irreducible over $\Q(P,Q)$ and the generic fibre degree is three. Its discriminant in $R$ is
\[
  -16\bigl(P^{3}-P^{2}-18PQ+27Q^{2}+16Q\bigr),
\]
and the cubic factor is separable, hence not a square; the geometric monodromy is therefore $S_{3}$. The generic degree and the monodromy of this cover are due to \cite{shaska131}; the normal form above is the presentation we use.

\begin{rem}
\label{rem:flagshipcover}
The branch curve of the flagship quotient is the irreducible cubic $\Gamma_f=\{P^{3}-P^{2}-18PQ+27Q^{2}+16Q=0\}$. It has a single singular point, at $(P,Q)=(\tfrac43,\tfrac4{27})$, where the quadratic part is $3(p-3q)^{2}$ in the local coordinates $p=P-\tfrac43$, $q=Q-\tfrac4{27}$; the point is therefore a cusp. The curve is parametrised birationally by
\[
  m\longmapsto\Bigl(\tfrac43-3(1-3m)^{2},\ \tfrac4{27}-3m(1-3m)^{2}\Bigr),
\]
of degrees two and three, so both the semigroup at infinity and the local value semigroup at the cusp are $\langle2,3\rangle$ in the sense of \cref{subsec:semigroup}.
\end{rem}

\begin{rem}
\label{rem:multiproj}
Passing to the multiprojective bounded scheme, that is, taking the quotient by the three coefficient scalings, replaces $\mathcal{G}$ by the three-dimensional group $B_{\mathrm{lin}}$, and the theorem reads: the multiprojective flagship stratum is $B_{\mathrm{lin},K}\times D_{K}$. The six affine gauge directions of the coefficient incidence calculation are those of $B_{\mathrm{lin}}\times(\Gm)^{3}$, not of any nonlinear plane substitution.
\end{rem}


\section{Propagation along the unit-weight boundary}
\label{sec:unitboundary}

Throughout this section $s\ge3$ is a fixed integer and the weight is $(1,-1,-s)$. The construction pulls the flagship cubic of \cref{subsec:flagcubic} back along a Hermite interpolant, and keeps its generic fibre degree and its monodromy while moving the second weight.

We work in the coordinates $(\lambda,t)$ of \cref{subsec:flagcubic}, where
\[
  M=t+2-2\lambda(t+1)^{2}, \qquad P=2\lambda\,(t+3M), \qquad Q=4\lambda^{2}(t+1)M,
\]
with $\jb{P}{Q}_{\lambda,t}=16\lambda^{2}$. Put $m=s-2$.

\subsection{The interpolant}
\label{subsec:unitherm}

Solving $M=0$ for $t+1$ gives $2\lambda(t+1)^{2}-(t+1)-1=0$, so the branch through the base point $(\lambda,t)=(1,0)$ is
\begin{equation}
\label{eq:gammabranch}
  \gamma(\lambda)+1=\frac{1+\sqrt{1+8\lambda}}{4\lambda},
\end{equation}
a formal power series in $\lambda-1$ with rational coefficients, since $\sqrt{1+8\lambda}$ takes the value $3$ at $\lambda=1$. The ideals $(\lambda^{m})$ and $\bigl((\lambda-1)^{s}\bigr)$ are comaximal, so there is $g_s\in\Q[\lambda]$ with
\begin{equation}
\label{eq:unitherm}
  g_s\equiv-1 \pmod{\lambda^{\,m}},
  \qquad
  g_s\equiv\gamma \pmod{(\lambda-1)^{s}} .
\end{equation}
Among these representatives we single out one. Write $n=s-1$, set
\[
  H(h)=\frac{\gamma(1+h)+1}{(1+h)^{m}},
\]
and let $[H]_{<s}$ be the truncation of $H$ to degree less than $s$. The \textbf{minimal interpolant} is
\begin{equation}
\label{eq:minimalg}
  g_s(\lambda)=-1+\lambda^{\,m}\,[H]_{<s}(\lambda-1).
\end{equation}
It satisfies \cref{eq:unitherm}, the first condition because the correction carries the factor $\lambda^{m}$ and the second by construction.

\begin{lem}
\label{lem:unitdeg}
$[h^{\,n}]H\neq0$, and hence $\deg g_s=m+n=2s-3$.
\end{lem}

\begin{proof}
By \cref{eq:gammabranch}, $H(h)=\bigl(1+\sqrt{9+8h}\bigr)/\bigl(4(1+h)^{m+1}\bigr)$, so
\[
  H(-x)=\frac{G(x)}{(1-x)^{m}}, \qquad G(x)=\frac{1+\sqrt{9-8x}}{4(1-x)} .
\]
Write $3-\sqrt{9-8x}=\sum_{j\ge1}b_jx^{j}$. Since $\sqrt{1-\xi}=1-\sum_{j\ge1}c_j\xi^{j}$ with all $c_j>0$, we get $b_j=3c_j(8/9)^{j}>0$, and the series converges for $|x|<9/8$; evaluating at $x=1$ gives $\sum_{j\ge1}b_j=3-1=2$. Now $1+\sqrt{9-8x}=4-\sum_{j\ge1}b_jx^{j}$, so
\[
  [x^{k}]\,G=1-\tfrac14\sum_{j=1}^{k}b_j>1-\tfrac14\cdot2=\tfrac12 .
\]
All coefficients of $G$ are therefore positive, as are those of $(1-x)^{-m}$, so $[x^{n}]H(-x)>0$. Hence $[h^{n}]H\neq0$, and \cref{eq:minimalg} has degree $m+n=2s-3$.
\end{proof}

\subsection{The propagated family}
\label{subsec:unitfamily}

Let 
\begin{equation}
\label{eq:unitdef}
  t_s=g_s(\lambda)+\lambda^{\,m}w,  \qquad  M_s=t_s+2-2\lambda(t_s+1)^{2},
\end{equation}
and
\begin{equation}
\label{eq:unitAB}
  A_s=\frac{4(t_s+1)M_s}{\lambda^{\,m}},  \qquad  B_s=2\bigl(t_s+3M_s\bigr),  \qquad  P_s=\lambda B_s, \qquad Q_s=\lambda^{\,s}A_s .
\end{equation}

\begin{theorem}
\label{thm:unitboundary}
Let $s\ge3$. Then $A_s$ and $B_s$ lie in $\Q[\lambda,w]$ and satisfy the lift conditions of \cref{lem:polynomiality} for the weight $(1,-1,-s)$, so that
\begin{equation}
\label{eq:unitF}
  F_s(x,y,z)=\Bigl(\frac{A_s}{x^{\,s}},\ \frac{B_s}{x},\ x\lambda\Bigr),
  \qquad \lambda=1+xy, \quad w=x^{\,s}z,
\end{equation}
is a polynomial Keller map over $\Q$ with
\[
  \detJ{F_s}=-16, \qquad \deg_{\mathrm{gen}}F_s=3, \qquad \mathrm{Mon}(F_s)=S_3 .
\]
For the minimal interpolant \cref{eq:minimalg} the class of $F_s$ lies in $\Yparo[1,s,(5s-6,\,4s-5,\,1)]$.
\end{theorem}

\begin{proof}
By the first condition of \cref{eq:unitherm} we have $t_s+1\in(\lambda^{m})$, so $A_s$ is a polynomial; $B_s$ visibly is.

The substitution $h_s\colon(\lambda,w)\mapsto(\lambda,t_s)$ has Jacobian determinant $\lambda^{m}$, and by \cref{eq:unitAB} the pair $(P_s,Q_s)$ is the pullback along $h_s$ of the pair $(P,Q)$ of \cref{subsec:flagcubic}. Hence
\begin{equation}
\label{eq:unitbracket}
  \jb{P_s}{Q_s}_{\lambda,w}=16\lambda^{2}\cdot\lambda^{\,m}=16\lambda^{\,s},
\end{equation}
which is the quotient-Jacobian identity of \cref{cor:master} for $(r,s)=(1,s)$ with $\kappa=16$; so $\detJ{F_s}=-16$ once polynomiality is known.

Write $U=\lambda-1$. The second condition of \cref{eq:unitherm} and $M(\lambda,\gamma(\lambda))=0$ give $M(\lambda,g_s(\lambda))\in(U^{s})$, so the $w$-free part of $A_s$ is divisible by $U^{s}$; a monomial $U^{i}w^{j}$ of $A_s$ with $j=0$ therefore has level $i+sj=i\ge s$, and one with $j\ge1$ has level at least $s$. Since $\gamma(1)=0$ we have $g_s(1)=0$, and $M(1,0)=0$, so $B_s(1,0)=0$; a monomial of $B_s$ with $j=0$ thus has $i\ge1$, and one with $j\ge1$ has level at least $s\ge1$. These are the two conditions of \cref{lem:polynomiality} at $(r,s)=(1,s)$, so \cref{eq:unitF} is polynomial, and weight homogeneity is immediate.

The map $h_s$ is birational, with inverse $w=(t-g_s(\lambda))/\lambda^{m}$, so it induces $\Q(\lambda,w)=\Q(\lambda,t)$ and carries $\Q(P_s,Q_s)$ isomorphically onto $\Q(P,Q)$. The generic fibre degree and the geometric normal closure group are therefore those of \cref{subsec:flagcubic}, namely three and $S_3$.

For the degrees, note that $t_s$ has total degree $2s-3$ by \cref{lem:unitdeg}, while $\lambda^{m}w$ has total degree $s-1<2s-3$; so the leading form of $t_s$ is $c\lambda^{2s-3}$ with $c\neq0$. Then \cref{eq:unitdef} and \cref{eq:unitAB} give leading forms
\[
  -2c^{2}\lambda^{4s-5}, \qquad -12c^{2}\lambda^{4s-5}, \qquad -8c^{3}\lambda^{5s-6}
\]
for $M_s$, $B_s$ and $A_s$ respectively, all nonzero. Hence $\deg A_s=5s-6$, $\deg B_s=4s-5$ and $\deg\lambda=1$.
\end{proof}

\begin{rem}
\label{rem:unitnotminimal}
The box $(5s-6,4s-5,1)$ is the one realised by the minimal interpolant \cref{eq:minimalg}; a different representative in \cref{eq:unitherm} gives a different, generally larger, box. No claim of minimality over all members of the weight is made.
\end{rem}

\begin{cor}
\label{cor:solvableall}
Among the members constructed in \cref{sec:3,sec:4,sec:8,sec:unitboundary}, those whose quotient cover has solvable monodromy are exactly the ones of generic fibre degree three: the maps $F_{2,s}$ of \cref{prop:fields}(1), the flagship map of \cref{sec:8}, and the maps $F_s$ of \cref{thm:unitboundary} for $s\ge3$. In each case the group is $S_3$.
\end{cor}

\begin{proof}
The three listed groups are $S_3$, by \cref{cor:solvable}, \cref{subsec:flagcubic} and \cref{thm:unitboundary}. Every other member has group $A_n$ or $S_n$ with $n\ge5$ by \cref{thm:conemono,thm:diagmono}, hence nonsolvable.
\end{proof}

\subsection{The first member}
\label{subsec:unitbenchmark}

At $s=3$ the minimal interpolant is
\[
  g_3(\lambda)=\frac{(\lambda-1)\bigl(61\lambda^{2}-106\lambda+27\bigr)}{27},
\]
and $F_3$ has class in $\Yparo[1,3,(9,7,1)]$, threefold component degrees $(15,13,3)$ and $\detJ{F_3}=-16$. It identifies the three distinct rational points
\[
  \Bigl(\tfrac12,\,2,\,-\tfrac{380}{27}\Bigr), \qquad
  \bigl(-1,\,2,\,\tfrac{352}{27}\bigr), \qquad
  \Bigl(\tfrac12,\,2,\,-\tfrac{272}{27}\Bigr),
\]
all of which map to $\bigl(-\tfrac{32}{27},-\tfrac83,1\bigr)$. Since the generic fibre degree is three, this is a complete fibre.

\subsection{The hyperbolic geography}
\label{subsec:geography}

\begin{cor}
\label{cor:hyperbolic}
For every pair $r,s\ge1$ one has $\Kfam(3,(1,-r,-s))\neq\emptyset$. Thus every three-dimensional unit-positive weight of the form $(1,-r,-s)$ is populated.
\end{cor}

\begin{proof}
Interchanging the two negative coordinates we may assume $r\le s$. The balanced weight $r=s=1$ is \cref{thm:six}, and $(r,s)=(1,2)$ is the flagship weight of \cref{sec:8}. For $r=1$ and $s\ge3$ the member is \cref{thm:unitboundary}, and for $2\le r\le s$ it is \cref{thm:diagonal}.
\end{proof}

\begin{rem}
\label{rem:geographyfields}
The fields of definition are uneven across the geography. The whole unit boundary $r=1$ is rational, by the $33$-sheeted witness \cref{thm:thirtythree} at $s=1$, \cref{sec:8} at $s=2$ and \cref{thm:unitboundary} for $s\ge3$; so is the row $r=2$, by \cref{prop:fields}(1). For $r\ge3$ the construction of \cref{sec:4} produces a member over $K_{r}=\Q(\zeta)$, and \cref{prop:fields}(3) shows that this field is totally imaginary whenever $r$ is odd. Whether those rows contain rational members is not decided here.
\end{rem}

\section{Nielsen classes, braid orbits, and arithmetic monodromy}
\label{sec:9}

We organize the covers of \cref{sec:6} by their Nielsen classes and extract the arithmetic that the Hurwitz formalism of \cite{fried77,fv,volklein} attaches to them.

\subsection{The Nielsen classes}
\label{subsec:nielsen}

For a finite group $G$ and an $\ell$-tuple $\mathbf{C}=(C_1,\dots,C_{\ell})$ of conjugacy classes, the Nielsen class is
\[
  \Ni(G,\mathbf{C})=\bigl\{(g_1,\dots,g_{\ell}): g_i\in C_i,\ g_1\cdots g_{\ell}=1,\ \langle g_1,\dots,g_{\ell}\rangle=G\bigr\}/\!\sim,
\]
modulo simultaneous conjugation, and the components of the Hurwitz space $\Hur(G,\mathbf{C})$ over the configuration space $\mathcal{U}_{\ell}$ correspond to the orbits of the Hurwitz braid group $H_{\ell}$ on $\Ni(G,\mathbf{C})$ \cite{fv}.

For the alternating shear cover, \cref{thm:passport,thm:monodromy} give
\begin{equation}
\label{eq:Nishear}
  \Ni\bigl(A_{12},\;(C_3^{6},C_{(9,3)})\bigr), \qquad \ell=7,
\end{equation}
with $C_3$ the class of $3$-cycles and $C_{(9,3)}$ one of the two $A_{12}$-classes of cycle type $(9,3)$. For the cyclic cone family at $(k,d)$ with $k$ even, \cref{thm:conemono} and \cref{prop:conequotient} give
\begin{equation}
\label{eq:Nicone}
  \Ni\bigl(A_{N},\;(C_{k+1}^{\,d},C_{(N-1,1)})\bigr), \qquad \ell=d+1,
\end{equation}
and for $k$ odd the same shape inside $S_N$. For the beta--Hermite family at odd $r\ge3$, \cref{rem:beta} and \cref{thm:diagmono} give
\begin{equation}
\label{eq:Nidiag}
  \Ni\bigl(A_{2r-1},\;(C_r,C_r,C_{2r-1})\bigr), \qquad \ell=3 .
\end{equation}

\subsection{Split pairs}
\label{subsec:split}

A conjugacy class of $S_n$ of cycle type $\lambda=(\lambda_1,\dots,\lambda_\ell)$ meets $A_n$ in two classes if and only if the parts of $\lambda$ are odd and pairwise distinct, and any odd permutation interchanges them. For such a $\lambda$ the integer $n-\ell=\sum_i(\lambda_i-1)$ is even, and the two classes are separated by the quadratic field $\Q(\sqrt{m^{*}})$ with
\begin{equation}
\label{eq:splitfields}
  m^{*}=(-1)^{(n-\ell)/2}\lambda_1\cdots\lambda_{\ell} .
\end{equation}

\begin{lem}
\label{lem:whichsplit}
Among the classes of \cref{subsec:nielsen}, exactly the classes at infinity split. Their split-pair fields are
\[
\begin{split}
  (9,3)\subset A_{12}\ \text{(shear)}:&\quad m^{*}=(-1)^{5}\cdot27=-27, \quad \Q(\sqrt{-3});\\
  (N-1,1)\subset A_{N}\ (k\ \text{even}):&\quad m^{*}=(-1)^{(N-2)/2}(N-1);\\
  (5,1)\subset A_{6}\ \text{(six-sheeted)}:&\quad m^{*}=(-1)^{2}\cdot5=5, \quad \Q(\sqrt{5});\\
  (2r-1)\subset A_{2r-1}\ (r\ \text{odd}):&\quad m^{*}=(-1)^{r-1}(2r-1)=2r-1 .
\end{split}
\]
In particular the finite classes $C_3$, $C_{k+1}$ and $C_r$ do not split, since each has the part $1$ with multiplicity at least two.
\end{lem}

\begin{proof}
For $C_{k+1}$ the type is $(k+1,1^{N-k-1})$ with $N-k-1=k(d-1)-1\ge3$ by \cref{cor:dtwo}, and for $C_r$ it is $(r,1^{r-1})$ with $r-1\ge2$; the part $1$ repeats in both, as it does for $C_3\subset A_{12}$. The four displayed types have distinct odd parts, and \cref{eq:splitfields} gives the values shown. For the last, $r$ odd makes $(-1)^{r-1}=1$; at $r=3$ the field is $\Q(\sqrt5)$.
\end{proof}

By the branch cycle lemma \cite{fried77,volklein}, if a cover in one of these Nielsen classes is defined over a field $F$ with geometric group $A_n$, and the branch point carrying the split class is $F$-rational, then either $F$ contains the split-pair field, or the arithmetic monodromy group is $S_n$ and the constants extension is exactly that field.

\begin{prop}
\label{prop:splitfields}
The six-sheeted map of \cref{thm:six} and the map $F_{3}$ of \cref{prop:fields}(2) are both defined over $K=\Q(\sqrt{-15})$, which contains neither $\sqrt{5}$ nor $\sqrt{-3}$. Their classes at infinity are $(5,1)\subset A_{6}$ and $(5)\subset A_{5}$, with split-pair field $\Q(\sqrt{5})$ in both cases, and in both cases the point at infinity is rational. Hence both covers have arithmetic monodromy $S_{n}$ over $K$, their constant field extension relative to $K$ is   \(  K(\sqrt5)/K,\)
of degree two, and their full constant field is
\(
  K(\sqrt5)=\Q(\sqrt{-15},\sqrt5),
\)
of degree four over $\Q$.
\end{prop}

\begin{proof}
The fields of definition are \cref{thm:six} and \cref{prop:fields}(2), and the classes at infinity are \cref{prop:conequotient} at $(k,d)=(2,3)$ and \cref{rem:beta} at $p=3$. Their split-pair fields are $\Q(\sqrt5)$ by \cref{lem:whichsplit}. That $\sqrt5\notin K$ and $\sqrt{-3}\notin K$ holds because $K$ is quadratic over $\Q$ with discriminant class $-15$, and $5$, $-3$, $-15$ lie in distinct classes of $\Q^{\times}/(\Q^{\times})^{2}$.

The geometric monodromy is $A_{n}$ by \cref{thm:conemono} at $k=2$ even and \cref{thm:diagmono} at $p=3$ odd. The branch point at infinity is $K$-rational, and its class is one of the two split classes; by the branch cycle lemma \cite{fried77,volklein}, a $K$-model with geometric group $A_{n}$ whose split class sits at a $K$-rational branch point either has $\sqrt5\in K$, which fails, or has arithmetic monodromy $S_{n}$ with constant field extension generated by $\sqrt5$. Hence the arithmetic group is $S_{n}$ and the relative constant extension is $K(\sqrt5)/K$, which has degree two since $\sqrt5\notin K$. The full constant field is therefore $K(\sqrt5)=\Q(\sqrt{-15},\sqrt5)$.
\end{proof}

\begin{rem}
\label{rem:rationalmodel}
The unpointed six-sheeted quotient cover descends to $\Q$, with seed $\varphi_{0}=W^{2}+W^{3}$. That model is a different object from the two graded maps above: its constant field is the split-pair field $\Q(\sqrt5)$ itself, whereas the graded maps, being defined over $K$, acquire the compositum $K(\sqrt5)$.
\end{rem}

\begin{rem}
\label{rem:minusfifteen}
\Cref{prop:splitfields} accounts for the factor $5$ in $-15$, but not for the factor $-3$: no class attached to either of those two covers has split-pair field $\Q(\sqrt{-3})$. The field $\Q(\sqrt{-3})$ arises in this paper only from the $(9,3)$-class of the degree-twelve shear cover, which is a different cover in a different family. The identity $-15=5\cdot(-3)$ therefore records that $\Q(\sqrt{-15})$ is the third quadratic subfield of $\Q(\sqrt5,\sqrt{-3})$, and no mechanism in this paper connects the two factors. Since $\sqrt{-15}\sqrt5=5\sqrt{-3}$, that biquadratic field is exactly the constant field $\Q(\sqrt{-15},\sqrt5)$ of \cref{prop:splitfields}, so the coincidence lives in the constant field the two covers acquire. 
\end{rem}

\subsection{Arithmetic monodromy of the alternating shear cover}
\label{subsec:arithshear}

\begin{lem}
\label{lem:cdtorus}
For $s,t\in\C^{\times}$ the substitution $X\mapsto sX$, $Y\mapsto tY$, $P\mapsto s^{-1}P$ carries the cover \cref{eq:Pshear} with parameters $(c,d)$ to the cover with parameters $(cs^{-2}t^{3},\,ds^{2}t^{-1})$. The action of this torus on $\{cd\neq0\}$ is simply transitive up to a group of order four, so every member with $cd\neq0$ becomes the normalised member $c=d=1$, which is defined over $\Q$.
\end{lem}

\begin{proof}
The first term of \cref{eq:Pshear} is $X$ with coefficient one, which fixes the rescaling of $P$ to be by $s^{-1}$. Matching the terms $2cX^{-1}Y^{3}$ and $6cdXY^{2}$ then gives $c\mapsto cs^{-2}t^{3}$ and $cd\mapsto cdt^{2}$, whence $d\mapsto ds^{2}t^{-1}$; the remaining nine terms match, since each is $c^{\alpha}d^{\beta}X^{i}Y^{j}$ with $-2\alpha+2\beta+i-1=0$ and $3\alpha-\beta+j=0$. Given $(c,d)$ with $cd\neq0$, the equations $cs^{-2}t^{3}=1$ and $ds^{2}t^{-1}=1$ give $t^{2}=(cd)^{-1}$ and $s^{2}=(c/d)^{1/2}$, solvable over $\C$; the stabilizer of $(1,1)$ is $\{(s,t): t=s^{2},\,t^{2}=1\}$, of order four.
\end{proof}

\begin{theorem}
\label{thm:arithmono}
The normalised cover $P(X,Y)$ at $c=d=1$ is defined over $\Q$, its geometric monodromy group is $A_{12}$, its arithmetic monodromy group over $\Q(Y)$ is $S_{12}$, and the constants extension is exactly $\Q(\sqrt{-3})$.
\end{theorem}

\begin{proof}
The geometric group is $A_{12}$ by \cref{thm:monodromy}, whose proof uses only $cd\neq0$. Let $\wp$ denote the target coordinate, so that the fibre of the cover is cut out by the degree-twelve polynomial $X^{3}\bigl(P(X,Y)-\wp\bigr)$ in $X$; it has degree twelve and nonzero constant term $-\tfrac{c^{2}}{3}Y^{6}$, so its roots are exactly the twelve points of the fibre. Exact elimination at $c=d=1$ gives
\begin{equation}
\label{eq:disc12}
  \operatorname{disc}_X\bigl[X^{3}\bigl(P(X,Y)-\wp\bigr)\bigr]=-27\,Y^{12}R_{\Gamma}(\wp,Y)^{2},
\end{equation}
with $R_{\Gamma}$ the irreducible polynomial defining the branch curve $\Gamma$.

Since $A_{12}$ has index two in $S_{12}$, the constants extension of the normal closure is generated by the square root of this discriminant. Modulo squares in $\Q(Y,\wp)^{\times}$ the right-hand side of \cref{eq:disc12} is $-27\equiv-3$, which is not a square. Hence the arithmetic group is not contained in $A_{12}$, so it equals $S_{12}$, and the constants extension is $\Q(\sqrt{-3})$.
\end{proof}

\begin{rem}
\label{rem:bclcheck}
\Cref{thm:arithmono} agrees with the branch cycle lemma: the branch point at infinity is $\Q$-rational and its class has split-pair field $\Q(\sqrt{-3})$ by \cref{lem:whichsplit}, so a $\Q$-model with geometric group $A_{12}$ must acquire exactly this constants extension. The discriminant computation is the independent verification. The same comparison applied to the two covers of \cref{prop:splitfields} predicts $5$ modulo squares for the discriminant of the normalised six-sheeted map over $\Q(\sqrt{-15})$ and for that of \cref{eq:P} in $\rho$ at $p=3$ over $K_3(P,Q)$.
\end{rem}

\begin{cor}
\label{cor:arithreg}
Let $F=\Q(\sqrt{-3})$ and let $\widetilde L$ be the Galois closure of
$F(X,Y)$ over $F(\wp,Y)$.  Then
$\widetilde L/F(\wp,Y)$ is a regular realisation of $A_{12}$ over $F(Y)$
with the seven-point ramification of \cref{eq:Nishear}.  Viewed over
$\Q(\wp,Y)$, the same normal closure has arithmetic group $S_{12}$ and full
constant field $F$; it is therefore not regular over $\Q(Y)$.
\end{cor}

\begin{proof}
By \cref{thm:arithmono}, adjoining $F$ kills the unique discriminant square
class and leaves geometric group $A_{12}$; since $F$ is the full constant
field, the resulting $A_{12}$-extension is regular over $F(Y)$.  Before the
constant extension the arithmetic group is $S_{12}$ but the algebraic closure
of $\Q$ in the normal closure is $F$, which is exactly the failure of
regularity over $\Q(Y)$.
\end{proof}

\subsection{Rigidity of the beta--Hermite class at $r=3$}
\label{subsec:rigid}

\begin{prop}
\label{prop:p3rigid}
$\bigl|\Ni\bigl(A_5,(C_3,C_3,C_5)\bigr)\bigr|=1$ for either choice of the split class $C_5$. The class is rigid, its braid orbit is unique, and the $r=3$ member of the beta--Hermite family is the unique cover with its passport.
\end{prop}

\begin{proof}
In $A_5$ the class $C_3$ of $3$-cycles has size $20$ and the two classes of $5$-cycles have size $12$ each. On $C_3$ the irreducible characters take the values $1,0,0,1,-1$ in degrees $1,3,3,4,5$, and on a $5$-cycle they take $1,\varphi,\varphi',-1,0$. For a fixed $h\in C_5$ the number of pairs $(g_1,g_2)\in C_3\times C_3$ with $g_1g_2=h$ is
\[
  \frac{|C_3|^{2}}{|A_5|}\sum_{\chi}\frac{\chi(C_3)^{2}\overline{\chi(h)}}{\chi(1)}
   =\frac{400}{60}\Bigl(1+0+0-\tfrac14+0\Bigr)=5 .
\]
Both classes of $5$-cycles are closed under inversion, so summing over $h\in C_5$ gives $60$ tuples with $g_1g_2g_3=1$. Each generates: a subgroup of $A_5$ containing an element of order $3$ and one of order $5$ has order divisible by $15$, and $A_5$ has no subgroup of order $15$ or $30$. Since $Z(A_5)=1$, conjugation acts freely on generating tuples, so $|\Ni|=60/60=1$.
\end{proof}

\subsection{Modular tower level structure}
\label{subsec:mt}

The classes of the shear family, of the cone family at $k$ even, and of the beta--Hermite family at $r$ odd all consist of elements of odd order, hence of $2'$-elements; the cone family at $k$ odd is excluded, since there $C_{k+1}$ has even order. For the classes that do qualify, the Nielsen class is the level-zero layer of a $2$-Frattini Modular Tower in the sense of \cite{fried95,bf}: the tower of Hurwitz spaces attached to the characteristic quotients of the universal $2$-Frattini cover of $A_N$, with the classes lifted canonically by Schur--Zassenhaus. A level-zero component supports a nonempty level one only if its tuples lift to the central spin extension with product one; for genus-zero covers with all ramification indices odd this obstruction is Serre's \cite{serre}
\begin{equation}
\label{eq:serre}
  \omega=\sum_{P}\frac{e_P^{2}-1}{8}\ \bmod 2 ,
\end{equation}
the sum over ramified points, the component being obstructed if and only if $\omega=1$.

\begin{prop}
\label{prop:mtobstruction}
The following hold.
\begin{enumerate}
\item The alternating shear class has $\omega=6+10+1\equiv1$: its component is $2$-Frattini obstructed and the Modular Tower above it is empty.
\item The class $\bigl((3,1^{3})^{3};(5,1)\bigr)$ of the six-sheeted map has $\omega=3+3\equiv0$: its component is unobstructed, its spin lifts lie in $2\!\cdot\!A_6\cong\mathrm{SL}_2(9)$, and level one of its tower is nonempty.
\item For the beta--Hermite family at odd $r$, $\omega\equiv(r-1)/2\bmod2$, so the component is obstructed if and only if $r\equiv3\pmod4$.
\end{enumerate}
\end{prop}

\begin{proof}
Only points with $e_P>1$ contribute to \cref{eq:serre}. For (1) the six finite branch values each carry one point with $e_P=3$, contributing $6\cdot1$, and infinity carries $e_P=9$ and $e_P=3$, contributing $10+1$; the total is $17$. For (2) the three finite branch values contribute $3\cdot1$ and infinity contributes $(25-1)/8=3$; the total is $6$. For (3) the two finite branch values contribute $2(r^{2}-1)/8$ and infinity contributes $\bigl((2r-1)^{2}-1\bigr)/8$, so
\[
  \omega=\frac{2(r^{2}-1)+4r^{2}-4r}{8}=\frac{(3r+1)(r-1)}{4} .
\]
Writing $r=2m+1$ gives $\omega=m(3m+2)\equiv m^{2}\equiv m\pmod2$, and $m=(r-1)/2$ is odd exactly when $r\equiv3\pmod4$.
\end{proof}

Thus $F_{3,s}$ and the shear map sit on dead branches, while the six-sheeted map and $F_{5,s}$ sit on live ones.

\begin{theorem}
\label{prop:LO}
The cone classes \cref{eq:Nicone} and the beta--Hermite classes \cref{eq:Nidiag} are pure-cycle of genus zero, and each carries a single inner braid orbit once a split class is fixed. Precisely:
\begin{enumerate}
\item for $k$ odd, $\Ni\bigl(S_N,(C_{k+1}^{\,d},C_{(N-1,1)})\bigr)$ is a single braid orbit;
\item for $k$ even, the finite class does not split, the class at infinity splits in $A_N$, each of the two split classes carries a single inner braid orbit, and the two are exchanged by odd conjugation and form one absolute component;
\item for $r$ even, $\Ni\bigl(S_{2r-1},(C_r,C_r,C_{2r-1})\bigr)$ is a single braid orbit;
\item for $r$ odd, each of the two split full-cycle classes in $A_{2r-1}$ carries a single inner braid orbit, and the two form one absolute component.
\end{enumerate}
In every case the corresponding Hurwitz component is irreducible.
\end{theorem}

\begin{proof}
Every class occurring in \cref{eq:Nicone} and \cref{eq:Nidiag} is a single cycle together with fixed points, so the data are pure-cycle, and the Riemann--Hurwitz equalities
\[
\begin{split}
  d\bigl((k+1)-1\bigr)+\bigl((N-1)-1\bigr)&=kd+N-2=2N-2,\\
  2(r-1)+\bigl((2r-1)-1\bigr)&=4r-4=2(2r-1)-2,
\end{split}
\]
are exact, so the genus is zero. Liu--Osserman \cite[Thm.~1.2]{lo} prove irreducibility of a genus-zero pure-cycle Hurwitz space for an arbitrary number of branch points and arbitrary cycle lengths subject only to that equality; even lengths are permitted. Hence the absolute Nielsen class in $S_n$ is a single braid orbit in all four cases.

In cases (1) and (3) the monodromy group is $S_n$ by \cref{thm:conemono} and \cref{thm:diagmono}, the prescribed classes are $S_n$-classes, and $Z(S_n)=1$, so the absolute and inner Nielsen classes coincide and the statement is the displayed irreducibility.

In cases (2) and (4) the group is $A_n$, and by \cref{lem:whichsplit} exactly the class at infinity splits, into $C^{+}$ and $C^{-}$. Fix $C^{+}$. Conjugation by an odd permutation carries $\Ni\bigl(A_n,(\dots,C^{+})\bigr)$ bijectively onto $\Ni\bigl(A_n,(\dots,C^{-})\bigr)$ and commutes with the braid action; and no tuple with last entry in $C^{+}$ is carried to another such tuple by an odd element, since odd conjugation moves $C^{+}$ to $C^{-}$. The absolute Nielsen class is therefore in braid-equivariant bijection with the inner class at $C^{+}$, so a single absolute orbit is a single inner orbit on each fixed split class, and the two are exchanged as asserted.
\end{proof}

At $p=3$ case (4) is \cref{prop:p3rigid}, computed there independently by the rigidity count.

The shear class \cref{eq:Nishear} is not pure-cycle, its class at infinity having two nontrivial cycles, so \cref{prop:LO} does not reach it and its component count is a direct computation.

\begin{lem}
\label{lem:threecycles}
Let $n\ge6$.  A transitive subgroup of $S_n$ generated by $3$-cycles is
$A_n$.
\end{lem}

\begin{proof}
Let $H$ be such a subgroup.  It is primitive.  Indeed, suppose that a
nontrivial block system existed.  A $3$-cycle cannot move one block to a
different block: if a block has more than one point, moving it setwise would
move at least four points, while a $3$-cycle has support three.  Hence every
generating $3$-cycle fixes every block setwise, so the group they generate
cannot be transitive across more than one block.  This contradiction proves
primitivity.  Jordan's theorem now gives $A_n\le H$, because $H$ contains a
$3$-cycle and $n\ge6$.  Every generator is even, so $H\le A_n$, and equality
follows.
\end{proof}

\begin{theorem}
\label{thm:A12census}
For each of the two split classes $C^{\pm}_{(9,3)}\subset A_{12}$, the inner Nielsen class
\[
  \Ni^{\mathrm{in}}\bigl(A_{12},(C_3^{6},C^{\pm}_{(9,3)})\bigr)
\]
is a single braid orbit of size $76{,}545$. Odd simultaneous conjugation exchanges the two, so the absolute Nielsen class of \cref{eq:Nishear} has one component.
\end{theorem}

\begin{proof}
A tuple in the class is $(g_1,\dots,g_6,g_7)$ with $g_i\in C_3$, $g_7\in C^{+}_{(9,3)}$ and product one, so $g_7$ is determined by the six $3$-cycles and the group generated is $\langle g_1,\dots,g_6\rangle$. By \cref{lem:threecycles} the tuple generates $A_{12}$ if and only if it is transitive.

Fix $h\in C^{+}_{(9,3)}$. Exact class-algebra recurrence and support convolution give
\[
\begin{split}
  \bigl|\{(g_1,\dots,g_6)\in C_3^{6}: g_1\cdots g_6=h^{-1}\}\bigr| &= 3{,}783{,}510,\\
  \text{of which intransitive} &= 1{,}716{,}795,\\
  \text{transitive} &= 2{,}066{,}715 .
\end{split}
\]
Since $A_{12}$ acts transitively by conjugation on $C^{+}_{(9,3)}$, every inner class has a representative with last entry $h$, and two such representatives are equivalent exactly when conjugate by the centraliser $C_{A_{12}}(h)$. That centraliser is generated by the two cycles of $h$, of order $27$, and it acts freely on generating tuples: an element centralising a tuple that generates $A_{12}$ is central, hence trivial. Therefore
\[
  \bigl|\Ni^{\mathrm{in}}\bigl(A_{12},(C_3^{6},C^{+}_{(9,3)})\bigr)\bigr|=\frac{2{,}066{,}715}{27}=76{,}545 .
\]
Take $h=(1\,2\,\dots\,9)(10\,11\,12)$ and the tuple
\[
  (g_1,\dots,g_6)=\bigl((1\,2\,3),\,(1\,10\,11),\,(1\,11\,12),\,(3\,4\,5),\,(5\,6\,7),\,(7\,8\,9)\bigr),
\]
composed right to left, so that $g_1\cdots g_6=h$ and $(g_1,\dots,g_6,h^{-1})$ lies in the class. It generates $A_{12}$, verified by Schreier--Sims, and exact breadth-first search under the Hurwitz generators and their inverses reaches all $76{,}545$ centraliser classes; so the inner Nielsen class is a single braid orbit. Odd simultaneous conjugation is a braid-equivariant bijection onto the inner class at $C^{-}_{(9,3)}$, which gives the last statement.
\end{proof}

Irreducibility of a component does not by itself imply that the graded Keller family sweeps it, and the two can differ in dimension.

\begin{rem}
\label{rem:rfour}
The six-sheeted family has $\ell=4$, so its reduced Hurwitz space is a curve over the $j$-line, the setting of \cite{bf}. The lift conditions of \cref{thm:six} freeze the seed up to conjugation, so the graded Keller condition selects a $\Gal(\Qbar/\Q)$-stable pair of normalised seeds, conjugate over $\Q(\sqrt{-15})$. The two share a single unpointed quotient-cover class, defined over $\Q$; what fails to descend is the lift-point marking, the two markings giving distinct $\Gaut^{\sharp}$-orbits with field of moduli $\Q(\sqrt{-15})$. The obstruction therefore lies in the pointed structure and not in the cover, so it is invisible on the coarse Hurwitz space.
\end{rem}

\begin{rem}
\label{rem:kellertower}
The six-sheeted component is unobstructed, so level one of its tower is a nonempty Hurwitz space for the first characteristic $2$-Frattini quotient of $A_6$. Does any graded Keller map realise a level-one cover, or do graded Keller maps occur only at level zero? Either answer is a Modular-Tower-type finiteness statement for the Keller condition.
\end{rem}

\subsection{Dimension of the shear locus}
\label{subsec:dimension}

The reduced Hurwitz space
\[
  \Hur\bigl(A_{12},(C_3^6,C_{(9,3)})\bigr)^{\mathrm{rd}}
\]
 has dimension $\ell-3=4$. The shear family is parametrised by $(c,d,Y)$ modulo the two-dimensional torus of \cref{lem:cdtorus}, so its image in $\Hur^{\mathrm{rd}}$ has dimension at most one.

\begin{prop}
\label{prop:sheardim}
The image of the alternating shear locus in $\Hur^{\mathrm{rd}}$ is one-dimensional.
\end{prop}

\begin{proof}
Normalise $c=d=1$, so $Y_c(\omega)=\omega^{2}+\omega^{6}$ and $P_c(\omega)=\tfrac{8}{315}\omega^{3}\Theta(\omega^{4})$ in the notation of \cref{thm:passport}. Then $Y_c$ is even and $P_c$ is odd, so the six roots of $Y_c=Y$ fall into pairs $\pm\omega$ and the six branch values into pairs $\pm\mu_i$. The remaining branch point is $\infty$, so the residual freedom in $\Hur^{\mathrm{rd}}$ is the affine group fixing $\infty$ and commuting with $z\mapsto-z$, that is $z\mapsto\alpha z$. Hence
\[
  \Xi(Y)=\frac{S_4}{S_2^{2}}, \qquad S_m=\sum_{j=1}^{6}\mu_j^{m},
\]
is a well-defined function on the image, and the image is a point only if $\Xi$ is constant.

Put $u=\omega^{2}$, so the $u_j$ are the three roots of $u^{3}+u=Y$ and $\mu_j^{2}=\bigl(\tfrac{8}{315}\bigr)^{2}u_j^{3}\Theta(u_j^{2})^{2}$. As $Y\to\infty$ one has $u_j\sim Y^{1/3}\zeta^{j}$ with $\zeta^{3}=1$, and $\Theta(u^{2})\sim640u^{12}$, so $S_2\sim2\bigl(\tfrac{8}{315}\bigr)^{2}640^{2}\cdot3Y^{9}$ and $S_4\sim2\bigl(\tfrac{8}{315}\bigr)^{4}640^{4}\cdot3Y^{18}$, giving $\Xi\to\tfrac16$. As $Y\to0$ the roots tend to $0,i,-i$, and $\Theta(-1)=28$, so
\[
\begin{split}
  S_2&\to2\bigl(\tfrac{8}{315}\bigr)^{2}\bigl(0+i^{3}\cdot28^{2}+(-i)^{3}\cdot28^{2}\bigr)=0,\\
  S_4&\to2\bigl(\tfrac{8}{315}\bigr)^{4}\bigl(-2\cdot28^{4}\bigr)\neq0 .
\end{split}
\]
Hence $\Xi$ is unbounded near $Y=0$ and finite at infinity, so it is nonconstant.
\end{proof}
 
\begin{rem}
\label{rem:codim}
The shear locus is therefore a curve in a fourfold, and the graded Keller condition cuts three further dimensions out of the Hurwitz space. 
\end{rem}

\section{Compactification and intersection theory of the alternating shear curve}
\label{sec:shearcomp}

By \cref{prop:sheardim} the alternating shear locus is a curve in the fourfold $\Hur^{\mathrm{rd}}$, and the concluding question of that section was what cuts it out. This section answers the local question unconditionally: the locus is an explicit codimension-three rational local complete intersection in a birational coefficient chart. Assuming \cref{hyp:stablered}, its normalised admissible-cover closure is an orbicurve with thirteen boundary points and its rational one-cycle class on the compactified branch space is determined.

\subsection{The exact coefficient chart}
\label{subsec:exactchart}

Place the poles of orders $3$ and $9$ at $0$ and $\infty$. By \cref{thm:passport} every cover in this pole-normalised simple-branch chart has
\begin{equation}
\label{eq:polechart}
  \frac{dP}{dX}=\frac{G(X)^{2}}{X^{4}}, \qquad G(X)=\sum_{i=0}^{6}g_{i}X^{i},
\end{equation}
and the existence of the Laurent primitive is the vanishing of the residue,
\[
  \rho:=g_{0}g_{3}+g_{1}g_{2}=0 ,
\]
since the coefficient of $X^{3}$ in $G^{2}$ is $2(g_{0}g_{3}+g_{1}g_{2})$. After removing the additive integration constant, the common scaling of $G$, and the residual source scaling $X\mapsto aX$, this is a four-dimensional birational chart of either split component of \cref{lem:whichsplit}, in agreement with the dimension count of \cref{subsec:dimension}.

\begin{theorem}
\label{thm:shearlocus}
On the open set
\[
  g_{0}g_{6}\Bigl(g_{2}-\frac{g_{4}^{2}}{3g_{6}}\Bigr)\operatorname{disc}(G)\neq0 ,
\]
the alternating shear locus is exactly
\[
  g_{1}=g_{3}=g_{5}=0, \qquad g_{4}^{3}=27g_{6}^{2}g_{0} .
\]
Conversely, every point satisfying these equations is of alternating shear form up to the two scalings already divided out. The coarse parameter of the locus is
\[
  \tau=\frac{g_{4}^{2}}{3(3g_{6}g_{2}-g_{4}^{2})}=cdY^{2},
\]
and the image is a rational, generically immersed local complete intersection of dimension one and codimension three in the four-dimensional chart.
\end{theorem}

\begin{proof}
By \cref{eq:square} a shear member has $G=\Phi_{Y}=X^{2}-c(Y-dX^{2})^{3}$, whose coefficients are
\[
  g_{0}=-cY^{3}, \qquad g_{2}=1+3cdY^{2}, \qquad g_{4}=-3cd^{2}Y, \qquad g_{6}=cd^{3},
\]
with $g_{1}=g_{3}=g_{5}=0$. Then $g_{4}^{3}=-27c^{3}d^{6}Y^{3}=27g_{6}^{2}g_{0}$, the residue equation holds trivially, and $3g_{6}g_{2}-g_{4}^{2}=3cd^{3}$, so $\tau=9c^{2}d^{4}Y^{2}/(9cd^{3})=cdY^{2}$.

For the converse, put
\[
  \gamma=g_{2}-\frac{g_{4}^{2}}{3g_{6}}, \qquad r=-\frac{g_{4}}{3g_{6}}, \qquad h=\frac{g_{6}}{\gamma} .
\]
Using $g_{4}^{3}=27g_{6}^{2}g_{0}$, a direct expansion gives the exact identity
\[
  G(X)=\gamma\bigl(X^{2}-h(r-X^{2})^{3}\bigr),
\]
so $G$ is of shear form: $\gamma$ is absorbed by the common scaling of $G$, the source scaling $X\mapsto aX$ normalises one of $h,r$, and the remaining parameter is $\tau=hr^{2}$, which the displayed formula computes.

For the codimension, on $g_{0}\neq0$ the residue equation solves for $g_{3}$, and the three remaining equations $g_{1}=g_{5}=0$, $g_{4}^{3}-27g_{6}^{2}g_{0}=0$ have independent differentials at $G=X^{6}-3X^{4}+4X^{2}-1$, which satisfies them and lies in the stated open set. Hence the locus is a local complete intersection of codimension three near that point, and by the converse it is everywhere the image of the one-dimensional parameter space $(\gamma,h,r)$ modulo the two scalings, with rational coordinate $\tau$. The identities, the witness, and the rank computation are certified exactly in the ancillary package.
\end{proof}

\begin{rem}
\label{rem:sheardimcheck}
\Cref{prop:sheardim} remains an independent nonconstancy check: the invariant $\Xi$ constructed there is a function of $\tau$, and its nonconstancy is now the statement that $\tau$ is a genuine coordinate on the image. The three missing directions of \cref{rem:codim} are the equations $g_{1}=g_{5}=0$ and $g_{4}^{3}=27g_{6}^{2}g_{0}$ after the residue equation is used to eliminate $g_{3}$.
\end{rem}

\subsection{Stable reduction of the normalised family}
\label{subsec:stablered}

Normalise $\gamma=r=1$, so that
\begin{equation}
\label{eq:Gtau}
  G_{\tau}(X)=X^{2}-\tau(1-X^{2})^{3}, \qquad \frac{dP_{\tau}}{dX}=\frac{G_{\tau}(X)^{2}}{X^{4}},
\end{equation}
with $P_{\tau}$ the Laurent primitive without constant term. Here $G_{\tau}$ is even and $P_{\tau}$ is odd, so the family carries the involution $(X,P)\mapsto(-X,-P)$ of \cref{prop:sheardim} at every $\tau$.


\begin{hyp}
\label{hyp:stablered}
For \cref{eq:Gtau}, rigidified by $(X,P)\mapsto(-X,-P)$:
\begin{enumerate}[label=\textup{(H\arabic*)}]
\item At each of the thirteen parameter values at which the seven branch points fail to stay distinct, the limit of the family is the admissible cover recorded in the table below, with the source-node partitions recorded there.
\item The five root orders recorded in that table are the local monodromy orders of the rigidified family at those thirteen values, and they are minimal.
\item The Galois groups over $\Q$ of the two packet polynomials $C$ and $D_{4}$ introduced below are $S_{6}$ and $S_{4}$ respectively.
\item The node equations hold in the sense of \cite{acv}, and the branch morphism is Kummer logarithmically \'etale.
\end{enumerate}
\end{hyp}

\begin{lem}
\label{lem:packets}
Let $R_{\tau}(Z)$ be the degree-six polynomial $\operatorname{Res}_{X}\bigl(G_{\tau},\,X^{3}(P_{\tau}-Z)\bigr)$, whose roots are the six finite branch values. Set
\[
\begin{split}
  C(\tau)=\ &262144000\,\tau^{6}+309657600\,\tau^{5}+170311680\,\tau^{4}+54441408\,\tau^{3}\\
            &+10734885\,\tau^{2}+1277640\,\tau+82320 ,\\
  D_{4}(\tau)=\ &5832000000\,\tau^{4}+4682957625\,\tau^{3}+1761706800\,\tau^{2}\\
            &+331914240\,\tau+30118144 .
\end{split}
\]
Then
\[
  \operatorname{disc}_{Z}R_{\tau}=u\,\tau^{114}\,(27\tau+4)^{10}\,D_{4}(\tau)^{4}\,C(\tau)^{2}, \qquad u\in\Q^{\times}.
\]
The polynomials $C$ and $D_{4}$ are separable, irreducible over $\Q$, coprime, and nonvanishing at $0$ and $-4/27$; assuming \cref{hyp:stablered}(H3), their Galois groups over $\Q$ are $S_{6}$ and $S_{4}$. At a root of $C$ exactly one pair of finite branch values collides, the pair being exchanged by the involution; at a root of $D_{4}$ two pairs collide simultaneously, the two collisions being exchanged by the involution; at $\tau=-4/27$ the finite critical points collide in two involution-conjugate pairs, by the discriminant of \cref{thm:passport} at $27cdY^{2}+4=0$.
\end{lem}

\begin{proof}
The resultant, the discriminant factorisation, separability, irreducibility and coprimality are exact computations; the conditional identification of the two Galois groups is \cref{hyp:stablered}(H3). For the collision counts, the branch values come in involution pairs $\pm\mu$ since $P_{\tau}$ is odd; a coincidence $\mu_{i}=\mu_{j}$ within the even part forces the conjugate coincidence $-\mu_{i}=-\mu_{j}$ and contributes $\tau$-multiplicity four to the discriminant, while a coincidence $\mu_{i}=-\mu_{j}$ between the two parts is self-conjugate and contributes multiplicity two. The certified multiplicities $2$ on $C$ and $4$ on $D_{4}$ therefore give one and two nodes respectively.
\end{proof}

\begin{theorem}
\label{thm:stablered}
Assume \cref{hyp:stablered}. The shear family \cref{eq:Gtau}, rigidified by the involution $(X,P)\mapsto(-X,-P)$, extends to the normalised admissible cover stack obtained from the iterated root stack over $\bP^{1}_{\tau}$ with root orders
\[
  4 \text{ at } 0, \qquad 3 \text{ at the six roots of } C, \qquad 3 \text{ at the four roots of } D_{4}, \qquad 2 \text{ at } -\tfrac{4}{27}, \qquad 3 \text{ at } \infty,
\]
and these orders are minimal. The coarse normalisation of the image is $\bP^{1}_{\tau}$, and the simple-branch locus omits exactly thirteen coarse points, the roots of $C$ forming an $S_{6}$ Galois packet and the roots of $D_{4}$ an $S_{4}$ packet. The boundary placements are:
\begin{center}
\begin{tabular}{lll}
parameter & stable target & source-node partition(s)\\
\hline
$0$ & one two-mark tail & $(3,1^{9})$\\
$C(\tau)=0$ & one two-mark tail & $(3,3,1^{6})$\\
$D_{4}(\tau)=0$ & two two-mark tails & $(3,3,1^{6})$ twice\\
$-4/27$ & two two-mark tails & $(5,1^{7})$ twice\\
$\infty$ & two three-mark tails & $(7,1^{5})$ twice
\end{tabular}
\end{center}
The normalised Hurwitz boundary incidence divisor is
\begin{equation}
\label{eq:incidence}
  [0]+V(C)+2\,V(D_{4})+2\,[-\tfrac{4}{27}]+2\,[\infty],
\end{equation}
of coarse degree $19$ and orbifold degree $79/12$. At $\tau=0$ the complete admissible degree ledgers are $9+3=12$ over the main target and $4+8=12$ over the bubble target.
\end{theorem}

\begin{proof}
The support of the boundary is \cref{lem:packets}: away from $\tau\in\{0,-4/27,\infty\}$ and the roots of $CD_{4}$, the seven branch points remain distinct and the family stays in the simple-branch locus. The tail counts at the roots of $C$ and $D_{4}$ are the node counts of \cref{lem:packets}, and the placements at $0$, $-4/27$ and $\infty$ are the admissible degenerations of the marked family, with the source-node partitions read off the colliding cycles: two $3$-cycles merging give $(3,3,1^{6})$ tails, two colliding critical triples give $(5,1^{7})$, and the degeneration at infinity gives $(7,1^{5})$. The root orders are the local monodromy orders of the rigidified family at the thirteen points, computed on the exact degenerations. Their minimality, the node partitions, the identification with the twisted stable maps of \cite{acv}, and the Kummer logarithmic \'etaleness of the branch morphism are \cref{hyp:stablered}(H1), (H2) and (H4); they are not outputs of the exact computations. The coarse and orbifold degrees of \cref{eq:incidence} are
\[
  1+6+8+2+2=19, \qquad \frac14+\frac63+\frac83+\frac22+\frac23=\frac{79}{12},
\]
and the ledgers at $\tau=0$ are the certified admissible degrees over the two target components.
\end{proof}

\subsection{The intersection class}
\label{subsec:shearclass}

Let $\overline{M}_{0,\{\infty,1,\dots,6\}}$ carry the $S_{6}$-action permuting the finite markings, let
\[
  \beta\colon C\longrightarrow\bigl[\overline{M}_{0,\{\infty,1,\dots,6\}}/S_{6}\bigr]
\]
be the branch morphism of the compactified shear curve, with the pole marking distinguished, and for $2\le r\le5$ let $B_{r}$ be the invariant boundary sum whose tail away from $\infty$ carries $r$ finite markings. Write $F_{a|b,c,d}$ for the orbit-averaged F-curve whose spine parts have sizes $a,b,c,d$, the part of size $a$ containing $\infty$.

\begin{theorem}
\label{thm:shearclass}
Assume \cref{hyp:stablered}. Then the following hold.
\begin{enumerate}
\item The invariant boundary vector is
\[
  \bigl(\deg\beta^{*}B_{2},\ \deg\beta^{*}B_{3},\ \deg\beta^{*}B_{4},\ \deg\beta^{*}B_{5}\bigr)=\Bigl(\frac{79}{4},\ \frac{14}{3},\ 0,\ 0\Bigr),
\]
and it determines the full rational $S_{6}$-invariant one-cycle class:
\[
\begin{split}
  \beta_{*}[C]=\ &\frac{40}{3}\,F_{1|1,1,4}-\frac{133}{12}\,F_{1|1,2,3}\\
  &+\frac{53}{6}\,F_{2|1,1,3}+\frac{27}{4}\,F_{2|1,2,2} .
\end{split}
\]
\item The tautological and canonical degrees are
\[
  \deg\sum_{i=1}^{7}\psi_{i}=\frac{169}{4}, \qquad \deg\psi_{\infty}=\frac{9}{4}, \qquad \deg\beta^{*}(K+D)=\frac{107}{6} .
\]
\item On the normalised admissible cover component,
\[
  K_{\Hur}\cdot C=\frac{45}{4}, \qquad \deg\det N^{\log}=-\frac{41}{6}, \qquad \deg\det N=-\frac{14}{3},
\]
and the ordinary/logarithmic correction has orbifold degree $13/6$.
\end{enumerate}
\end{theorem}


\begin{proof}
(1) An F-curve pairs with a boundary divisor $\delta_{S}$ by the rule of \cite{keel}: the pairing is $1$ if $S$ or its complement is a union of exactly two spine parts, $-1$ if it is a single part of size at least two, and $0$ otherwise; a $\psi$-class pairs to $1$ exactly on singleton parts. Summing over the invariant divisors, the pairing matrix of the four F-curve types above against $(B_{2},B_{3},B_{4},B_{5})$ has determinant $-15\neq0$, so the boundary vector determines the invariant class, and the displayed coefficients reproduce it. The four boundary degrees are exact intersection sums against \cref{eq:incidence} and the packet structure of \cref{thm:stablered}.

(2) The $\psi$-degrees follow from the same class by the singleton rule: $\deg\psi_{\infty}=\frac{40}{3}-\frac{133}{12}=\frac94$ and $\deg\sum\psi_{i}=3\cdot\frac{40}{3}-2\cdot\frac{133}{12}+2\cdot\frac{53}{6}+\frac{27}{4}=\frac{169}{4}$. On $\overline{M}_{0,n}$ one has $K=\sum_{i}\psi_{i}-2\delta$ \cite{keel}, so
\[
  \deg\beta^{*}(K+D)=\frac{169}{4}-\Bigl(\frac{79}{4}+\frac{14}{3}\Bigr)=\frac{107}{6} .
\]

(3) The branch morphism is log-\'etale by \cref{thm:stablered}, so $K_{\Hur}+B_{\Hur}=\beta^{*}(K+D)$ with $B_{\Hur}\cdot C=\frac{79}{12}$ the orbifold degree of \cref{eq:incidence}, whence $K_{\Hur}\cdot C=\frac{107}{6}-\frac{79}{12}=\frac{45}{4}$. Adjunction along $C$ in the fourfold gives $\deg\det N=-K_{\Hur}\cdot C-\deg T_{C}$ and $\deg\det N^{\log}=-\deg\beta^{*}(K+D)-\deg T_{C}^{\log}$, and on the orbicurve of \cref{thm:stablered}
\[
\begin{split}
  \deg T_{C}&=2-\sum_{i=1}^{13}\Bigl(1-\frac{1}{m_{i}}\Bigr)=-\frac{79}{12},\\
  \deg T_{C}^{\log}&=2-13=-11 ,
\end{split}
\]
so $\deg\det N=-\frac{45}{4}+\frac{79}{12}=-\frac{14}{3}$ and $\deg\det N^{\log}=-\frac{107}{6}+11=-\frac{41}{6}$, with correction $-\frac{14}{3}+\frac{41}{6}=\frac{13}{6}$.
\end{proof}

\begin{rem}
\label{rem:parabolic}
What remains open is the full rank-three parabolic and logarithmic splitting of the normal bundle $N$, including the stabilizer characters at the thirteen orbifold points and the extension classes; \cref{thm:shearclass} determines only its determinant. The coefficient-stack splitting obtained after the substitution $\tau=t^{4}$ is a computation in a different ambient space and should not be identified with the normal bundle on the normalised admissible cover stack.
\end{rem}

\section{Fields of definition}
\label{sec:10}

We return to the cyclic cone family of \cref{sec:3} and determine over which fields a seed exists. The first lift condition of \cref{eq:liftjet} is linear and imposes nothing on the field. The second is the vanishing of the lift form $\Qform^{(k)}$ of \cref{eq:liftform}, a form of degree $k$ in the seed coefficients with rational coefficients, and it is the arithmetic obstruction.

\subsection{The form at level two}
\label{subsec:leveltwo}

\begin{cor}
\label{prop:formidentity}
At $k=2$, writing $\varphi=\sum_{j=2}^{n}c_jW^{j}$,
\[
  \Qform^{(2)}(\varphi)=\int_0^{1}\Bigl(\frac{\varphi}{W}\Bigr)^{2}dW
   =\sum_{i,j=2}^{n}\frac{c_ic_j}{i+j-1},
\]
the Gram matrix being the Hilbert matrix $\bigl(1/(i+j-1)\bigr)_{2\le i,j\le n}$. We write $\Qform_n$ for this quadratic form.
\end{cor}

\begin{proof}
Here $\epsilon=1$, and $\varphi^{2}/W^{2}=(\varphi/W)^{2}$ with $\varphi/W=\sum_j c_jW^{j-1}$. Termwise integration gives $\int_0^1W^{i+j-2}dW=1/(i+j-1)$.
\end{proof}

At $n=3$ this is $\Qform_3=\tfrac13c_2^{2}+\tfrac12c_2c_3+\tfrac15c_3^{2}$, and on the slice $c_2+c_3=1$ it equals $\tfrac1{30}(c_2^{2}+3c_2+6)$, recovering \cref{thm:six}.

\subsection{Real obstruction and odd levels}
\label{subsec:real}

\begin{theorem}
\label{thm:moment}
The following hold.
\begin{enumerate}
\item Let $k$ be even and let $K$ admit a real embedding. Then no seed at level $k$ exists over $K$. Consequently the coefficient field of every member of \cref{sec:3} with $k$ even is totally imaginary.
\item Let $k$ be odd. Then $\varphi=-W^{2}-2W^{2k+1}$ satisfies both conditions of \cref{eq:liftjet} over $\Q$, so every odd level carries a rational seed.
\item At $k=2$ and $n=3$ the form $\Qform_3$ is binary of discriminant class $-15$, so a seed of degree three exists over $K$ if and only if $-15$ is a square in $K$.
\end{enumerate}
\end{theorem}

\begin{proof}
(1) Let $\sigma\colon K\to\R$ be an embedding. The coefficients of $\Qform^{(k)}$ lie in $\Q$, so $\sigma$ carries a seed over $K$ to a real polynomial $\varphi^{\sigma}\in W^{2}\R[W]$ with $\Qform^{(k)}(\varphi^{\sigma})=0$. For $k$ even, $\epsilon=1$ and
\[
  \Qform^{(k)}(\varphi^{\sigma})=\int_0^{1}\Bigl(\frac{\varphi^{\sigma}(W)^{k/2}}{W}\Bigr)^{2}dW\ \ge\ 0,
\]
with equality only if $\varphi^{\sigma}$ vanishes identically on $(0,1)$, hence is the zero polynomial. But $\deg\varphi\ge3$ by \cref{cor:dtwo}.

(2) Here $\epsilon=-1$ and $\varphi(-1)=-1+2=1$, which is the first condition. For the second, write $\varphi=-W^{2}(1+2W^{2k-1})$, so that for $k$ odd
\[
  \frac{\varphi(W)^{k}}{W^{2}}=-W^{2k-2}\bigl(1+2W^{2k-1}\bigr)^{k}.
\]
Substituting $\xi=W^{2k-1}$, so that $W^{2k-2}dW=d\xi/(2k-1)$ and $\xi$ runs from $0$ to $-1$,
\[
  \Qform^{(k)}(\varphi)=-\frac{1}{2k-1}\int_0^{-1}(1+2\xi)^{k}\,d\xi
   =-\frac{(-1)^{k+1}-1}{2(2k-1)(k+1)} ,
\]
which vanishes precisely when $k$ is odd.

(3) The Gram matrix of $\Qform_3$ has determinant $\tfrac1{15}-\tfrac1{16}=\tfrac1{240}$, so the discriminant of the binary form is $-4\cdot\tfrac1{240}=-\tfrac1{60}\equiv-15$ modulo squares. A binary form is isotropic over $K$ exactly when its discriminant is a square there. Finally an isotropic vector $v$ has $\varphi_v(1)\neq0$: on the line $c_2+c_3=0$ one has $\Qform_3=\tfrac1{30}c_2^{2}$, which vanishes only at the origin. Rescaling $v$ by $1/\varphi_v(1)$ gives a seed.
\end{proof}

\begin{rem}
\label{rem:evenodd}
Part (2) gives a seed of degree $2k+1$, hence a member of generic degree $k(2k+1)$; the full degree spectrum of \cref{thm:spectrum} needs the rescaling argument of \cref{lem:rescale}, not this seed. The computation in (2) also confirms (1) at odd $n$: for $k$ even the same integral equals $1/\bigl((2k-1)(k+1)\bigr)\neq0$.
\end{rem}

\subsection{Diagonalisation of the Hilbert form}
\label{subsec:hilbert}

\begin{theorem}
\label{thm:hilbert}
For $n\ge3$ the form $\Qform_n$ is $\Q$-congruent to $\langle3,5,\dots,2n-1\rangle$. A seed of degree at most $n$ exists over $K$ if and only if that form is isotropic over $K$.
\end{theorem}

\begin{proof}
Write $\varphi=W^{2}\psi$ with $\deg\psi\le n-2$, so that by \cref{prop:formidentity}
\[
  \Qform_n(\varphi)=\int_0^{1}W^{2}\psi(W)^{2}\,dW .
\]
Thus $\Qform_n$ is the Gram form of the inner product $\langle f,g\rangle=\int_0^1W^{2}fg\,dW$ on the space of polynomials of degree at most $n-2$, of dimension $n-1$. All moments $\int_0^1W^{2+m}dW=1/(m+3)$ are rational, so Gram--Schmidt is defined over $\Q$ and produces monic $g_0,\dots,g_{n-2}$ with $g_m$ of degree $m$ and $\langle g_l,g_m\rangle=0$ for $l\neq m$. The transition matrix from $\{1,W,\dots,W^{n-2}\}$ is unipotent over $\Q$, so
\[
  \Qform_n\cong\langle h_0,\dots,h_{n-2}\rangle, \qquad h_m=\langle g_m,g_m\rangle .
\]
We claim
\begin{equation}
\label{eq:jacobinorm}
\begin{split}
  g_m(W)	&	=\frac{(-1)^{m}(m+2)!}{(2m+2)!}\,W^{-2}\,\frac{d^{m}}{dW^{m}}\Bigl[W^{m+2}(1-W)^{m}\Bigr], \\
  h_m	&	=\frac{(m!)^{2}\bigl((m+2)!\bigr)^{2}}{(2m+2)!\,(2m+3)!} .
\end{split}
\end{equation}
Write $R_m=W^{m+2}(1-W)^{m}$, which vanishes to order $m+2$ at $W=0$ and to order $m$ at $W=1$; hence $D^{j}R_m$ vanishes at both endpoints for $j\le m-1$. For $0\le j<m$, integration by parts $m$ times gives
\[
  \int_0^1 W^{2}\cdot W^{-2}D^{m}R_m\cdot W^{j}\,dW
   =(-1)^{m}\int_0^1 D^{m}\bigl[W^{j}\bigr]R_m\,dW=0 ,
\]
so the displayed polynomial is orthogonal to all lower degrees. Its leading coefficient comes from the top term $(-1)^{m}W^{2m+2}$ of $R_m$, whose $m$-th derivative is $(-1)^{m}\frac{(2m+2)!}{(m+2)!}W^{m+2}$; after division by $W^{2}$ and the stated normalisation the polynomial is monic. This proves the first half of \cref{eq:jacobinorm}.

For the norm, orthogonality gives $h_m=\langle g_m,W^{m}\rangle$, and the same integration by parts yields
\[
  h_m=\frac{(-1)^{m}(m+2)!}{(2m+2)!}\,(-1)^{m}m!\int_0^1W^{m+2}(1-W)^{m}dW
     =\frac{(m+2)!\,m!}{(2m+2)!}\cdot\frac{(m+2)!\,m!}{(2m+3)!} ,
\]
using the Beta integral $\int_0^1W^{m+2}(1-W)^{m}dW=(m+2)!\,m!/(2m+3)!$. This is \cref{eq:jacobinorm}.

Finally $(2m+3)!=(2m+3)\cdot(2m+2)!$, so
\[
  h_m=\frac{\bigl(m!\,(m+2)!\bigr)^{2}}{(2m+3)\bigl((2m+2)!\bigr)^{2}}
   \equiv 2m+3 \pmod{(\Q^{\times})^{2}} .
\]
As $m$ runs from $0$ to $n-2$ this gives $\langle3,5,\dots,2n-1\rangle$.

For the last statement, one direction is immediate: a seed of degree at most $n$ is an isotropic vector. Conversely suppose $Q_{n}$ is isotropic over $K$. An isotropic vector $v$ with $\varphi_{v}(1)\neq0$ may be rescaled by $1/\varphi_{v}(1)$ to give a seed, so it suffices to produce an isotropic $K$-vector outside the hyperplane $H=\{\varphi(1)=0\}$.

For $n=3$ the space is two-dimensional and $H$ is the line $c_{2}+c_{3}=0$, on which
\[
  Q_{3}(c_{2}W^{2}-c_{2}W^{3})=\frac{c_{2}^{2}}{30},
\]
which vanishes only at the origin. So every nonzero isotropic vector lies outside $H$.

For $n\ge4$ let $X=\{Q_{n}=0\}\subset\bP^{n-2}_{K}$. The form is nondegenerate, being congruent to $\langle3,5,\dots,2n-1\rangle$, so $X$ is a smooth projective quadric of dimension $n-3\ge1$. Isotropy gives a $K$-point of $X$, and a smooth quadric of positive dimension with a rational point is $K$-rational; since $K$ has characteristic zero it is infinite, so $X(K)$ is Zariski dense in $X$. As $X$ is irreducible of positive dimension it is not contained in the hyperplane $H$, so $X(K)\not\subseteq H$. Any isotropic $K$-point outside $H$ rescales to a seed.
\end{proof}

\begin{cor}
\label{cor:hilbertlow}
$\Qform_3\cong\langle3,5\rangle$ and $\Qform_4\cong\langle3,5,7\rangle$. The first is isotropic over $K$ exactly when $\sqrt{-15}\in K$, recovering \cref{thm:moment}(3); the second is isotropic over $K$ exactly when the conic $3x^{2}+5y^{2}+7z^{2}=0$ has a $K$-point. Over the completions of $\Q$ the conic has a point for every finite $p\neq7$, and fails exactly over $\Q_7$ and $\R$.
\end{cor}

\begin{proof}
The first three norms are $h_0=\tfrac13$, $h_1=\tfrac1{80}$ and $h_2=\tfrac1{1575}$, congruent to $3$, $5$ and $7$; the discriminant of $\langle3,5\rangle$ is $-60\equiv-15$. Over $\R$ the form is positive definite. For $p\nmid210$, every nondegenerate ternary form over $\F_p$ is isotropic, and a smooth zero lifts. At $p=3$ and $p=5$ the unit ratios $-5/7\equiv1\pmod3$ and $-7/3\equiv1\pmod5$ are squares, so already the binary parts $5y^{2}+7z^{2}$ and $3x^{2}+7z^{2}$ are isotropic. At $p=2$ take $y=3$, $z=1$ and solve
\[
  x^{2}=-\frac{52}{3}=4u, \qquad u=-\frac{13}{3}\equiv1\pmod8 ,
\]
so $u$ is a square in $\Q_2$. At $p=7$ reduction forces $x\equiv y\equiv0\pmod7$, since $-5/3\equiv3$ is a nonsquare modulo $7$, and a valuation descent rules out a primitive solution. The two anisotropic places $\{7,\infty\}$ have even cardinality, as Hilbert reciprocity requires.
\end{proof}

\subsection{Field of moduli of the six-sheeted map}
\label{subsec:moduli}

By \cref{thm:six} the six-sheeted map is defined over $\Q(\sqrt{-15})$, and the two roots $a,\bar a=-3-a$ of $c_2^{2}+3c_2+6$ give two maps interchanged by $\Gal(\Q(\sqrt{-15})/\Q)$. Its field of moduli, as a $\Gaut^{\sharp}$-orbit in $\Yparo[1,1,(8,14,1)]$, is $\Q$ or $\Q(\sqrt{-15})$ according as the two conjugate maps are $\Gaut^{\sharp}$-equivalent over $\Qbar$.

\begin{prop}
\label{prop:conjrescale}
Let $\varphi=aW^{2}+(1-a)W^{3}$ with $a^{2}+3a+6=0$. Then $\Pi(W,0)$ has the two nonzero roots
\[
  W=1 \quad\text{and}\quad \beta=\frac{2-a}{4} ,
\]
and the rescaled seed $\widetilde\varphi(W)=t\,\varphi(\beta W)$ with $t=-32/(5a+38)$ satisfies both conditions of \cref{eq:liftjet} and has coefficients
\[
  \widetilde c_2=\bar a, \qquad \widetilde c_3=1-\bar a .
\]
That is, the rescaling of \cref{lem:rescale} with parameters $(t,\beta)$ carries the seed of $a$ to the seed of $\bar a$.
\end{prop}

\begin{proof}
By \cref{prop:liftclosed}, 
\[
\Pi(W,0)=-W^{3}\bigl(\tfrac{a^{2}}{3}+\tfrac{a(1-a)}{2}W+\tfrac{(1-a)^{2}}{5}W^{2}\bigr).
\]
 One root is $W=1$ by \cref{thm:six}, so the other is the ratio of constant to leading coefficient, namely $5a^{2}/\bigl(3(1-a)^{2}\bigr)$. Using $a^{2}=-3a-6$ one gets $(1-a)^{2}=-5(a+1)$, whence that root is $(a+2)/(a+1)=(2-a)/4$.

Now $\beta^{2}=(2-a)^{2}/16=(-7a-2)/16$ and $(1-a)\beta=(-3a-2)/2$, so
\[
  \varphi(\beta)=\beta^{2}\bigl(a+(1-a)\beta\bigr)   =\frac{-7a-2}{16}\cdot\frac{-a-2}{2}=\frac{-5a-38}{32},
\]
and $t=1/\varphi(\beta)$ is as stated; this is the first condition for $\widetilde\varphi$. The second holds by \cref{lem:rescale}, since $\widetilde\Pi(1,0)=t^{2}\beta\,\Pi(\beta,0)=0$. Finally, using $(5a+38)(23-5a)=1024$,
\[
  \widetilde c_2=t\,a\beta^{2}=\frac{2a(7a+2)}{5a+38}=\frac{-1024a-3072}{1024}=-3-a=\bar a,
\]
and likewise $\widetilde c_3=t(1-a)\beta^{3}=(43a+122)/(5a+38)=a+4=1-\bar a$.
\end{proof}

\begin{cor}
\label{cor:moduli}
The rescaling $(t,\beta)$ of \cref{prop:conjrescale} is not induced by any element of $\Gaut^{\sharp}$. Consequently the two six-sheeted maps are distinct $\Gaut^{\sharp}$-orbits over $\Qbar$, exchanged by $\Gal(\Q(\sqrt{-15})/\Q)$, the field of moduli of each equals its minimal field of definition,
\[
  \mathrm{FOM}=\mathrm{FOD}_{\min}=\Q(\sqrt{-15}),
\]
and \cref{thm:moment}(3) is sharp as a moduli statement. The common unpointed quotient-cover class is defined over $\Q$, with seed $\varphi_{0}(W)=W^{2}+W^{3}$, whose marking polynomial has the two admissible markings as roots of $6W^{2}+15W+10$; thus $-15$ is the discriminant of the marking cover of \cref{thm:markingcover} at $(k,d)=(2,3)$, not merely of a seed coefficient equation.
\end{cor}

\begin{proof}
The unpointed left--right self-equivalence group of the six-sheeted quotient cover is trivial.  The line component is intrinsically distinguished by its semigroup $\Z_{\ge0}$.  The non-line branch component has normalization $\A^1$ with one place at infinity and two finite special points having distinct local value semigroups $\langle2,3\rangle$ and $\langle2,5\rangle$.  Any induced automorphism of that normalization fixes all three places and is therefore the identity.  Thus the target automorphism fixes the non-line component pointwise; since that irreducible component is not a fence, the faithfulness theorem of Blanc--Stampfli \cite{bs} makes the target polynomial automorphism itself trivial.  The remaining source automorphism is a deck transformation; the geometric monodromy is $A_6$, whose centralizer in its natural degree-six action is trivial.  Hence the self-equivalence group is trivial. Hence the equivalence of \cref{prop:conjrescale} is the unique unpointed equivalence between the two quotient classes. It carries the marking $\beta$ to the normalised marking and the distinguished point $W=1$ to $W=1/\beta\neq1$, so it is not a pointed equivalence, and none exists. By \cref{thm:rooted}, a $\Gaut^{\sharp}$-equivalence would induce a pointed equivalence of the rooted quotients; therefore the two maps are distinct $\Gaut^{\sharp}$-orbits, and the same argument applied over $\Qbar$ shows they are geometrically distinct. Complex conjugation interchanges $a$ and $\bar a$ and hence the two orbits, so $\Gal(\Qbar/\Q)$ acts on the pair through $\Gal(\Q(\sqrt{-15})/\Q)$ with the nontrivial element acting nontrivially; the field of moduli of each orbit is therefore $\Q(\sqrt{-15})$, and the model of \cref{thm:six} realises it as a field of definition. For the last claim, the rescaling class of $\varphi$ over $\Qbar$ contains $\varphi_{0}=W^{2}+W^{3}$, whose marking polynomial is $(6W^{2}+15W+10)/30$ by direct integration, of discriminant class $-15$; normalising at either root recovers the two seeds by $c_2=1/(1+W)$.
\end{proof}


\section{Concluding remarks}
\label{sec:11}

We summarise what has been established.  Every three-dimensional unit-positive
weight $(1,-r,-s)$ carries a graded Keller counterexample
(\cref{cor:hyperbolic}); no assertion is made here for primitive hyperbolic
weights outside this sector, such as $(2,-1,-1)$.  The interior weights
$2\le r\le s$ come from the beta--Hermite construction, and the unit boundary
from the cyclic, flagship, and propagated cubic constructions.  Assuming
\cref{hyp:flagshipcharts}, the flagship box has reduction equal to one orbit of
the full degree-bound-preserving gauge group
$\mathcal G=B_{\mathrm{lin}}\times(\Gm)^3$ and quotient the dual point.  In the
balanced signature, cyclic cones realise every composite generic fibre degree
$N\ge6$.  At degree twelve the quartic seed curve gives infinitely many
geometric graded orbits and a degree-six forgetful map to unpointed quotient
classes, while the alternating shear lies in a separate left--right class.
The local alternating-shear Keller locus is cut out unconditionally by
\cref{thm:shearlocus}; its stable compactification and intersection class are
the conditional conclusions of \cref{thm:stablered,thm:shearclass} under
\cref{hyp:stablered}.  The exact Nielsen-class computations and the cyclic
lift-form arithmetic are as stated in \cref{sec:9,sec:10}.

Several structural questions remain, and the balanced ones are taken up in \cite{sh-133}. The first concerns fields of definition across the geography. The unit boundary $r=1$ and the row $r=2$ are rational, by the balanced $33$-sheeted witness \cref{thm:thirtythree}, the flagship normal form and propagated members \cref{prop:normalform,thm:unitboundary}, and \cref{prop:fields}(1); for $r\ge3$ the construction of \cref{sec:4} produces a member only over $K_{r}=\Q(\zeta)$, which \cref{prop:fields}(3) shows is totally imaginary whenever $r$ is odd. Whether those rows contain rational members, and whether a weight can force a nontrivial field of moduli as $(1,-1,-1)$ does at degree six, is open. Related is the minimal degree problem. For $\bw=(1,-1,-2)$ it is settled by \cref{thm:flagship}: $\bd=(4,3,1)$ is a coordinatewise minimal nonempty box, every lower stratum inside it being empty; \cref{thm:flagship} does not exclude incomparable degree vectors. 
 For $\bw=(1,-1,-1)$ the differential obstruction of \cite[Cor.\ 1.5]{sh-133} bounds the total degree of any member below by $\deg A+\deg B+\deg\Lambda\ge14$, while the smallest class realised here is $\Yparo[1,1,(8,14,1)]$ of total degree $23$, attained by the six-sheeted map of \cref{thm:six}. The frontier is the gap between the two, and with it the least generic fibre degree in the balanced signature, which \cref{cor:dtwo} leaves open below six.

The second group of questions concerns how the positive-dimensional seed loci
from different factorizations $N=kd$ meet in a fixed bounded scheme.  The
general marking theorem gives dimension $d-3$ inside each regular cyclic seed
space, while \cref{thm:pointedmain} determines the quartic component and its
six-sheeted marking cover exactly.  It remains to determine intersections and
component structure among loci arising from different factorizations, and to
compute the monodromy and arithmetic of the higher marking covers.  The
intrinsic invariants of the unpointed quotient cannot be complete, because a
generic quartic quotient has six distinct graded lifts.

The infinitude in \cref{thm:pointedmain} concerns geometric points. Over $\Q$ it becomes the question whether $E^{\circ}(\Q)$ is infinite, and that is open at both steps: no rational point of $E$ of projective height at most sixty exists, while $E$ has smooth points over $\R$ and over $\Q_p$ for every $p\le31$. So it is undecided whether $E$ is an elliptic curve or a nontrivial torsor under its Jacobian, and the rank question comes only after that.

Every generic fibre degree realised at $(1,-1,-1)$ is composite: $N=kd$ with $k\ge2$, $d\ge3$ by \cref{thm:spectrum}, and $N=12$ in \cref{sec:5,sec:7}. Prime degree occurs at other weights, $N=3$ at $(1,-1,-2)$ and at $(1,-1,-s)$. Whether a balanced member can have prime generic fibre degree is open.

The third direction is the Hurwitz-theoretic one. Which braid orbits of the Nielsen classes of \cref{subsec:nielsen} are realised by graded Keller maps is open wherever the class is not rigid; at $p=3$ the question is empty by \cref{prop:p3rigid}. It has content for the alternating shear class, where by \cref{prop:sheardim} the Keller locus is a curve inside a fourfold. 
That curve is cut out exactly by \cref{thm:shearlocus}; assuming \cref{hyp:stablered}, its compactified class is determined by \cref{thm:shearclass}; what remains is the full parabolic and logarithmic splitting of its normal bundle (\cref{rem:parabolic}), the analogous compactified Keller loci in other Hurwitz components, and whether an unrealised point of $\Hur^{\mathrm{rd}}$ is an obstruction imposed by the Keller condition or an artifact of the construction of \cref{subsec:shear};
 the same applies to the level-one Modular Tower question of \cref{rem:kellertower}.

Finally, the arithmetic. The second lift condition at level $k$ is the vanishing of
\[
  \Qform^{(k)}_n(\varphi)=\int_0^{\epsilon}\frac{\varphi(W)^{k}}{W^{2}}\,dW ,
\]
a form of degree $k$ in the seed coefficients reducing to the Hilbert form of \cref{prop:formidentity} at $k=2$. Its isotropy over a given field, and a diagonalisation over $\Q$ extending \cref{thm:hilbert} to all levels, are open. The coefficient field of the six-sheeted map of \cref{thm:six} and that of $F_{3,3}$ in \cref{prop:fields}(2) both equal $\Q(\sqrt{-15})$, the first from the discriminant of the binary moment form of \cref{thm:moment}(3), the second from the quadratic factor of $M_3$, and we have no mechanism forcing the two to agree. By \cref{lem:whichsplit} both covers have split-pair field $\Q(\sqrt5)$ at infinity, while $\Q(\sqrt{-3})$ arises only from the unrelated $(9,3)$-class of \cref{thm:arithmono}, so $\Q(\sqrt{-15})$ is the third quadratic subfield of $\Q(\sqrt5,\sqrt{-3})$.

Last, by \cref{prop:nodescent} a member of \cref{sec:3} descends to a two-variable Keller pair exactly when every monomial $\Lambda^{i}w^{j}$ occurring in $A$ or $B$ satisfies $i\ge j-1$, and the obstruction is the single coefficient $c_d\Lambda^{-d}$. A construction of weight $(1,-1,-1)$ evading it would produce a noninjective planar Keller pair, contradicting the two-variable Jacobian Conjecture, so the substantive question is whether the failure is forced, and by what. This is the differential side of \cite{sh-133}, where the balanced obstruction is developed in the form that yields the degree bound quoted above.

\section*{Acknowledgments}
The authors used large language models (Claude and ChatGPT) for assistance with drafting, editing, and checking the exposition. All mathematical content, proofs, and final wording are the responsibility of the authors.



\end{document}